\documentclass[12pt,oneside,reqno,a4paper]{article}
\usepackage[utf8]{inputenc}
\usepackage{amsmath,amsthm,amssymb}     
\usepackage{bbm}
\usepackage{tikz}
\usetikzlibrary{arrows, math}
\usepackage{tikz-cd}
\usepackage{xcolor}
\usepackage{adjustbox}
\usepackage{caption}
\usepackage{hyperref}
\usepackage{mathtools}

\usepackage[T1]{fontenc}

\usepackage[numbers]{natbib}
\usepackage{graphicx}
\usepackage{url}
\usepackage{fancyhdr}
\usepackage{verbatim}
\usepackage{abstract}
\usepackage[english]{babel}
\usepackage{tipa}

\tikzcdset{
  cells={font=\everymath\expandafter{\the\everymath\displaystyle}},
}
\usepackage{fancyhdr}
\newtheorem{theorem}{Theorem}[section]

\newtheorem{corollary}[theorem]{Corollary}
\newtheorem{proposition}[theorem]{Proposition}
\newtheorem{lemma}[theorem]{Lemma}
\theoremstyle{definition}
\newtheorem{definition}[theorem]{Definition}
\newtheorem{remark}[theorem]{Remark}
\newtheorem{example}[theorem]{Example}
\newtheorem{conjecture}[theorem]{Conjecture}
\newtheorem{construction}[theorem]{Construction}

\DeclareFontFamily{U}{dmjhira}{}
\DeclareFontShape{U}{dmjhira}{m}{n}{
  <-> dmjhira
}{}
\DeclareFontSubstitution{U}{dmjhira}{m}{n}

\newcommand{\yo}{\textup{{\usefont{U}{dmjhira}{m}{n}\symbol{"48}}}}
\newcommand{\esh}{\textup{\raisebox{1pt}\textesh}}
\newcommand{\op}{^\textup{op}}
\newcommand{\co}{^\textup{co}}
\newcommand{\inv}{{^{-1}}}

\newcommand{\Ob}{\textup{Ob}}
\newcommand{\Set}{\textbf{\textup{Set}}}
\newcommand{\Cat}{\textbf{\textup{Cat}}}
\newcommand{\Catic}{\textbf{\textup{Cat}}_\textup{ic}}

\newcommand{\Bicat}{\textbf{\textup{Bicat}}}
\newcommand{\Bicatlic}{{\textbf{\textup{Bicat}}_\textup{lic}}}
\newcommand{\Bicatic}{{\textbf{\textup{Bicat}}_\textup{ic}}}
\newcommand{\Psh}[1]{\textbf{\textup{Psh}}(#1)}
\newcommand{\Pshic}[1]{\textbf{\textup{Psh}}_\textup{ic}(#1)}
\newcommand{\Pshat}[1]{\textbf{\textup{Psh}}^\textup{at}(#1)}
\newcommand{\Pshattwo}{\textbf{\textup{Psh}}^\textup{at}}
\newcommand{\Pshicat}[1]{\textbf{\textup{Psh}}^\textup{at}_\textup{ic}(#1)}
\newcommand{\PshP}[1]{\textbf{\textup{Psh}}_{P}(#1)}

\newcommand{\rAct}[1]{\mathbb{A}\textup{ct-}#1}
\newcommand{\clb}{\clubsuit}
\newcommand{\spd}{\spadesuit}
\newcommand{\Kar}[1]{\textup{Kar}(#1)}
\newcommand{\Kartwo}{\textup{Kar}}

\renewcommand{\tilde}[1]{\widetilde{#1}}
\renewcommand{\phi}{\varphi}
\newcommand{\is}{^\sharp}
\newcommand{\es}{^\flat}
\newcommand{\iso}{\,\cong\,}
\renewcommand{\equiv}{\,\simeq\,}
\newcommand{\id}{\textup{id}}
\newcommand{\B}{\textup{B}}
\newcommand{\C}{\mathcal{C}}
\newcommand{\D}{\mathcal{D}}
\newcommand{\E}{\mathcal{E}}

\newcommand{\J}{\mathcal{J}}
\newcommand{\K}{\mathcal{K}}
\newcommand{\U}{\mathcal{U}}

\newcommand{\bB}{\mathbb{B}}
\newcommand{\bC}{\mathbb{C}}
\newcommand{\bD}{\mathbb{D}}
\newcommand{\bI}{\mathbb{I}}
\newcommand{\bJ}{\mathbb{J}}
\newcommand{\lt}{\vartriangleleft}
\newcommand{\rt}{\vartriangleright}
\newcommand{\p}{^\prime}
\newcommand{\pp}{^{\prime\prime}}
\newcommand{\colim}{\textup{colim}}
\newcommand{\colimic}{\textup{colim}_{\textup{ic}}}

\begin{document}

\renewcommand{\thepage}{C-\Roman{page}}
\title{Bicategorical Idempotent Completion}
\author{Tessa Kammermeier}
\maketitle

\pagenumbering{arabic}
\begin{abstract}
In recent years, higher categorical variants of idempotents have become an object of interest. Notably, they appear when categorifying algebraic structures motivated by mathematical physics. Working in the setting of bicategories, we give a rigorous definition of the bicategorical idempotent completion as outlined by Gaiotto and Johnson-Freyd in \cite{GJF}. We show that this agrees with the Cauchy completion of a locally idempotent complete bicategory in the sense of Lawvere.
\end{abstract}

\tableofcontents

\section{Introduction}

Completions are an important part of a category theorist's toolkit. They allow one to guarantee the existence of certain objects simply by forcing them into existence. This is, of course, a prevailing notion in all of mathematics, e.g., the Cauchy completion of a metric space.

The idempotent completion of a category ensures that all idempotents in this category split. It is also often called the Karoubi envelope, as this construction was first written down by Max Karoubi in his introduction to K-theory \cite{K}. In the last few years, this construction has been generalised to higher categorical contexts such as the equivariant completion of a locally idempotent complete bicategory \cite{CR}, the idempotent completion of a locally idempotent complete bicategory \cite{DR}, the Euler completion of a pivotal 2-category \cite{CRS}, the orbifold completion of a 3-category \cite{CM}, and the Karoubi envelope of an $n$-category \cite{GJF}.

Similarly, William F. Lawvere introduced the Cauchy completion of an enriched category \cite{L} generalising the Cauchy completion of metric spaces. The Cauchy completion ensures the existence of all absolute colimits, or more generally, all absolute weighted colimits, i.e., colimits which are preserved by all functors. In the setting of enriched categories, explicitly describing the Cauchy completion takes great effort, see \cite{BD}, \cite{NST}, \cite{LT}. In the case of ordinary, $\Set$-enriched categories, it was quickly realised, that Cauchy completion and idempotent completion agree.

In this paper, we show that the Cauchy completion of a locally idempotent complete bicategory, which is a category weakly enriched in idempotent complete categories, agrees with its bicategorical idempotent completion as defined by Davide Gaiotto and Theo Johnson-Freyd \cite{GJF}.

\begin{theorem}(\ref{x8}, \ref{x10})
Let $\bB$ be a bicategory whose Hom-categories are all idempotent complete. The Cauchy completion of $\bB$, as a category weakly enriched in idempotent complete categories, and the bicategorical idempotent completion of $\bB$ are equivalent. It follows, that $\bB$ is complete under absolute weighted colimits if and only if it is complete under splittings of 2-idempotents.
\end{theorem}

\subsection{Outline}

In section \ref{sec3}, we recall the 1-categorical case of idempotent completion and the Cauchy completion of an ordinary, $\Set$-enriched category. We do this as we use the 1-categorical case as a scaffold for the rest of the paper and to prove the result that idempotent completing a category defines a left-adjoint pseudofunctor.

In section \ref{sec4}, we define the idempotent completion of a locally idempotent complete bicategory and prove its universal property, following the outline of \cite[Theorem 2.3.11.]{GJF}. To prove this result for bicategories, we introduce a graphical calculus for locally idempotent complete bicategories which we use frequently to simplify calculations.

In section \ref{sec5}, we show that the splitting of a 2-idempotent defines a weighted colimit in a locally idempotent complete bicategory, proving that splittings of a given 2-idempotent define a trivial groupoid. We do this using extensions of pseudofunctors, which we detail in the appendix \ref{app3}. We then define the Cauchy completion of a locally idempotent complete bicategory and prove its universal property.

In section \ref{sec6}, we prove the main result of this paper, namely that for a locally idempotent complete bicategory, its Cauchy completion is equivalent to its idempotent completion. To prove some necessary statements about this equivalence, we give an explicit construction for weighted colimits in $\Cat$ and $\Catic$ which can also be found in the appendix \ref{app1}, \ref{app2}.

\subsection{Conventions and Notation}

We refer to \cite{JY} for any details and proofs regarding bicategorical fundamentals that are not present in this paper. In it, one can find definitions and proofs of all basic bicategorial concepts up to the bicategorical Yoneda lemma and many more which are not relevant to this work. With some exceptions, we try to stick to the terminology and notation used therein. For the basic theory of weighted colimits we refer to \cite[Chapter 7]{R} and \cite[Chapter 4]{Lor}. For the theory of bicategorical coends we refer to \cite[Chapter 7.1]{Lor}. 

We use $\C$, $\D$, $\E$, $\ldots$ to denote 1-categories. We assume these to be small categories. In general, we denote arbitrary objects by lowercase Latin letters, objects carrying some structure by uppercase Latin letters and morphisms by lowercase Latin letters. We define $\Psh{\C}$ to be the large category of presheaves on a given category $\C$, i.e., the category of functors $\C\op\to\Set$, where $\Set$ is the large category of small sets. We use $\yo_\C$ to denote the Yoneda embedding $\C\to\Psh{\C}$ and we use $\Delta_c$ to denote the constant functor at an object $c$. We assume all (weighted) colimits to be small.

We call weak 2-categories \textit{bicategories} and strict 2-categories simply \textit{2-categories}. We denote them by $\bB$, $\bC$, $\bD$, $\ldots$ and assume them to be small. In general, we denote arbitrary objects by lowercase Latin letters, objects carrying some structure by uppercase Latin letters, 1-morphisms by lowercase Latin letters and 2-morphisms by lowercase Greek letters. Two objects in a bicategory are \textit{equivalent} if they are isomorphic up to 2-isomorphisms. We denote composition of 1-morphisms by $\circ$ and composition of 2-morphisms by $\cdot$. The opposite bicategory $\bB\op$ of a bicategory $\bB$ reverses the direction of 1-morphisms.

Let $P$ be a property of categories. We call a bicategory $\bB$ \textit{locally} $P$ if every Hom-category of $\bB$ satisfies $P$. Examples include locally discrete and locally idempotent complete bicategories. We denote the large 2-category of small categories by $\Cat$ and the full sub-2-category of categories satisfying $P$ by $\Cat_P$.

A \textit{pseudofunctor} is a morphism between bicategories such that all its coherence 2-morphisms are isomorphisms. A \textit{2-functor} is a pseudofunctor between 2-categories such that all its coherence 2-morphisms are given by identities. We denote pseudofunctors by uppercase Latin letters. 

A \textit{strong transformation} is a morphism between pseudofunctors such that all its coherence 2-morphisms are isomorphisms. A \textit{2-natural transformation} is a strong transformation between 2-functors such that all its coherence 2-morphisms are given by identities. We denote strong transformations by lowercase Greek Letters.

A \textit{modification} is a morphism between strong transformations. We denote these by uppercase Greek letters.

We use $f^*$ to denote precomposition with a morphism and $f_*$ to denote postcomposition with a morphism. This applies all types of morphism.

We define $\Psh{\bB}$ to be the large bicategory of presheaves on a given bicategory $\bB$, i.e., the bicategory of pseudofunctors $\bB\op\to\Cat$, strong transformations between them and modifications between those. We use $\yo_\bB$ to denote the Yoneda embedding $\bB\to\Psh{\bB}$. We use $\PshP{\bB}$ to denote the full the subbicategory of presheaves on a given bicategory $\bB$ taking values in $\Cat_P$. More generally we use $\Bicat(\bB,\bC)$ to denote the bicategory of pseudofunctors $\bB\to\bC$, strong transformations between them and modifications between those.

\subsection{Acknowledgements}

This paper originated as a master's thesis supervised by David Reutter. I want to thank him for having given me the opportunity to write this thesis with him and for his continued support since then. Furthermore, I would like to thank Sam Bauer, Julia Path and Liam Urban for many fruitful discussions and helpful conversations and all my friends for both motivating me to keep working on this paper and reminding me take breaks whenever necessary.

This work was funded by the Deutsche Forschungsgemeinschaft (DFG, German Research Foundation) -- 493608176; 531713354.

\section{The 1-Categorical Case} \label{sec3}

The theory of 1-categorical idempotent completion, Cauchy completion and their equivalence can readily be found in \cite{BD}. Yet, in the following, we will describe both constructions and the proof of their equivalence in detail, laying out the blueprint we will later use to prove their bicategorical counterparts.

\subsection{Idempotents and their Splittings}

\begin{definition}{(Idempotent Complete Category)}
Let $\C$ be a category. An \textit{idempotent} is an endomorphism $p:A\to A$ in $\C$ such that $p^2=p$.

An idempotent $p$ \textit{splits} if there exists an object $B$, and morphisms $f:A\to B$ and $g:B\to A$ in $\C$ such that $f\circ g=\id_B$ and $g\circ f=p$. The data of $B$, $f$ and $g$ defines a \textit{splitting} of $p$ and we call $B$ a retract of $A$.

A category $\C$ is \textit{idempotent complete} if every idempotent in it splits.
\end{definition}

\begin{definition}{(Free Walking Idempotent (Splitting))}
We define $\spd_1$ to be the category with two objects $X$ and $Y$ and morphisms generated by $f:X\to Y$ and $g:Y\to X$ subject to the relation $f\circ g=\id_Y$. We call this category the \textit{free walking idempotent splitting}, as functors $F:\spd_1\to\C$ directly correspond to split idempotents in $\C$. Analogously, we define $\clb_1\subseteq\spd_1$ to be the full subcategory on $X$ and call it the \textit{free walking idempotent}.
\end{definition}

\begin{remark}
For a given idempotent $F:\clb_1\to\C$, a splitting corresponds to an extension of $F$ along $\iota:\clb_1\to\spd_1$.
\end{remark}

\begin{proposition}\label{spliscol}
Let $F:\clb_1\to\C$ be an idempotent and denote $F(g\circ f)$ by $p$. The  colimit of $F$ defines a splitting of $p$. This colimit is given by the coequaliser of
\begin{center}
\begin{tikzcd}
FX \arrow[r, "\id_{FX}", shift left=1] \arrow[r, "p"', shift right=1] & FX.
\end{tikzcd}
\end{center}
Furthermore every splitting of $p$ defines such a colimit.
\end{proposition}

\begin{corollary}
A splitting of an idempotent is unique up to unique isomorphism.
\end{corollary}

We note here that the category $\clb_1$ is special. A splitting of an idempotent is defined purely equationally. Since functors preserve equations, functors preserve splittings, i.e., if an idempotent $p:A\to A$ splits via $f:A\to B$ and $g:B\to A$ in $\C$ and $F:\C\to\D$ is a functor, the idempotent $Fp:FA\to FA$ splits via $Ff:FA\to FB$ and $Fg:FB\to FA$ in $\D$. Since an idempotent splitting is defined by a colimit of a functor out of $\clb_1$, $\clb_1$ has the property that every colimit of a functor out of it is preserved by every functor.

\subsection{Idempotent Completion}

\begin{definition}{(Idempotent Completion)}
Let $\C$ be a category. We define the \textit{idempotent completion} or \textit{Karoubi envelope} $\Kar{\C}$ to be the category with the following data.
\begin{itemize}
	\item Objects in $\Kar{\C}$ are idempotents in $\C$, i.e., an object in $\Kar{\C}$ consists of an object $A$ in $\C$ and a morphism $p:A\to A$ in $\C$ such that $p^2=p$. We will denote this object by $A_p$ and we will denote $A_{\id_A}$ simply by $A_{\id}$.
	\item A 1-morphism between $A_p$ and $B_q$ is given by a morphism $f:A\to B$ in $\C$ such that $f\circ p=f$ and $q\circ f=f$.
	\item Composition of morphisms in $\Kar{\C}$ is given by the composition in $\C$.
	\item The identity morphism on $A_p$ is given by $p:A_p\to A_p$. 
\end{itemize}
\end{definition}

For $\Kar{\C}$ to be the completion of $\C$ with respect to splitting idempotents, we want $\Kar{\C}$ itself to be idempotent complete, which we will show in the following.

\begin{proposition} \label{kar1com}
For any category $\C$, the idempotent completion $\Kar{\C}$ is idempotent complete.
\end{proposition} 

\begin{proof}
Let $A_p$ be an object in $\Kar{\C}$ and $e:A_p \to A_p$ an idempotent on $A_p$, i.e., $e\circ p=e$, $p\circ e=e$ and $e^2=e$. It therefore follows that $A_e$ is an object in $\Kar{\C}$ and we have morphisms $e:A_p\to A_e$ and $e:A_e\to A_p$. Composing these morphisms, we see that $e\circ e=e:A_p\to A_p$ and $e\circ e=e=\id_{A_e}:A_e\to A_e$. Thus the idempotent $e:A_p\to A_p$ splits.
\end{proof}

Since an idempotent $p:A\to A$ in $\C$ defines an idempotent $p:A_{\id}\to A_{\id}$ in $\Kar{\C}$, we also get the following statement.

\begin{remark}\label{eobsplit}
For an object $A_p$ in $\Kar{\C}$, $A_p$ is a splitting of the idempotent $p$ on $A_{\id}$.
\end{remark}

We also want to show that $\Kar{\C}$ is a completion of $\C$ in the sense that $\C$ embeds into $\Kar{\C}$ and that it is equivalent to it if $\C$ was already idempotent complete.

\begin{proposition} \label{kar1emb}
For every category $\C$, there exists a fully faithful functor $\iota_\C:\C\to\Kar{\C}$. If $\C$ is furthermore idempotent complete, i.e., every idempotent in $\C$ splits, this functor is an equivalence.
\end{proposition}

\begin{proof}
We define $\iota_\C$ to map an object $A$ in $\C$ onto the object $A_{\id}$ in $\Kar{\C}$ and a morphism $f:A\to B$ onto $f:A_{\id}\to B_{\id}$. Since any morphism $f:A\to B$ has the property that $f\circ \id_A=f$ and $\id_B\circ f=f$, we see that $\iota_\C$ is fully faithful.

Now assume that $\C$ is idempotent complete and let $A_p$ be an object in $\C$. We know that the idempotent $p:A\to A$ splits in $\C$ via an object $B$. We now have $A_p\iso B_{\id}$ since they both define splittings of $p:A_{\id}\to A_{\id}$ and thus $\iota_\C$ is an equivalence of categories.
\end{proof}

Finally, we want to show that the idempotent completion $\Kar{\C}$ is universal among all possible idempotent completions of $\C$, which is why we can call it \textit{the} idempotent completion of $\C$. By universal, we mean that for any functor $F$ from $\C$ into an arbitrary idempotent complete category $\D$, there is a functor $F\p:\Kar{\C}\to \D$ such that $F^\prime\circ \iota_\C\cong F$. For this we will make use of 2-categories.

This universality takes the form of an adjunction $\Kartwo\dashv \U$ where $\Kartwo:\Cat\to\Catic$ is idempotent completion and $\U:\Catic\to\Cat$ is the forgetful functor where we define $\Catic$ to be the full sub-2-category of idempotent complete categories. The exact definition of an adjunction between 2-categories can be found in the appendix \ref{2adj}.

\begin{theorem} \label{1adj}
Idempotent completion defines a 2-functor $\Kartwo:\Cat\to\Catic$ which is left adjoint to the forgetful 2-functor $\U:\Catic\to\Cat$.
\end{theorem} 

\begin{proof}
The 2-functor $\Kartwo$ maps a category $\C$ to the category $\Kar{\C}$. A functor $F:\C\to\D$ is mapped onto the functor $\Kar{F}:\Kar{\C}\to\Kar{\D}$ which maps an idempotent $A_p$ onto the idempotent $FA_{Fp}$ and a morphism $f:A_p\to B_q$ onto the morphism $Ff:FA_{Fp}\to FB_{Fq}$. The functoriality of $\Kar{F}$ follows directly from the functoriality of $F$. Lastly, a natural transformation $\phi:F\to G$ is mapped onto the natural transformation $\Kar{\phi}:\Kar{F}\to\Kar{G}$ which has components $\Kar{\phi}_{A_p}=G(p)\circ\phi_A=\phi_A\circ F(p):FA_{Fp}\to GA_{Gp}$. This defines the 2-functor $\Kartwo$.

To have an adjunction $\Kartwo\vdash\U$, we now need strong transformations $\eta:\id_\Cat\to \U\circ\Kartwo$ and $\epsilon: \Kartwo\circ\U\to \id_{\Catic}$. These will have component 1-morphisms, i.e., functors $\eta_\C:\C\to\Kar{\C}$ and $\epsilon_\D:\Kar{\D}\to \D$ for categories $\C$ and idempotent complete categories $\D$.

Since we already defined functors $\iota_\C:\C\to\Kar{\C}$, we define $\eta$ to have component 1-morphisms $\iota_\C$ and denote the strong transformation itself simply by $\iota$. This forms a 2-natural transformation as its component 2-morphisms are given by identites. We can define $\epsilon$ and its component 1-morphisms $\epsilon_\D:\Kar{\D}\to\D$ in the following way:

Let $A_p$ be an object in $\Kar{\D}$. Since $\D$ is idempotent complete, we have a splitting $(A,B,f,g)$ of $p$, i.e., an object $B$ in $\D$ and morphisms $f:A\to B$ and $g:B\to A$ such that $g\circ f=p$ and $f\circ g=\id_B$. To define $\epsilon_\D$, we need to choose a splitting for each idempotent $A_p$ and we choose these such that any identity idempotent $A_{\id}$ splits via $(A,A,\id_A,\id_A)$. We now map the idempotent $A_p$ onto $B$. A morphism $h$ between objects $A_p$ and $A\p_{p\p}$ that have splittings $(A,B,f,g)$ and $(A\p,B\p,f\p,g\p)$ is mapped onto the morphism $f\p\circ h\circ g:B\to B\p$. One can check that this assignment is functorial.

In general, $\epsilon$ has non-identity component 2-morphisms. For an object $A_p$ in $\Kar{\D}$ with a splitting $(A,B,f,g)$ and a functor $F:\D\to\D\p$, the splitting we choose for $FA_{Fp}$ does not have to agree with $(FA,FB,Ff,Fg)$. Still, these splittings must be uniquely isomorphic and we definte $\epsilon$ to have component 2-morphisms given by these unique isomorphisms.

Lastly, we need to check that the two triangle identities $(\epsilon\circ\Kartwo)\cdot(\Kartwo\circ\iota)\iso\id_{\Kartwo}$ and ${(\U\circ\epsilon)\cdot(\iota\circ\U)}\iso\id_\U$ hold. By looking at their components, we see that these identities translate to $\epsilon_{\Kar{\C}}\iota_{\Kar{\C}}\iso\id_{\Kar{\C}}$ and $\epsilon_\D\iota_\D\iso\id_\D$, which means we just have to check that $\epsilon_\D\iota_\D\iso\id_\D$ holds for any idempotent complete category $\D$. Since we defined $\epsilon_\D$ by choosing that an identity idempotent splits via the identity, this holds automatically.
\end{proof}

\begin{corollary}\label{kar1equ}
For each category $\C$ and idempotent complete category $\D$, we have an equivalence of categories $\Cat(\Kar{\C},\D)\simeq\Cat(\C,\D)$ induced by precomposing with $\iota_\C$.
\end{corollary}

\begin{proof}
This follows from the bicategorical generalisation that adjunction induce isomorphisms between Hom-sets and is detailed in proposition \ref{2adjequiv}.
\end{proof}

We will later use this adjunction to prove that the 2-categories $\Catic$ and $\Pshic{\bB}$ are cocomplete where $\bB$ is an arbitrary locally idempotent complete bicategory. Details can be found in \ref{app2}. 

\subsection{Cauchy Completion} \label{1cau}

In the following we unpack the definition of Cauchy completion as detailed in \cite{L} and \cite{BD}. In section \ref{sec5} of this paper, we will go over its bicategorical analogue and as to not repeat ourselves, we will not give detailed proofs for the statements in this section. Their proofs can both be readily found in the literature and they also follow from their bicategorical versions since $\Set$ is a full sub-2-category of $\Catic$, the 2-category of idempotent complete categories. By viewing sets as discrete categories, every set becomes an idempotent complete category and every ordinary category becomes one enriched over $\Catic$.

\begin{definition}{(Absolute Colimit)}
Let $\J$, $\C$ be categories and $F:\J\to\C$ a functor. A colimit of $F$ is called an \textit{absolute colimit} if for every category $\D$ and functor $G:\C\to\D$, it is preserved by $G$.
\end{definition}

\begin{example}
The splitting of an idempotent defines an absolute colimit.
\end{example}

\begin{definition}{(Absolute Functor)}
A functor $V:\E\to\J$ is called an \textit{absolute functor} if for every category $\C$ and functor $F:\J\to\C$, the colimit of $FV:\E\to\C$ is absolute if it exists. 
\end{definition}

\begin{example}
We have already seen that $\id_{\clb_1}:\clb_1\to\clb_1$ defines an absolute functor since the colimit of a functor out of $\clb_1$ is given by the splitting of an idempotent. Another, yet trivial, example is the identity functor $\id_\mathbbm{1}$ on the terminal category $\mathbbm{1}$ which has one object and no non-trivial morphisms.
\end{example}

\begin{definition}{(Completeness under Absolute Colimits)}
A category $\C$ is called \textit{complete under absolute colimits} if for every absolute functor $V:\E\to\J$ and every functor $F:\J\to\C$ the colimit of $FV$ exists.
\end{definition}

This definition of completeness might exclude some absolute colimits but we will see that even those absolute colimits that do not stem from absolute functors can be expressed non-trivially via absolute functors.

We will now define the Cauchy completion of a category, which we want to be its completion with respect to absolute colimits.

\begin{definition}{(Cauchy Completion)}
We call an object $A$ in a cocomplete category $\C$ \textit{atomic} if the functor $\C(A,-):\C\to\Set$ is cocontinuous, i.e., preserves all colimits.

Let $\C$ be a category. The \textit{Cauchy completion} of $\C$ is defined to be the the full subcategory of atomic objects in the category $\Psh{\C}$ of presheaves on $\C$, i.e., an object in this category is a functor $S:\C\op\to\Set$ such that $\Psh{\C}(S,-):\Psh{\C}\to\Set$ is cocontinuous. We will denote this category as $\Pshat{\C}$.
\end{definition}

\begin{proposition} \label{repato}
Representable presheaves are atomic.
\end{proposition}

\begin{proof}
For a given representable presheaf $\C(-,c)$, the functor $\Psh{\C}(\C(-,c),-):\Psh{\C}\to\Set$ is isomorphic to the functor $\textup{ev}_c:\Psh{\C}\to\Set$ given by evaluation at $c$ via the Yoneda lemma. Since colimits in $\Psh{\C}$ are computed pointwise, evaluation preserves them, and $\C(-,c)$ is atomic. This also follows from proposition \ref{repato2}.
\end{proof}

We will later see that Cauchy completion is indeed the completion with respect to absolute colimits, yet this does not hold a priori. In general, Cauchy completion is the completion with respect to absolute \textit{weighted} colimits, and only due to a quirk of $\Set$ do these notions agree. To prove this universal property of the Cauchy completion, we will briefly introduce the following.

\begin{definition}{(Weighted Colimit)}
Let $\J$, $\C$ be categories, and $F:\J\to\C$, $W:\J\op\to\Set$ functors. We call $W$ a \textit{weight}. The colimit of $F$ weighted by $W$ is an object representing the functor
\begin{align*}
\Psh{\J}(W,\C(F,-)): \C\to\Set,
\end{align*}
i.e., an object $\colim^W F$ in $\C$ together with a natural isomorphism 
\begin{align*}
\C(\colim^W F,-)\iso\Psh{\J}(W,\C(F,-)).
\end{align*}
\end{definition}

\begin{remark}
Given a functor $F:\J\to\C$, and a weight $W:\J\op\to\Set$ the weighted colimit $\colim^W F$ can be computed as an ordinary colimit in the following way. Applying the Grothendieck construction to $W$, yields its category of elements $\esh W$ together with a forgetful functor $\U:\esh W \to \J$. The colimit of $F \U$ then is canonically isomorphic to the colimit of $F$ weighted by $W$. Conversely, every ordinary colimit can be thought of as a weighted colimit by taking the weight to be the functor which maps each object to the set with one element.

This immediately implies that any functor which preserves all ordinary colimits also preserves all weighted colimits and vice versa. This means we do not need to differentiate between different notions of cocontinuity or atomicity.
\end{remark}

\begin{definition}{(Absolute Weighted Colimit)}
We call a weighted colimit \textit{absolute} if every functor preserves it. For brevity's sake, we will shorten "absolute weighted colimit" to AWC whenever appropriate. We call a weight $W:\J\op\to\Set$ \textit{absolute} if every colimit weighted by $W$ is an AWC.
\end{definition}

\begin{remark}
Applying the Grothendieck construction to an absolute weight $W:\J\op\to\Set$ yields an absolute functor $\U:\esh W \to \J$.
\end{remark}

\begin{definition}{(Completeness under AWCs)}
A category $\C$ is called \textit{complete under AWCs} if for every absolute weight $\J\op\to\Set$ and every functor $F:\J\to\C$ the colimit of $F$ weighted by $W$ exists.
\end{definition}

First, one needs to check that the Cauchy completion of a category is indeed complete under AWCs.

\begin{proposition}\label{cau1com}
For any category $\C$, its Cauchy completion $\Pshat{\C}$ is complete under AWCs.
\end{proposition}

\begin{proof}
Follows from proposition \ref{cau2com}.
\end{proof}

Next, one wants to show that $\Pshat{\C}$ is a completion of $\C$ in the sense that $\C$ embeds into $\Pshat{\C}$ and is equivalent to it if $\C$ was already complete under AWCs.

\begin{proposition}\label{cau1emb}
For every category $\C$, the Yoneda embedding $\yo_\C:\C\to\Psh{\C}$ takes values in $\Pshat{\C}$ and thus defines a fully faithful functor $\yo_\C:\C\to\Pshat{\C}$. If $\C$ is furthermore complete under AWCs, this functor is an equivalence.
\end{proposition}

\begin{proof}
Follows from proposition \ref{cau2emb}.
\end{proof}

Lastly, one would like to show that Cauchy completion is universal among all completions with respect to AWCs, by showing that it is the smallest among all possible completions in the sense that every functor $F:\C\to\D$ from an arbitrary category into a category complete under AWCs factors through the Yoneda embedding $\yo_\C:\C\to\Pshat{\C}$.

\begin{proposition} \label{cau1equ}
For any two categories $\C$, $\D$, with $\D$ being complete under AWCs, there is an equivalence
\begin{align*}
\Cat(\C,\D) \equiv \Cat(\Pshat{\C},\D)
\end{align*}
given by precomposition with $\yo_{\C}$. 
\end{proposition}

\begin{proof}
Follows from corollary \ref{cau2equ}.
\end{proof}

\begin{remark} \label{cau1adj}
Analogously to corollary \ref{kar1equ}, the above equivalence is derived from a left adjoint 2-functor $\Pshattwo:\Cat\to\Cat_{\textup{awc}}$, where $\Cat_{\textup{awc}}$ is the 2-category of categories complete under AWCs. A proof of this statement will be given later.
\end{remark}

\subsection{Their Equivalence}

We will now see that the two constructions of idempotent completion and Cauchy completion are two sides of the same coin in the sense that they yield equivalent categories. To prove this, we first need to prove some statements relating idempotents and absolute colimits.

\begin{lemma} \label{ret1ato}
A retract of an atomic object is atomic.
\end{lemma}

\begin{proof}
Let $B$ be a retract of an atomic object $A$ in a category $\C$, i.e., we have morphisms $f:A\to B$ and $g:B\to A$ such that $fg=\id_B$. Let $F:\J\to \C$ be a functor with universal cone $\lambda:F\to\Delta_{\colim_{\J} F}$. We will need to show that $\lambda_*:\C(B,F)\to\Delta_{\C(B,\colim_{\J} F)}$ also defines a universal cone.

Let $\phi:\C(B,F)\to \Delta_X$ be a natural transformation, we need to show that there is a unique $\psi:\C(B,\colim_{\J} F)\to X$ such that the diagram 
\begin{center}
\begin{tikzcd}
{\C(B,F)} \arrow[d, "\lambda_*"] \arrow[r, "\phi"] & \Delta_X \\
{\Delta_{\C(B,\colim_{\J} F})} \arrow[ru, "\psi"']            &         
\end{tikzcd}
\end{center}
commutes. We can expand this diagram in the following way.
\begin{center}
\begin{tikzcd}
{\C(A,F)} \arrow[d, "\lambda_*"] \arrow[r, "g^*", bend left=15] & {\C(B,F)} \arrow[r, "\phi"] \arrow[d, "\lambda_*"] \arrow[l, "f^*"', bend left=15] & \Delta_X \\
{\Delta_{\C(A,\colim_{\J} F)}} \arrow[r, "g^*"', bend left=15]               & { \Delta_{\C(B,\colim_{\J} F)}} \arrow[ru, "\psi"'] \arrow[l, "f^*", bend left=15]            &  
\end{tikzcd}
\end{center}
Since $A$ is atomic, we know that $\lambda_*:\C(A,F)\to\Delta_{\C(A,\colim_{\J} F)}$ is a universal cone. We can now define $\tilde{\phi}=\phi \circ g^*$ and thus have a unique $\tilde{\psi}: \C(A,\colim_{\J} F)\to X$ such that $\tilde{\psi}\circ \lambda_*=\tilde{\phi}$. We now have a unique $\psi=\tilde{\psi}\circ f^*$ such that $\psi\circ\lambda_*=\phi$.
\end{proof}

\begin{proposition} \label{ato1ret}
Every atomic presheaf is a retract of a representable presheaf. 
\end{proposition}

\begin{proof}
Let $S$ be an atomic presheaf on $\C$. The density theorem states that each presheaf is a colimit of representables. This means there exists a category $\J$, a functor $F:\J\to\C$ and a universal cone $\lambda:\yo_\C F\to \Delta_S$.

Since $S$ is atomic, applying $\Psh{\C}(S,-)$ yields another universal cone
\begin{align*}
\lambda_*:\Psh{\C}(S,\yo_\C F)\to\Delta_{\Psh{\C}(S,S)}.
\end{align*}
Since the functor $\Psh{\C}(S,\yo_\C F):\J\to\Set$ takes values in $\Set$, we can explicitly construct a second colimit cone. We define a relation on the set
\begin{align*}
\coprod_{j\in\Ob\J}\Psh{\C}(S,\C(-,Fj))
\end{align*}
in the following way: For each $\alpha:S\to\C(-,Fj_1)$ and $\beta:S\to\C(-,Fj_2)$, we set $\alpha \sim\beta$ if there exists a morphism $f:j_1\to j_2$ in $\J$ such that $F(f)_*\circ \alpha=\beta$. We now take the equivalence relation generated by this relation and define $\colim_{\J} \Psh{\C}(S,\yo_\C F)$ to be the set of equivalence classes under this relation. The cone 
\begin{align*}
\tilde{\lambda}:\Psh{\C}(S,\yo_\C F)\to\Delta_{\colim_{\J} \Psh{\C}(S,\yo_\C F)}
\end{align*}
is now defined via $\tilde{\lambda}_j(\alpha)=[\alpha]$, i.e., for each $j$ in $\J$, $\tilde{\lambda}_j$ maps $\alpha:S\to \C(-,F(j))$ onto its equivalence class $[\alpha]$. We now also have a unique map
\begin{align*}
\phi:\colim_{\J} \Psh{\C}(S,\yo_\C F)\to\Psh{\C}(S,S)
\end{align*}
such that $\phi\circ\tilde{\lambda}=\lambda_*$, which is defined via $\phi([\alpha])=\lambda_j\circ \alpha$.

We know that this map has to be an isomorphism, which implies that there exists a $j$ in $\J$ and an $\alpha:S\to \C(-,Fj)$ such that $\lambda_j\circ \alpha=\id_S$. Therefore $\alpha\circ \lambda_j$ is an idempotent on $\C(-,Fj)$, which splits via $S$.
\end{proof}

\begin{theorem} \label{kar1cau}
Let $\C$ be a category. The functor given by the composition
\begin{center}
\begin{tikzcd}
\Kar{\C} \arrow[r, "\yo_{\Kar{\C}}"] & \Psh{\Kar{\C}} \arrow[r, "\iota_\C^*"] & \Psh{\C}
\end{tikzcd}
\end{center}
defines an equivalence of categories $\Kar{\C}\equiv\Pshat{\C}$.
\end{theorem}

\begin{proof}
First, we want to show that the functor takes values in atomic objects, i.e., for every object $A_p$ in $\Kar{\C}$, the presheaf $\Kar{\C}(\iota_\C,A_p):\C\op\to\Set$ is atomic. By remark \ref{eobsplit}, $A_p$ is a splitting of the idempotent $p$ on $A_{\id}$. By absoluteness of splittings, we have that $\Kar{\C}(\iota_\C,A_p)$ is a retract of $\Kar{\C}(\iota_\C,A_{\id})=\Kar{\C}(\iota_\C,\iota_\C A)$. Since $\iota_\C$ is  fully faithful, we have $\Kar{\C}(\iota_\C,\iota_\C A)\cong \C(-,A)$. This means that $\Kar{\C}(\iota_\C,A_p)$ is a retract of a representable presheaf and by lemma \ref{ret1ato}, it therefore has to be atomic.

Next, we will show that the functor is fully faithful. Since the Yoneda embedding is fully faithful and by corollary \ref{kar1equ} with $\D=\Set\op$, precomposition with $\iota_\C$ is fully faithful, their composition must also be fully faithful.

Lastly, we need to show that the functor is essentially surjective. Let $S$ be an atomic presheaf on $\C$. By proposition \ref{ato1ret}, there exists an object $A$ in $\C$ and an idempotent $p$ on $A$ such that $S$ is a splitting of the idempotent $p_*$ on $\C(-,A)$. By remark \ref{eobsplit}, $A_p$ is a splitting of the idempotent $p$ on $A_{\id}=\iota_\C A$. By absoluteness of splittings, $\Kar{\C}(\iota_\C,A_p)$ is a splitting of the idempotent $p_*$ on $\Kar{\C}(\iota_\C,\iota_\C A)\iso\C(-,A)$. Since splittings of idempotents are unique up to isomorphisms, it follows that $S\iso \Kar{\C}(\iota_\C,A_p)$.
\end{proof}

With this equivalence, its easy to prove remark \ref{cau1adj}.

\begin{corollary}
The 2-functor $\Pshattwo:\Cat\to\Cat_{\textup{awc}}$ is left adjoint to the forgetful functor $\U:\Cat_{\textup{awc}}\to\Cat$.
\end{corollary}

\begin{proof}
Combining theorem \ref{1adj} with theorem \ref{kar1cau} yields the desired result.
\end{proof}

It now follows that the notions of idempotent completeness and completeness under absolute colimits are equivalent.

\begin{corollary} \label{icisabc}
A category is idempotent complete if and only if it is complete under AWCs.
\end{corollary}

\begin{proof}
Let $\C$ be an idempotent complete category. By proposition \ref{kar1emb} and theorem \ref{kar1cau}, we now have $\C\equiv\Kar{\C}\equiv \Pshat{\C}$. Since $\Pshat{\C}$ is complete under AWCs, $\C$ also must be. The opposite direction follows analogously with proposition \ref{cau1emb}.
\end{proof}

This corollary leads us to be believe that retracts are the only type of absolute colimit that actually exists. We are in fact able to prove the following, validating the statement that any absolute colimit can be expressed via absolute functors.

\begin{lemma} \label{abssplit}
Every absolute colimit is a retract of an object in the image of its defining functor.
\end{lemma}

\begin{proof}
Let $F:\J\to\C$ be a functor and let $\lambda:F\to \Delta_A$ define an absolute colimit for some $A$ in $\C$. We want to show that $A$ is a retract of an object in the image of $F$. If we apply the functor $\C(A,-):\C\to\Set$ to the given colimit, we get a universal cone $\lambda_*:\C(A,F)\to \Delta_{\C(A,A)}$. Since $\C(A,F)$ is now a functor into $\Set$, there is, analogously to the proof of proposition \ref{ato1ret}, an idempotent $p:Fj\to Fj$ for a $j$ in $\J$ which splits via $A$.
\end{proof}

\begin{conjecture}
Furthermore, we conjecture that this idempotent $p:Fj\to Fj$ lies in the image of $F$ as well. 
\end{conjecture}

\section{Idempotent Completion} \label{sec4}

We will now look at the bicategorical analogue of idempotent completion. The ideas laid out in this section follow \cite{DR} and \cite{GJF}, the latter of which is phrased in the language of general $n$-categories. What we present here is a fully rigorous proof of \cite[Theorem 2.3.11.]{GJF} for bicategories. We will suppress coherence data whenever appropriate. Formally, this can be justified since every bicategory is equivalent to a 2-category.

\subsection{2-Idempotents and their Splittings}

To start, we define the bicategorical analogue of an idempotent and of its splitting. These definitions follow \cite{GJF}.

\begin{definition}{(2-Idempotent)}
Let $\bB$ be a bicategory. A \textit{2-idempotent} in $\bB$ consists of an object $A$ in $\bB$, a 1-morphism $p:A\to A$ and 2-morphisms $m:p^2\to p$ and $\Delta:p\to p^2$ such that 
	\begin{align*}
	(\id_p\circ m)\cdot(\Delta\circ \id_p) & =(m\circ \id_p)\cdot(\id_p\circ\Delta)=\Delta\cdot m \text{ and}\\ 
	m\cdot \Delta & =\id_p.
	\end{align*}
We denote a 2-idempotent by $(A,p,m,\Delta)$ and whenever it is clear from context, we will simply denote it as $A_p$.
\end{definition}

\begin{definition}{(2-Idempotent Splitting)}
Let $\bB$ be a bicategory and let $(A,p,m,\Delta)$ be a 2-idempotent in $\bB$. A \textit{splitting} of $A_p$ is given by an object $B$ in $\bB$, 1-morphisms $f:A\to B$ and $g:B\to A$, 2-morphisms $\varphi:f\circ g\to \id_B$, $\psi:\id_B \to f\circ g$ and an isomorphism $\gamma:g\circ f\to p$ such that $m=\gamma\cdot(\id_g \circ \varphi \circ \id_f)\cdot(\gamma^{-1}\circ\gamma^{-1})$ and $\Delta=(\gamma\circ\gamma)\cdot(\id_g \circ \psi \circ \id_f)\cdot\gamma^{-1}$. We say that the 2-idempotent $p$ \textit{splits} and we call $B$ a \textit{2-retract} of $A$. We will later show that in certain cases a splitting of a 2-idempotent is unique up to equivalence.
\end{definition}

\begin{definition}{(2-Idempotent Completeness)}
We call a bicategory $\bB$ \textit{2-idem\-po\-tent complete} if it is locally idempotent complete and if every 2-idempotent in $\bB$ splits.
\end{definition}

\begin{definition}{(Free Walking 2-Idempotent (Splitting))} \label{walkin}
We will define the bicategory $\spd_2$ to be the bicategory with two objects $X$ and $Y$, 1-morphisms freely generated by $f:X\to Y$ and $g:Y\to X$ and 2-morphisms freely generated by $\phi:f\circ g\to \id_Y$ and $\psi:\id_Y \to f\circ g$ satisfying the relation $\varphi\cdot\psi=\id_{\id_Y}$. We will call this bicategory $\spd_2$ the \textit{free walking 2-idempotent splitting}.

The bicategory $\clb_2$ is defined to be the full subbicategory of $\spd_2$ on the object $X$. We will call this bicategory the \textit{free walking 2-idempotent} and we have a fully faithful pseudofunctor $\iota:\clb_2\to\spd_2$.
\end{definition}

\begin{remark}\label{ExtIsSplit}
We can now also think of a 2-idempotent in $\bB$ as a pseudofunctor $F:\clb_2\to\bB$ since every such pseudofunctor determines a 2-idempotent and we can define such a pseudofunctor simply by choosing a 2-idempotent in $\bB$. Furthermore, a splitting of $F$ is then given by a functor $F\p:\spd_2\to\bB$ such that $F\p\iota\equiv F$.
\end{remark}

\subsection{Idempotent Completion of a Bicategory}

In the following chapter, we define the idempotent completion of a locally idempotent complete bicategory $\bB$ and show that it forms a 2-idempotent complete bicategory and furthermore that it is universal among all bicategories with this property.

We use the graphical calculus of 2-categories to ease calculations. We follow the convention that 1-morphisms are read from right to left and 2-morphisms from bottom to top.

\begin{definition}{(Idempotent Completion)}
Let $\bB$ be a locally idempotent complete bicategory. We define the \textit{idempotent completion} $\Kar{\bB}$ to have the following data:

\begin{itemize}
	\item Objects in $\Kar{\bB}$ are 2-idempotents in $\bB$, i.e., an object in $\Kar{\bB}$ is a collection $(A,p,m,\Delta)$ consisting of an object $A$ in $\bB$, a 1-morphism $p:A\to A$ in $\bB$ and 2-morphisms $m:p^2\to p$ and $\Delta:p\to p^2$ in $\bB$ such that
	\begin{align*}
	(\id_p\circ m)\cdot(\Delta\circ \id_p) & =(m\circ \id_p)\cdot(\id_p\circ\Delta)=\Delta\cdot m \text{ and}\\ 
	m\cdot \Delta & =\id_p.
	\end{align*}
	Graphically, this can be represented as follows:
	\begin{center}
	\begin{tikzpicture}
	\draw[dashed] (0,6) --(2,6) --(2,8) --(0,8) --(0,6);
	\draw[dashed] (3,6) --(5,6) --(5,8) --(3,8) --(3,6);
	\draw[dashed] (6,6) --(8,6) --(8,8) --(6,8) --(6,6);
	\draw[line width=0.35mm] (1,6) to (1,8);
	\draw[line width=0.35mm] (3.75,6) to (3.75,6.75) to[out=90,in=180] (4,7) to[out=0,in=90] (4.25,6.75) to (4.25,6);
	\draw[line width=0.35mm] (4,7) to (4,8);
	\draw[line width=0.35mm] (7,6) to (7,7);
	\draw[line width=0.35mm] (6.75,8) to (6.75,7.25) to[out=270,in=180] (7,7) to[out=0,in=270] (7.25,7.25) to (7.25,8);
	\node[right] at (0,7) {$A$};
	\node[left] at (2,7) {$A$};
	\node[right] at (1,7.5) {$p$};
	\node at (1,5.6) {$\id_p$};
	\node at (4,5.6) {$m$};
	\node at (7,5.6) {$\Delta$}; 

	\draw[dashed] (0,3) --(2,3) --(2,5) --(0,5) --(0,3);
	\draw[dashed] (3,3) --(5,3) --(5,5) --(3,5) --(3,3);
	\draw[dashed] (6,3) --(8,3) --(8,5) --(6,5) --(6,3);
	\draw[line width=0.35mm] (0.5,5) to (0.5,3.75) to[out=270,in=180] (0.75,3.5) to[out=0,in=270] (1,3.75) to (1,4.25) to[out=90,in=180] (1.25,4.5) to[out=0,in=90] (1.5,4.25) to (1.5,3);
	\draw[line width=0.35mm] (0.75,3) to (0.75,3.5);
	\draw[line width=0.35mm] (1.25,4.5) to (1.25,5);
	\draw[line width=0.35mm] (3.5,3) to (3.5,4.25) to[out=90,in=180] (3.75,4.5) to[out=0,in=90] (4,4.25) to (4,3.75) to[out=270,in=180] (4.25,3.5) to[out=0,in=270] (4.5,3.75) to (4.5,5);
	\draw[line width=0.35mm] (3.75,5) to (3.75,4.5);
	\draw[line width=0.35mm] (4.25,3.5) to (4.25,3);
	\draw[line width=0.35mm] (6.75,3) to (6.75,3.5) to[out=90,in=180] (7,3.75) to[out=0,in=90] (7.25,3.5) to (7.25,3);
	\draw[line width=0.35mm] (6.75,5) to (6.75,4.5) to[out=270,in=180] (7,4.25) to[out=0,in=270] (7.25,4.5) to (7.25,5);
	\draw[line width=0.35mm] (7,3.75) to (7,4.25);
	\node at (2.5,4) {$=$};
	\node at (5.5,4) {$=$};
	\node at (3.25,2.6) {$(\id_p\circ m)\cdot(\Delta\circ\id_p)=(m\circ\id_p)\cdot(\id_p\circ\Delta)=\Delta\cdot m$};

	\draw[dashed] (1.5,0) --(3.5,0) --(3.5,2) --(1.5,2) --(1.5,0);
	\draw[dashed] (4.5,0) --(6.5,0) --(6.5,2) --(4.5,2) --(4.5,0);
	\draw[line width=0.35mm] (2.75,1) to[out=90,in=0] (2.5,1.25) to[out=180,in=90] (2.25,1) to[out=270,in=180] (2.5,0.75) to[out=0,in=270] (2.75,1);
	\draw[line width=0.35mm] (2.5,0) to (2.5,0.75);
	\draw[line width=0.35mm] (2.5,1.25) to (2.5,2);
	\draw[line width=0.35mm] (5.5,0) to (5.5,2);
	\node at (4,1) {$=$};
	\node at (2.5,-0.4) {$m\cdot\Delta$};
	\node at (4,-0.4) {$=$};
	\node at (5.5,-0.4) {$\id_p$};
	\end{tikzpicture}
	\end{center}
	With these identities, we can derive the following equation:
	\begin{center}
	\begin{tikzpicture}
	\draw[dashed] (0,0) --(2,0) --(2,2) --(0,2) --(0,0);
	\draw[line width=0.35mm] (0.5,0) to (0.5,0.75) to[out=90,in=180] (0.75,1) to[out=0,in=90] (1,0.75) to (1,0);
	\draw[line width=0.35mm] (0.75,1) to[out=90,in=180] (1.125,1.375) to[out=0,in=90] (1.5,1) to (1.5,0);
	\draw[line width=0.35mm] (1.125,1.375) to (1.125,2);
	\draw[dashed] (4,0) --(6,0) --(6,2) --(4,2) --(4,0);
    \draw[line width=0.35mm] (4.5,0) to (4.5,0.83) to[out=90,in=180] (4.67,1) to[out=0,in=90] (4.83,0.83) to[out=270,in=180] (5,0.66) to[out=0,in=270] (5.17,0.83) to[out=90,in=180] (5.33,1) to[out=0,in=90] (5.5,0.83) to (5.5,0);
    \draw[line width=0.35mm] (5,0) to (5,0.66);
    \draw[line width=0.35mm] (4.67,1) to[out=90,in=180] (5,1.33) to[out=0,in=90] (5.33,1);
    \draw[line width=0.35mm] (5,1.33) to (5,2);
	\node at (3,1) {$=$};
	\node at (3.6,-0.4) {$m\cdot(m\circ\id_p)\quad =\ m\cdot(m\circ m)\cdot(\id_p\circ\Delta\circ\id_p)$};
	\end{tikzpicture}
	\end{center}
From this, we can derive associativity and analogously coassociativity.
	\begin{center}
	\begin{tikzpicture}
	\draw[dashed] (0,0) --(2,0) --(2,2) --(0,2) --(0,0);
	\draw[line width=0.35mm] (0.5,0) to (0.5,0.75) to[out=90,in=180] (0.75,1) to[out=0,in=90] (1,0.75) to (1,0);
	\draw[line width=0.35mm] (0.75,1) to[out=90,in=180] (1.125,1.375) to[out=0,in=90] (1.5,1) to (1.5,0);
	\draw[line width=0.35mm] (1.125,1.375) to (1.125,2);
	\draw[dashed] (3,0) --(5,0) --(5,2) --(3,2) --(3,0);
	\draw[line width=0.35mm] (4.5,0) to (4.5,0.75) to[out=90,in=0] (4.25,1) to[out=180,in=90] (4,0.75) to (4,0);
	\draw[line width=0.35mm] (4.25,1) to[out=90,in=0] (3.875,1.375) to[out=180,in=90] (3.5,1) to (3.5,0);
	\draw[line width=0.35mm] (3.875,1.375) to (3.875,2);
	
	\draw[dashed] (6.5,0) --(8.5,0) --(8.5,2) --(6.5,2) --(6.5,0);
	\draw[line width=0.35mm] (7,2) to (7,1.25) to[out=270,in=180] (7.25,1) to[out=0,in=270] (7.5,1.25) to (7.5,2);
	\draw[line width=0.35mm] (7.25,1) to[out=270,in=180] (7.625,0.625) to[out=0,in=270] (8,1) to (8,2);
	\draw[line width=0.35mm] (7.625,0.625) to (7.625,0);
	\draw[dashed] (9.5,0) --(11.5,0) --(11.5,2) --(9.5,2) --(9.5,0);
	\draw[line width=0.35mm] (11,2) to (11,1.25) to[out=270,in=0] (10.75,1) to[out=180,in=270] (10.5,1.25) to (10.5,2);
	\draw[line width=0.35mm] (10.75,1) to[out=270,in=0] (10.375,0.625) to[out=180,in=270] (10,1) to (10,2);
	\draw[line width=0.35mm] (10.375,0.625) to (10.375,0);
	
	\node at (2.5,1) {$=$};
	\node at (9,1) {$=$};
	\node at (2.5,-0.4) {$m\cdot(m\circ\id_p)\quad =\quad m\cdot(\id_p\circ m)$};
	\node at (9,-0.4) {$(\Delta\circ \id_p)\cdot \Delta\quad =\quad (\id_p\circ\Delta)\cdot\Delta$};
	\end{tikzpicture}
	\end{center}
	With this, the data $(A,p,m,\Delta)$ forms a special non-unital non-counital Frobenius algebra. Whenever it is clear from context, we will denote $(A,p,m,\Delta)$ as $A_p$ and $A_{\id_A}$ simply as $A_{\id}$.  
	\item For each pair of objects $A_p=(A,p,m_p,\Delta_p)$ and $B_q=(B,q,m_q,\Delta_q)$ in $\Kar{\bB}$, we have a Hom-category $\Kar{\bB}(A_p,B_q)$. A 1-morphism in $\Kar{\bB}$ is a collection ${(f,\lt,\rho,\rt,\lambda)}$ consisting of a 1-morphism $f:A\to B$ in $\bB$ , and 2-morphisms $\lt:f\circ p \to f$, $\rho:f\to f\circ p$, $\rt:q\circ f \to f$ and $\lambda: f\to q\circ f$ in $\bB$ such that
	\begin{align*}
	& (\id_f\circ m_p)\cdot(\rho\circ \id_p) =(\lt\circ \id_p)\cdot(\id_f\circ\Delta_p)=\rho\cdot \lt, \ \lt\cdot \rho=\id_f, \\
	& (\id_q\circ \rt)\cdot(\Delta_q\circ \id_f) =(m_q\circ \id_f)\cdot (\id_q\circ\lambda)=\lambda\cdot \rt, \ \rt\cdot \lambda=\id_f, \\
	& \rt\cdot(\id_q\circ\lt) =\lt\cdot(\rt\circ\id_p) \text{ and } (\lambda\circ\id_p)\cdot\rho=(\id_q\circ\rho)\cdot\lambda.
	\end{align*}		
	Whenever it is clear from context, we will denote $(f,\lt,\rho,\rt,\lambda)$ as $f:A_p\to B_q$ or even just as $f$. Graphically, this can be represented as follows:
	\tikzmath{\x1 = 0; \y1 =0; 
			\x2 = 6; \y2 =0; 
			\x3 = 0; \y3 =0;
			\x4 = 9; \y4 =0; 
			\x5 = 0; \y5 =2.5; 
			\x6 = 9; \y6 =2.5; 
			\x7 = 0; \y7 =5.5;
			\x8 = 3; \y8 =8.5;
			\x9 = 3; \y9 =11.5;}
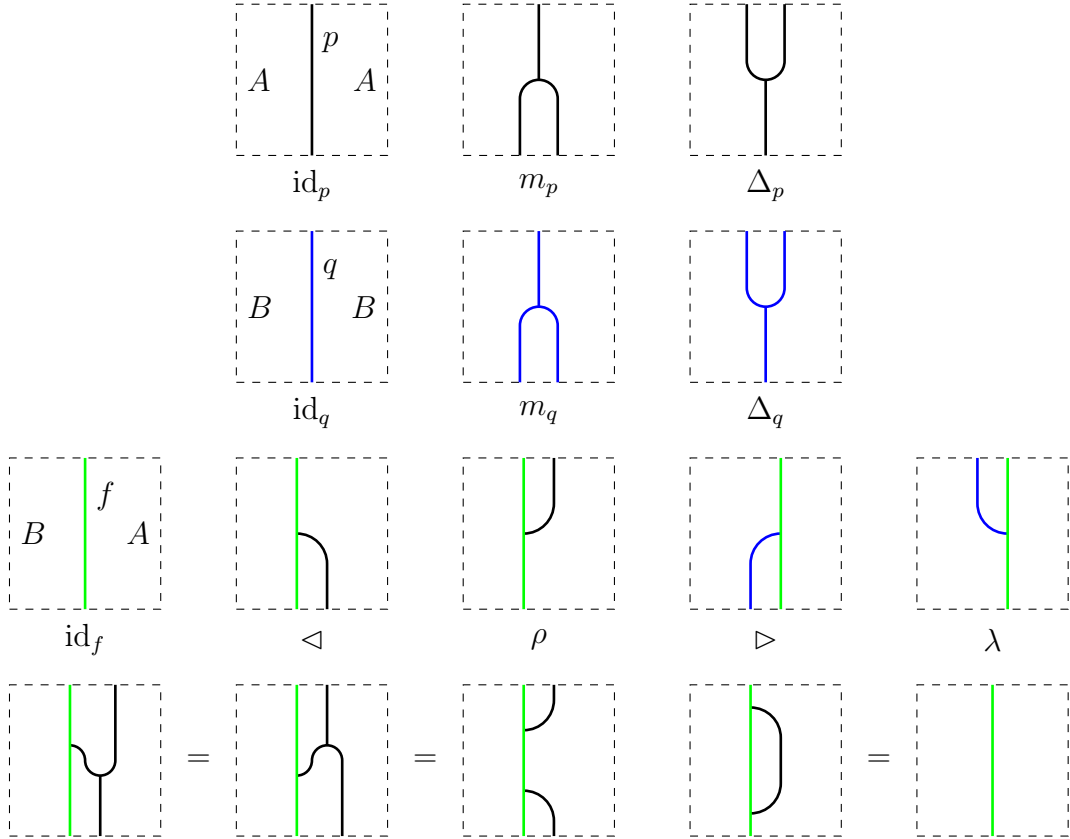
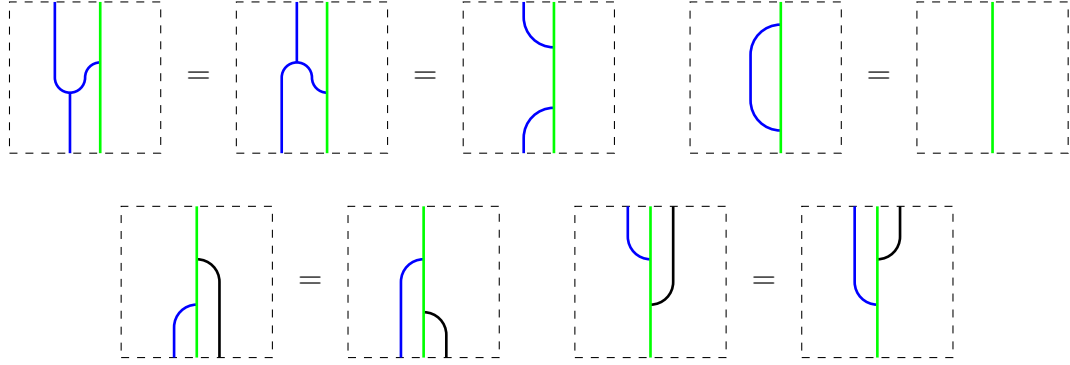
    \begin{center}
    \begin{tikzpicture}
	\draw[dashed] (\x5+0,\y5+0) --(\x5+2,\y5+0) --(\x5+2,\y5+2) --(\x5+0,\y5+2) --(\x5+0,\y5+0);
	\draw[line width=0.35mm] (\x5+1.4,\y5+2) to (\x5+1.4,\y5+1) to[out=270,in=0] (\x5+1.2,\y5+0.8) to[out=180,in=270] (\x5+1,\y5+1) to[out=90, in=0] (\x5+0.8,\y5+1.2);
	\draw[line width=0.35mm] (\x5+1.2,\y5+0) to (\x5+1.2,\y5+0.8);
	\draw[green, line width=0.35mm] (\x5+0.8,\y5+0) to (\x5+0.8,\y5+2);
	
	\node at (\x5+2.5,\y5+1) {$=$};
	
	\draw[dashed] (\x5+3,\y5+0) --(\x5+5,\y5+0) --(\x5+5,\y5+2) --(\x5+3,\y5+2) --(\x5+3,\y5+0);
	\draw[line width=0.35mm] (\x5+4.4,\y5+0) to (\x5+4.4,\y5+1) to[out=90,in=0] (\x5+4.2,\y5+1.2) to[out=180,in=90] (\x5+4,\y5+1) to[out=270,in=0] (\x5+3.8,\y5+0.8);
	\draw[line width=0.35mm] (\x5+4.2,\y5+2) to (\x5+4.2,\y5+1.2);
	\draw[green, line width=0.35mm] (\x5+3.8,\y5+0) to (\x5+3.8,\y5+2);
	
	\node at (\x5+5.5,\y5+1) {$=$};
	
	\draw[dashed] (\x5+6,\y5+0) --(\x5+8,\y5+0) --(\x5+8,\y5+2) --(\x5+6,\y5+2) --(\x5+6,\y5+0);
	\draw[line width=0.35mm] (\x5+7.2,\y5+0) to (\x5+7.2,\y5+0.2) to[out=90,in=0] (\x5+6.8,\y5+0.6);
	\draw[line width=0.35mm] (\x5+7.2,\y5+2) to (\x5+7.2,\y5+1.8) to[out=270,in=0] (\x5+6.8,\y5+1.4);
	\draw[green, line width=0.35mm] (\x5+6.8,\y5+0) to (\x5+6.8,\y5+2);

	\draw[dashed] (\x6+0,\y6+0) --(\x6+2,\y6+0) --(\x6+2,\y6+2) --(\x6+0,\y6+2) --(\x6+0,\y6+0);
	\draw[line width=0.35mm] (\x6+0.8,\y6+0.3) to[out=0,in=270] (\x6+1.2,\y6+0.7) to (\x6+1.2,\y6+1.3) to[out=90,in=0] (\x6+0.8,\y6+1.7);
	\draw[green, line width=0.35mm] (\x6+0.8,\y6+0) to (\x6+0.8,\y6+2);
	
	\node at (\x6+2.5,\y6+1) {$=$};
	
	\draw[dashed] (\x6+3,\y6+0) --(\x6+5,\y6+0) --(\x6+5,\y6+2) --(\x6+3,\y6+2) --(\x6+3,\y6+0);
	\draw[green, line width=0.35mm] (\x6+4,\y6+0) to (\x6+4,\y6+2);

	\draw[dashed] (\x7+0,\y7+0) --(\x7+2,\y7+0) --(\x7+2,\y7+2) --(\x7+0,\y7+2) --(\x7+0,\y7+0);
	\draw[green, line width=0.35mm] (\x7+1,\y7+0) to (\x7+1,\y7+2);
	\node[right] at (\x7+1,\y7+1.5) {$f$};
	\node[right] at (\x7,\y7+1) {$B$};
	\node[left] at (\x7+2,\y7+1) {$A$};
	\node at (\x7+1,\y7-0.4) {$\id_f$};
	
	\draw[dashed] (\x7+3,\y7+0) --(\x7+5,\y7+0) --(\x7+5,\y7+2) --(\x7+3,\y7+2) --(\x7+3,\y7+0);
	\draw[line width=0.35mm] (\x7+4.2,\y7+0) to (\x7+4.2,\y7+0.6) to[out=90,in=0] (\x7+3.8,\y7+1);
	\draw[green, line width=0.35mm] (\x7+3.8,\y7+0) to (\x7+3.8,\y7+2);
	\node at (\x7+4,\y7-0.4) {$\lt$};
	
	\draw[dashed] (\x7+6,\y7+0) --(\x7+8,\y7+0) --(\x7+8,\y7+2) --(\x7+6,\y7+2) --(\x7+6,\y7+0);
	\draw[line width=0.35mm] (\x7+7.2,\y7+2) to (\x7+7.2,\y7+1.4) to[out=270,in=0] (\x7+6.8,\y7+1);
	\draw[green, line width=0.35mm] (\x7+6.8,\y7+0) to (\x7+6.8,\y7+2);
	\node at (\x7+7,\y7-0.4) {$\rho$};
	
	\draw[dashed] (\x7+9,\y7+0) --(\x7+11,\y7+0) --(\x7+11,\y7+2) --(\x7+9,\y7+2) --(\x7+9,\y7+0);
	\draw[blue, line width=0.35mm] (\x7+9.8,\y7+0) to (\x7+9.8,\y7+0.6) to[out=90,in=180] (\x7+10.2,\y7+1);
	\draw[green, line width=0.35mm] (\x7+10.2,\y7+0) to (\x7+10.2,\y7+2);
	\node at (\x7+10,\y7-0.4) {$\rt$};
	
	\draw[dashed] (\x7+12,\y7+0) --(\x7+14,\y7+0) --(\x7+14,\y7+2) --(\x7+12,\y7+2) --(\x7+12,\y7+0);
	\draw[blue, line width=0.35mm] (\x7+12.8,\y7+2) to (\x7+12.8,\y7+1.4) to[out=270,in=180] (\x7+13.2,\y7+1);
	\draw[green, line width=0.35mm] (\x7+13.2,\y7+0) to (\x7+13.2,\y7+2);
	\node at (\x7+13,\y7-0.4) {$\lambda$};

	\draw[dashed] (\x8+0,\y8+0) --(\x8+2,\y8+0) --(\x8+2,\y8+2) --(\x8+0,\y8+2) --(\x8+0,\y8+0);
	\draw[blue, line width=0.35mm] (\x8+1,\y8+0) to (\x8+1,\y8+2);
	\node[right] at (\x8+1,\y8+1.5) {$q$};
	\node[right] at (\x8,\y8+1) {$B$};
	\node[left] at (\x8+2,\y8+1) {$B$};
	\node at (\x8+1,\y8-0.4) {$\id_q$};
	
	\draw[dashed] (\x8+3,\y8+0) --(\x8+5,\y8+0) --(\x8+5,\y8+2) --(\x8+3,\y8+2) --(\x8+3,\y8+0);
	\draw[blue, line width=0.35mm] (\x8+3.75,\y8+0) to (\x8+3.75,\y8+0.75) to[out=90,in=180] (\x8+4,\y8+1) to[out=0,in=90] (\x8+4.25,\y8+0.75) to (\x8+4.25,\y8+0);
	\draw[blue, line width=0.35mm] (\x8+4,\y8+1) to (\x8+4,\y8+2);
	\node at (\x8+4,\y8-0.4) {$m_q$};
	
	\draw[dashed] (\x8+6,\y8+0) --(\x8+8,\y8+0) --(\x8+8,\y8+2) --(\x8+6,\y8+2) --(\x8+6,\y8+0);
	\draw[blue, line width=0.35mm] (\x8+7,\y8+0) to (\x8+7,\y8+1);
	\draw[blue, line width=0.35mm] (\x8+6.75,\y8+2) to (\x8+6.75,\y8+1.25) to[out=270,in=180] (\x8+7,\y8+1) to[out=0,in=270] (\x8+7.25,\y8+1.25) to (\x8+7.25,\y8+2);
	\node at (\x8+7,\y8-0.4) {$\Delta_q$};

    \draw[dashed] (\x9+0,\y9+0) --(\x9+2,\y9+0) --(\x9+2,\y9+2) --(\x9+0,\y9+2) --(\x9+0,\y9+0);
	\draw[line width=0.35mm] (\x9+1,\y9+0) to (\x9+1,\y9+2);
	\node[right] at (\x9+1,\y9+1.5) {$p$};
	\node[right] at (\x9,\y9+1) {$A$};
	\node[left] at (\x9+2,\y9+1) {$A$};
	\node at (\x9+1,\y9-0.4) {$\id_p$};
	
	\draw[dashed] (\x9+3,\y9+0) --(\x9+5,\y9+0) --(\x9+5,\y9+2) --(\x9+3,\y9+2) --(\x9+3,\y9+0);
	\draw[line width=0.35mm] (\x9+3.75,\y9+0) to (\x9+3.75,\y9+0.75) to[out=90,in=180] (\x9+4,\y9+1) to[out=0,in=90] (\x9+4.25,\y9+0.75) to (\x9+4.25,\y9+0);
	\draw[line width=0.35mm] (\x9+4,\y9+1) to (\x9+4,\y9+2);
	\node at (\x9+4,\y9-0.4) {$m_p$};
	
	\draw[dashed] (\x9+6,\y9+0) --(\x9+8,\y9+0) --(\x9+8,\y9+2) --(\x9+6,\y9+2) --(\x9+6,\y9+0);
	\draw[line width=0.35mm] (\x9+7,\y9+0) to (\x9+7,\y9+1);
	\draw[line width=0.35mm] (\x9+6.75,\y9+2) to (\x9+6.75,\y9+1.25) to[out=270,in=180] (\x9+7,\y9+1) to[out=0,in=270] (\x9+7.25,\y9+1.25) to (\x9+7.25,\y9+2);
	\node at (\x9+7,\y9-0.4) {$\Delta_p$};    
    \end{tikzpicture}
    \end{center}
	\begin{center}
	\begin{tikzpicture}
	\draw[dashed] (\x3+0,\y3+0) --(\x3+2,\y3+0) --(\x3+2,\y3+2) --(\x3+0,\y3+2) --(\x3+0,\y3+0);
	\draw[blue, line width=0.35mm] (\x3+0.6,\y3+2) to (\x3+0.6,\y3+1) to[out=270,in=180] (\x3+0.8,\y3+0.8) to[out=0,in=270] (\x3+1,\y3+1) to[out=90, in=180] (\x3+1.2,\y3+1.2);
	\draw[blue, line width=0.35mm] (\x3+0.8,\y3+0) to (\x3+0.8,\y3+0.8);
	\draw[green, line width=0.35mm] (\x3+1.2,\y3+0) to (\x3+1.2,\y3+2);
	
	\node at (\x3+2.5,\y3+1) {$=$};
	
	\draw[dashed] (\x3+3,\y3+0) --(\x3+5,\y3+0) --(\x3+5,\y3+2) --(\x3+3,\y3+2) --(\x3+3,\y3+0);
	\draw[blue, line width=0.35mm] (\x3+3.6,\y3+0) to (\x3+3.6,\y3+1) to[out=90,in=180] (\x3+3.8,\y3+1.2) to[out=0,in=90] (\x3+4,\y3+1) to[out=270,in=180] (\x3+4.2,\y3+0.8);
	\draw[blue, line width=0.35mm] (\x3+3.8,\y3+2) to (\x3+3.8,\y3+1.2);
	\draw[green, line width=0.35mm] (\x3+4.2,\y3+0) to (\x3+4.2,\y3+2);
	
	\node at (\x3+5.5,\y3+1) {$=$};
	
	\draw[dashed] (\x3+6,\y3+0) --(\x3+8,\y3+0) --(\x3+8,\y3+2) --(\x3+6,\y3+2) --(\x3+6,\y3+0);
	\draw[blue, line width=0.35mm] (\x3+6.8,\y3+0) to (\x3+6.8,\y3+0.2) to[out=90,in=180] (\x3+7.2,\y3+0.6);
	\draw[blue, line width=0.35mm] (\x3+6.8,\y3+2) to (\x3+6.8,\y3+1.8) to[out=270,in=180] (\x3+7.2,\y3+1.4);
	\draw[green, line width=0.35mm] (\x3+7.2,\y3+0) to (\x3+7.2,\y3+2);

	\draw[dashed] (\x4+0,\y4+0) --(\x4+2,\y4+0) --(\x4+2,\y4+2) --(\x4+0,\y4+2) --(\x4+0,\y4+0);
	\draw[blue, line width=0.35mm] (\x4+1.2,\y4+0.3) to[out=180,in=270] (\x4+0.8,\y4+0.7) to (\x4+0.8,\y4+1.3) to[out=90,in=180] (\x4+1.2,\y4+1.7);
	\draw[green, line width=0.35mm] (\x4+1.2,\y4+0) to (\x4+1.2,\y4+2);
	
	\node at (\x4+2.5,\y4+1) {$=$};
	
	\draw[dashed] (\x4+3,\y4+0) --(\x4+5,\y4+0) --(\x4+5,\y4+2) --(\x4+3,\y4+2) --(\x4+3,\y4+0);
	\draw[green, line width=0.35mm] (\x4+4,\y4+0) to (\x4+4,\y4+2);
	\end{tikzpicture}
	\end{center}
	\begin{center}
	\begin{tikzpicture}
	\draw[dashed] (\x1+0,\y1+0) --(\x1+2,\y1+0) --(\x1+2,\y1+2) --(\x1+0,\y1+2) --(\x1+0,\y1+0);
	\draw[blue, line width=0.35mm]  (\x1+0.7,\y1+0) to (\x1+0.7,\y1+0.4) to[out=90,in=180] (\x1+1,\y1+0.7);
	\draw[line width=0.35mm] (\x1+1.3,\y1+0) to (\x1+1.3,\y1+1) to[out=90,in=0] (\x1+1,\y1+1.3);
	\draw[green, line width=0.35mm] (\x1+1,\y1+0) to (\x1+1,\y1+2);
	
	\node at (\x1+2.5,\y1+1) {$=$};
	
	\draw[dashed] (\x1+3,\y1+0) --(\x1+5,\y1+0) --(\x1+5,\y1+2) --(\x1+3,\y1+2) --(\x1+3,\y1+0);
	\draw[blue, line width=0.35mm] (\x1+3.7,\y1+0) to (\x1+3.7,\y1+1) to[out=90,in=180] (\x1+4,\y1+1.3);
	\draw[line width=0.35mm] (\x1+4.3,\y1+0) to (\x1+4.3,\y1+0.3) to[out=90,in=0] (\x1+4,\y1+0.6);
	\draw[green, line width=0.35mm] (\x1+4,\y1+0) to (\x1+4,\y1+2);

	\draw[dashed] (\x2+0,\y2+0) --(\x2+2,\y2+0) --(\x2+2,\y2+2) --(\x2+0,\y2+2) --(\x2+0,\y2+0);
	\draw[blue, line width=0.35mm] (\x2+0.7,\y2+2) to (\x2+0.7,\y2+1.6) to[out=270,in=180] (\x2+1,\y2+1.3);
	\draw[line width=0.35mm] (\x2+1.3,\y2+2) to (\x2+1.3,\y2+1) to[out=270,in=0] (\x2+1,\y2+0.7);
	\draw[green, line width=0.35mm] (\x2+1,\y2+0) --(\x2+1,\y2+2);
	
	\node at (\x2+2.5,\y2+1) {$=$};
	
	\draw[dashed] (\x2+3,\y2+0) --(\x2+5,\y2+0) --(\x2+5,\y2+2) --(\x2+3,\y2+2) --(\x2+3,\y2+0);
	\draw[blue, line width=0.35mm] (\x2+3.7,\y2+2) to (\x2+3.7,\y2+1) to[out=270,in=180] (\x2+4,\y2+0.7);
	\draw[line width=0.35mm] (\x2+4.3,\y2+2) to (\x2+4.3,\y2+1.6) to[out=270,in=0] (\x2+4,\y2+1.3);
	\draw[green, line width=0.35mm] (\x2+4,\y2+0) to (\x2+4,\y2+2);
	\end{tikzpicture}
	\end{center}
	We can also think of a 1-morphism in $\Kar{\bB}$ as a bimodule over special non-unital non-counital Frobenius algebras with $\lt$ and $\rho$ being the the right (co)action and $\rt$ and $\lambda$ being the left (co)action.
	\item A 2-morphism between $(f,\lt_f,\rho_f,\rt_f,\lambda_f)$ and $(g,\lt_g,\rho_g,\rt_g,\lambda_g)$ is a 2-morphism $\varphi:f\to g$ in $\bB$ such that
	\begin{align*}
	& \varphi\cdot\lt_f=\lt_g\cdot(\varphi\circ\id_p), \ (\varphi\circ\id_p)\cdot\rho_f=\rho_g\cdot\varphi, \\
	& \varphi\cdot\rt_f=\rt_g\cdot(\id_q\circ \varphi) \text{ and } (\id_q\circ\varphi)\cdot\lambda_f=\lambda_g\cdot \varphi.
	\end{align*}
	This can be graphically represented in the following way:
	\tikzmath{\x1 = 0; \y1 =5.5; 
		\x2 = 1.5; \y2 =2.5; 
		\x3 = 1.5; \y3 =0; }
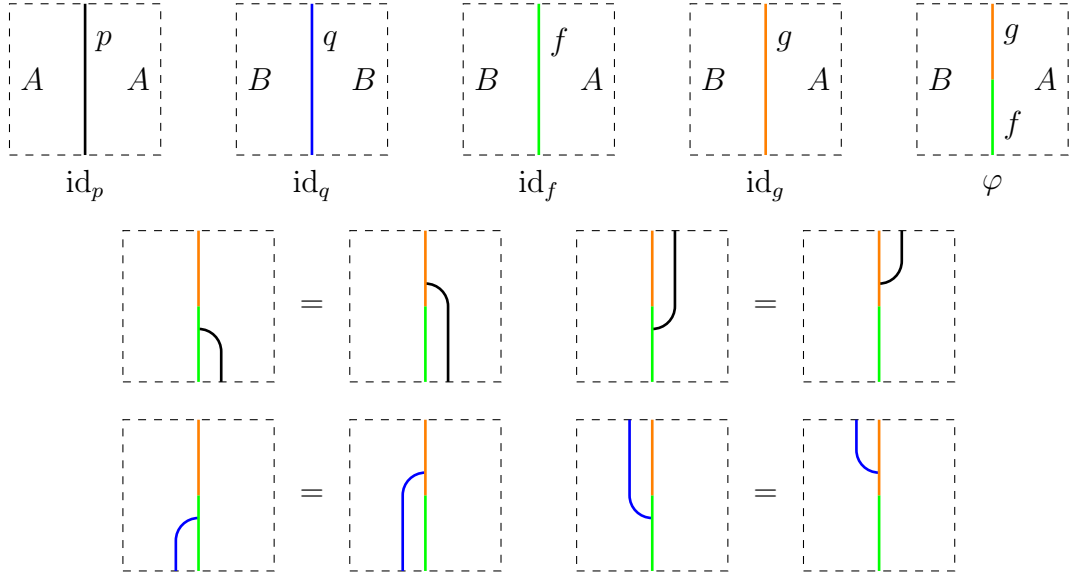
	\begin{center}
	\begin{tikzpicture}
	\draw[dashed] (\x1+0,\y1+0) --(\x1+2,\y1+0) --(\x1+2,\y1+2) --(\x1+0,\y1+2) --(\x1+0,\y1+0);
	\draw[line width=0.35mm] (\x1+1,\y1+0) to (\x1+1,\y1+2);
	\node[right] at (\x1+1,\y1+1.5) {$p$};
	\node[right] at (\x1,\y1+1) {$A$};
	\node[left] at (\x1+2,\y1+1) {$A$};
	\node at (\x1+1,\y1-0.4) {$\id_p$};
	
	\draw[dashed] (\x1+3,\y1+0) --(\x1+5,\y1+0) --(\x1+5,\y1+2) --(\x1+3,\y1+2) --(\x1+3,\y1+0);
	\draw[blue, line width=0.35mm] (\x1+4,\y1+0) to (\x1+4,\y1+2);
	\node[right] at (\x1+4,\y1+1.5) {$q$};
	\node[right] at (\x1+3,\y1+1) {$B$};
	\node[left] at (\x1+5,\y1+1) {$B$};
	\node at (\x1+4,\y1-0.4) {$\id_q$};
	
	\draw[dashed] (\x1+6,\y1+0) --(\x1+8,\y1+0) --(\x1+8,\y1+2) --(\x1+6,\y1+2) --(\x1+6,\y1+0);
	\draw[green, line width=0.35mm] (\x1+7,\y1+0) to (\x1+7,\y1+2);
	\node[right] at (\x1+7,\y1+1.5) {$f$};
	\node[right] at (\x1+6,\y1+1) {$B$};
	\node[left] at (\x1+8,\y1+1) {$A$};
	\node at (\x1+7,\y1-0.4) {$\id_f$};
	
	\draw[dashed] (\x1+9,\y1+0) --(\x1+11,\y1+0) --(\x1+11,\y1+2) --(\x1+9,\y1+2) --(\x1+9,\y1+0);
	\draw[orange, line width=0.35mm] (\x1+10,\y1+0) to (\x1+10,\y1+2);
	\node[right] at (\x1+10,\y1+1.5) {$g$};
	\node[right] at (\x1+9,\y1+1) {$B$};
	\node[left] at (\x1+11,\y1+1) {$A$};
	\node at (\x1+10,\y1-0.4) {$\id_g$};
	
	\draw[dashed] (\x1+12,\y1+0) --(\x1+14,\y1+0) --(\x1+14,\y1+2) --(\x1+12,\y1+2) --(\x1+12,\y1+0);
	\draw[green, line width=0.35mm] (\x1+13,\y1+0) to (\x1+13,\y1+1);
	\draw[orange, line width=0.35mm] (\x1+13,\y1+1) to (\x1+13,\y1+2);
	\node[right] at (\x1+13,\y1+1.6) {$g$};
	\node[right] at (\x1+13,\y1+0.4) {$f$};
	\node[right] at (\x1+12,\y1+1) {$B$};
	\node[left] at (\x1+14,\y1+1) {$A$};
	\node at (\x1+13,\y1-0.4) {$\phi$};

	\draw[dashed] (\x2+0,\y2+0) --(\x2+2,\y2+0) --(\x2+2,\y2+2) --(\x2+0,\y2+2) --(\x2+0,\y2+0);
	\draw[line width=0.35mm] (\x2+1.3,\y2+0) to (\x2+1.3,\y2+0.4) to[out=90, in=0] (\x2+1,\y2+0.7);
	\draw[green, line width=0.35mm] (\x2+1,\y2+0) to (\x2+1,\y2+1);
	\draw[orange, line width=0.35mm] (\x2+1,\y2+1) to (\x2+1,\y2+2);
	
	\node at (\x2+2.5,\y2+1) {$=$};
	
	\draw[dashed] (\x2+3,\y2+0) --(\x2+5,\y2+0) --(\x2+5,\y2+2) --(\x2+3,\y2+2) --(\x2+3,\y2+0);
	\draw[line width=0.35mm] (\x2+4.3,\y2+0) to (\x2+4.3,\y2+1) to[out=90, in=0] (\x2+4,\y2+1.3);
	\draw[green, line width=0.35mm] (\x2+4,\y2+0) to (\x2+4,\y2+1);
	\draw[orange, line width=0.35mm] (\x2+4,\y2+1) to (\x2+4,\y2+2);
	
	\draw[dashed] (\x2+6,\y2+0) --(\x2+8,\y2+0) --(\x2+8,\y2+2) --(\x2+6,\y2+2) --(\x2+6,\y2+0);
	\draw[line width=0.35mm] (\x2+7.3,\y2+2) to (\x2+7.3,\y2+1) to[out=270, in=0] (\x2+7,\y2+0.7);
	\draw[green, line width=0.35mm] (\x2+7,\y2+0) to (\x2+7,\y2+1);
	\draw[orange, line width=0.35mm] (\x2+7,\y2+1) to (\x2+7,\y2+2);

	\node at (\x2+8.5,\y2+1) {$=$};

	\draw[dashed] (\x2+9,\y2+0) --(\x2+11,\y2+0) --(\x2+11,\y2+2) --(\x2+9,\y2+2) --(\x2+9,\y2+0);
	\draw[line width=0.35mm] (\x2+10.3,\y2+2) to (\x2+10.3,\y2+1.6) to[out=270, in=0] (\x2+10,\y2+1.3);
	\draw[green, line width=0.35mm] (\x2+10,\y2+0) to (\x2+10,\y2+1);
	\draw[orange, line width=0.35mm] (\x2+10,\y2+1) to (\x2+10,\y2+2);

	\draw[dashed] (\x3+0,\y3+0) --(\x3+2,\y3+0) --(\x3+2,\y3+2) --(\x3+0,\y3+2) --(\x3+0,\y3+0);
	\draw[blue, line width=0.35mm] (\x3+0.7,\y3+0) to (\x3+0.7,\y3+0.4) to[out=90, in=180] (\x3+1,\y3+0.7);
	\draw[green, line width=0.35mm] (\x3+1,\y3+0) to (\x3+1,\y3+1);
	\draw[orange, line width=0.35mm] (\x3+1,\y3+1) to (\x3+1,\y3+2);
	
	\node at (\x3+2.5,\y3+1) {$=$};
	
	\draw[dashed] (\x3+3,\y3+0) --(\x3+5,\y3+0) --(\x3+5,\y3+2) --(\x3+3,\y3+2) --(\x3+3,\y3+0);
	\draw[blue, line width=0.35mm] (\x3+3.7,\y3+0) to (\x3+3.7,\y3+1) to[out=90, in=180] (\x3+4,\y3+1.3);
	\draw[green, line width=0.35mm] (\x3+4,\y3+0) to (\x3+4,\y3+1);
	\draw[orange, line width=0.35mm] (\x3+4,\y3+1) to (\x3+4,\y3+2);
	
	\draw[dashed] (\x3+6,\y3+0) --(\x3+8,\y3+0) --(\x3+8,\y3+2) --(\x3+6,\y3+2) --(\x3+6,\y3+0);
	\draw[blue, line width=0.35mm] (\x3+6.7,\y3+2) to (\x3+6.7,\y3+1) to[out=270, in=180] (\x3+7,\y3+0.7);
	\draw[green, line width=0.35mm] (\x3+7,\y3+0) to (\x3+7,\y3+1);
	\draw[orange, line width=0.35mm] (\x3+7,\y3+1) to (\x3+7,\y3+2);

	\node at (\x3+8.5,\y3+1) {$=$};

	\draw[dashed] (\x3+9,\y3+0) --(\x3+11,\y3+0) --(\x3+11,\y3+2) --(\x3+9,\y3+2) --(\x3+9,\y3+0);
	\draw[blue, line width=0.35mm] (\x3+9.7,\y3+2) to (\x3+9.7,\y3+1.6) to[out=270, in=180] (\x3+10,\y3+1.3);
	\draw[green, line width=0.35mm] (\x3+10,\y3+0) to (\x3+10,\y3+1);
	\draw[orange, line width=0.35mm] (\x3+10,\y3+1) to (\x3+10,\y3+2);
	\end{tikzpicture}
	\end{center}
	\item Vertical composition, i.e., composition of 2-morphisms is simply given by the composition of 2-morphisms in $\bB$. It follows that identity 2-morphisms are given by the identity 2-morphisms in $\bB$. This turns $\Kar{\bB}(A_p,B_q)$ into a category.
	\item Composition of 1-morphisms (or horizontal composition) is defined as follows. Let $f:A_p\to B_q$ and $g:B_q\to C_r$ be 1-morphisms in $\Kar{\bB}$. We can now form the idempotent
	\begin{align*}
	\kappa_{g\circ f}=(\id_g\circ \rt_f)\cdot(\rho_g\circ\id_f)=(\lt_g\circ \id_f)\cdot(\id_g\circ \lambda_f):g\circ f\to g\circ f.
	\end{align*}		
	We define the composition of $f$ and $g$ as the splitting of this idempotent. We can graphically check that these two ways of writing this morphism do in fact coincide and that it forms an idempotent.
	\tikzmath{\x1 = 3; \y1 =5; 
		\x2 = 0; \y2 =2.5; 
		\x3 = 0; \y3 =0; 
		\x4 = -1.5; \y4 = -2.5; }
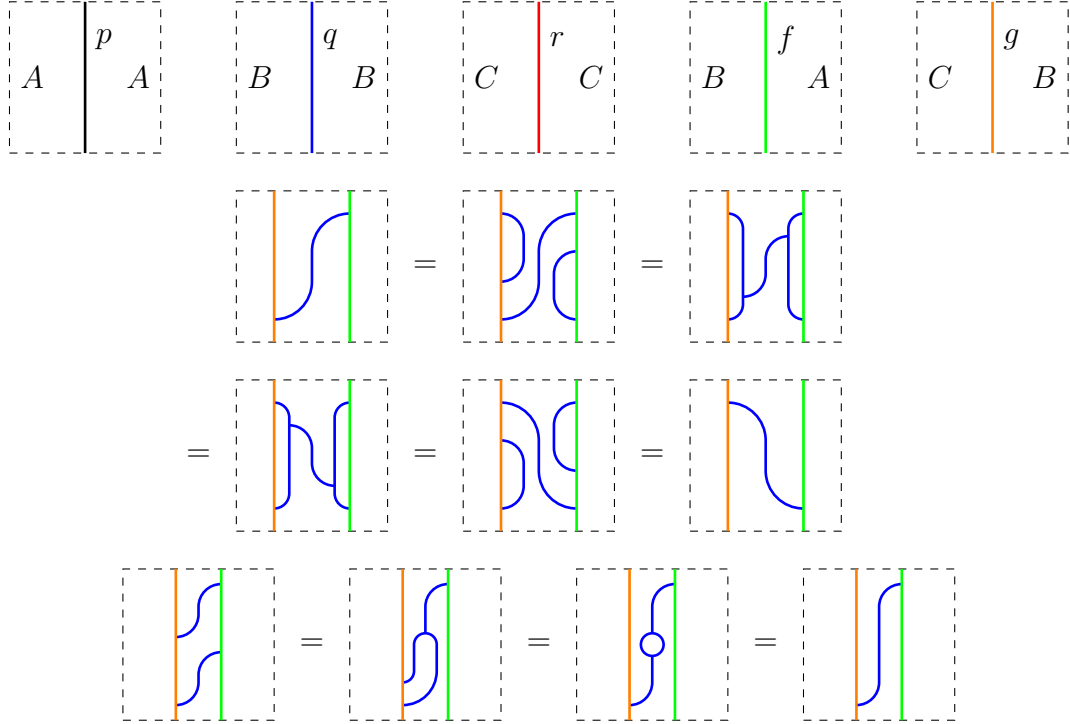
	\begin{center}
	\begin{tikzpicture}
	\draw[dashed] (\x1-6,\y1+0) --(\x1-4,\y1+0) --(\x1-4,\y1+2) --(\x1-6,\y1+2) --(\x1-6,\y1+0);
	\draw[line width=0.35mm] (\x1-5,\y1+0) to (\x1-5,\y1+2);
	\node[right] at (\x1-5,\y1+1.5) {$p$};
	\node[right] at (\x1-6,\y1+1) {$A$};
	\node[left] at (\x1-4,\y1+1) {$A$};	
	
	\draw[dashed] (\x1-3,\y1+0) --(\x1-1,\y1+0) --(\x1-1,\y1+2) --(\x1-3,\y1+2) --(\x1-3,\y1+0);
	\draw[blue, line width=0.35mm] (\x1-2,\y1+0) to (\x1-2,\y1+2);
	\node[right] at (\x1-2,\y1+1.5) {$q$};
	\node[right] at (\x1-3,\y1+1) {$B$};
	\node[left] at (\x1-1,\y1+1) {$B$};	
	
	\draw[dashed] (\x1+0,\y1+0) --(\x1+2,\y1+0) --(\x1+2,\y1+2) --(\x1+0,\y1+2) --(\x1+0,\y1+0);
	\draw[red, line width=0.35mm] (\x1+1,\y1+0) to (\x1+1,\y1+2);
	\node[right] at (\x1+1,\y1+1.5) {$r$};
	\node[right] at (\x1,\y1+1) {$C$};
	\node[left] at (\x1+2,\y1+1) {$C$};
	
	\draw[dashed] (\x1+3,\y1+0) --(\x1+5,\y1+0) --(\x1+5,\y1+2) --(\x1+3,\y1+2) --(\x1+3,\y1+0);
	\draw[green, line width=0.35mm] (\x1+4,\y1+0) to (\x1+4,\y1+2);
	\node[right] at (\x1+4,\y1+1.5) {$f$};
	\node[right] at (\x1+3,\y1+1) {$B$};
	\node[left] at (\x1+5,\y1+1) {$A$};
	
	\draw[dashed] (\x1+6,\y1+0) --(\x1+8,\y1+0) --(\x1+8,\y1+2) --(\x1+6,\y1+2) --(\x1+6,\y1+0);
	\draw[orange, line width=0.35mm] (\x1+7,\y1+0) to (\x1+7,\y1+2);
	\node[right] at (\x1+7,\y1+1.5) {$g$};
	\node[right] at (\x1+6,\y1+1) {$C$};
	\node[left] at (\x1+8,\y1+1) {$B$};

	\draw[dashed] (\x2+0,\y2+0) --(\x2+2,\y2+0) --(\x2+2,\y2+2) --(\x2+0,\y2+2) --(\x2+0,\y2+0);
	\draw[blue, line width=0.35mm] (\x2+0.5,\y2+0.3) to[out=0,in=270] (\x2+1,\y2+0.8) to (\x2+1,\y2+1.2) to[out=90,in=180] (\x2+1.5,\y2+1.7);
	\draw[orange, line width=0.35mm] (\x2+0.5,\y2+0) to (\x2+0.5,\y2+2);
	\draw[green, line width=0.35mm] (\x2+1.5,\y2+0) to (\x2+1.5,\y2+2);

	\node at (\x2+2.5,\y2+1) {$=$};
	
	\draw[dashed] (\x2+3,\y2+0) --(\x2+5,\y2+0) --(\x2+5,\y2+2) --(\x2+3,\y2+2) --(\x2+3,\y2+0);
	\draw[blue, line width=0.35mm] (\x2+3.5,\y2+0.3) to[out=0,in=270] (\x2+4,\y2+0.8) to (\x2+4,\y2+1.2) to[out=90,in=180] (\x2+4.5,\y2+1.7);
	\draw[blue, line width=0.35mm] (\x2+3.5,\y2+0.8) to[out=0,in=270] (\x2+3.8,\y2+1.1) to (\x2+3.8,\y2+1.4) to[out=90,in=0] (\x2+3.5,\y2+1.7);
	\draw[blue, line width=0.35mm] (\x2+4.5,\y2+0.3) to[out=180,in=270] (\x2+4.2,\y2+0.6) to (\x2+4.2,\y2+0.9)  to[out=90,in=180] (\x2+4.5,\y2+1.2);
	\draw[orange, line width=0.35mm] (\x2+3.5,\y2+0) to (\x2+3.5,\y2+2);
	\draw[green, line width=0.35mm] (\x2+4.5,\y2+0) to (\x2+4.5,\y2+2);

	\node at (\x2+5.5,\y2+1) {$=$};
	
	\draw[dashed] (\x2+6,\y2+0) --(\x2+8,\y2+0) --(\x2+8,\y2+2) --(\x2+6,\y2+2) --(\x2+6,\y2+0);
	\draw[blue, line width=0.35mm] (\x2+6.5,\y2+0.3) to[out=0, in=270] (\x2+6.7,\y2+0.5) to (\x2+6.7,\y2+1.5) to[out=90,in=0] (\x2+6.5,\y2+1.7);
	\draw[blue, line width=0.35mm] (\x2+7.5,\y2+0.3) to[out=180, in=270] (\x2+7.3,\y2+0.5) to (\x2+7.3,\y2+1.5) to[out=90,in=180] (\x2+7.5,\y2+1.7);
	\draw[blue, line width=0.35mm] (\x2+6.7,\y2+0.6) to[out=0,in=270] (\x2+7,\y2+0.9) to (\x2+7,\y2+1.1) to[out=90,in=180] (\x2+7.3,\y2+1.4);
	\draw[orange, line width=0.35mm] (\x2+6.5,\y2+0) to (\x2+6.5,\y2+2);
	\draw[green, line width=0.35mm] (\x2+7.5,\y2+0) to (\x2+7.5,\y2+2);

	\node at (\x3-0.5,\y3+1) {$=$};
	
	\draw[dashed] (\x3+0,\y3+0) --(\x3+2,\y3+0) --(\x3+2,\y3+2) --(\x3+0,\y3+2) --(\x3+0,\y3+0);
	\draw[blue, line width=0.35mm] (\x3+0.5,\y3+0.3) to[out=0, in=270] (\x3+0.7,\y3+0.5) to (\x3+0.7,\y3+1.5) to[out=90,in=0] (\x3+0.5,\y3+1.7);
	\draw[blue, line width=0.35mm] (\x3+1.5,\y3+0.3) to[out=180, in=270] (\x3+1.3,\y3+0.5) to (\x3+1.3,\y3+1.5) to[out=90,in=180] (\x3+1.5,\y3+1.7);
	\draw[blue, line width=0.35mm] (\x3+0.7,\y3+1.4) to[out=0,in=90] (\x3+1,\y3+1.1) to (\x3+1,\y3+0.9) to[out=270,in=180] (\x3+1.3,\y3+0.6);
	\draw[orange, line width=0.35mm] (\x3+0.5,\y3+0) to (\x3+0.5,\y3+2);
	\draw[green, line width=0.35mm] (\x3+1.5,\y3+0) to (\x3+1.5,\y3+2);

	\node at (\x3+2.5,\y3+1) {$=$};
	
	\draw[dashed] (\x3+3,\y3+0) --(\x3+5,\y3+0) --(\x3+5,\y3+2) --(\x3+3,\y3+2) --(\x3+3,\y3+0);
	\draw[blue, line width=0.35mm] (\x3+3.5,\y3+1.7) to[out=0,in=90] (\x3+4,\y3+1.2) to (\x3+4,\y3+0.8) to[out=270,in=180] (\x3+4.5,\y3+0.3);
	\draw[blue, line width=0.35mm] (\x3+3.5,\y3+0.3) to[out=0,in=270] (\x3+3.8,\y3+0.6) to (\x3+3.8,\y3+0.9) to[out=90,in=0] (\x3+3.5,\y3+1.2);
	\draw[blue, line width=0.35mm] (\x3+4.5,\y3+1.7) to[out=180,in=90] (\x3+4.2,\y3+1.4) to (\x3+4.2,\y3+1.1)  to[out=270,in=180] (\x3+4.5,\y3+0.8);
	\draw[orange, line width=0.35mm] (\x3+3.5,\y3+0) to (\x3+3.5,\y3+2);
	\draw[green, line width=0.35mm] (\x3+4.5,\y3+0) to (\x3+4.5,\y3+2);

	\node at (\x3+5.5,\y3+1) {$=$};
	
	\draw[dashed] (\x3+6,\y3+0) --(\x3+8,\y3+0) --(\x3+8,\y3+2) --(\x3+6,\y3+2) --(\x3+6,\y3+0);
	\draw[blue, line width=0.35mm] (\x3+6.5,\y3+1.7) to[out=0,in=90] (\x3+7,\y3+1.2) to (\x3+7,\y3+0.8) to[out=270,in=180] (\x3+7.5,\y3+0.3);
	\draw[orange, line width=0.35mm] (\x3+6.5,\y3+0) to (\x3+6.5,\y3+2);
	\draw[green, line width=0.35mm] (\x3+7.5,\y3+0) to (\x3+7.5,\y3+2);

	\draw[dashed] (\x4+0,\y4+0) --(\x4+2,\y4+0) --(\x4+2,\y4+2) --(\x4+0,\y4+2) --(\x4+0,\y4+0);
	\draw[blue, line width=0.35mm] (\x4+0.7,\y4+0.2) to[out=0,in=270] (\x4+1,\y4+0.5) to (\x4+1,\y4+0.6) to[out=90,in=180] (\x4+1.3,\y4+0.9);
	\draw[blue, line width=0.35mm] (\x4+0.7,\y4+1.1) to[out=0,in=270] (\x4+1,\y4+1.4) to (\x4+1,\y4+1.5) to[out=90,in=180] (\x4+1.3,\y4+1.8);
	\draw[orange, line width=0.35mm] (\x4+0.7,\y4+0) to (\x4+0.7,\y4+2);
	\draw[green, line width=0.35mm] (\x4+1.3,\y4+0) to (\x4+1.3,\y4+2);

	\node at (\x4+2.5,\y4+1) {$=$};
	
	\draw[dashed] (\x4+3,\y4+0) --(\x4+5,\y4+0) --(\x4+5,\y4+2) --(\x4+3,\y4+2) --(\x4+3,\y4+0);
	\draw[blue, line width=0.35mm] (\x4+4.3,\y4+1.8) to[out=180,in=90] (\x4+4,\y4+1.5) to (\x4+4,\y4+1.15);
	\draw[blue, line width=0.35mm] (\x4+3.7,\y4+0.5) to[out=0,in=270] (\x4+3.85,\y4+0.65) to (\x4+3.85,\y4+1) to[out=90,in=180] (\x4+4,\y4+1.15) to[out=0,in=90] (\x4+4.15,\y4+1) to (\x4+4.15,\y4+0.65) to[out=270,in=0] (\x4+3.7,\y4+0.2);
	\draw[orange, line width=0.35mm] (\x4+3.7,\y4+0) to (\x4+3.7,\y4+2);
	\draw[green, line width=0.35mm] (\x4+4.3,\y4+0) to (\x4+4.3,\y4+2);

	\node at (\x4+5.5,\y4+1) {$=$};
	
	\draw[dashed] (\x4+6,\y4+0) --(\x4+8,\y4+0) --(\x4+8,\y4+2) --(\x4+6,\y4+2) --(\x4+6,\y4+0);
	\draw[blue, line width=0.35mm] (\x4+7.3,\y4+1.8) to[out=180,in=90] (\x4+7,\y4+1.5) to (\x4+7,\y4+1.15);
	\draw[blue, line width=0.35mm] (\x4+6.7,\y4+0.2) to[out=0,in=270] (\x4+7,\y4+0.5) to (\x4+7,\y4+0.85);
	\draw[blue, line width=0.35mm] (\x4+7,\y4+0.85) to[out=0,in=270] (\x4+7.15,\y4+1) to[out=90,in=0] (\x4+7,\y4+1.15) to[out=180,in=90] (\x4+6.85,\y4+1) to[out=270,in=180] (\x4+7,\y4+0.85);
	\draw[orange, line width=0.35mm] (\x4+6.7,\y4+0) to (\x4+6.7,\y4+2);
	\draw[green, line width=0.35mm] (\x4+7.3,\y4+0) to (\x4+7.3,\y4+2);
	
	\node at (\x4+8.5,\y4+1) {$=$};
	
	\draw[dashed] (\x4+9,\y4+0) --(\x4+11,\y4+0) --(\x4+11,\y4+2) --(\x4+9,\y4+2) --(\x4+9,\y4+0);
	\draw[blue, line width=0.35mm] (\x4+9.7,\y4+0.2) to[out=0,in=270] (\x4+10,\y4+0.5) to (\x4+10,\y4+1.5) to[out=90,in=180] (\x4+10.3,\y4+1.8);
	\draw[orange, line width=0.35mm] (\x4+9.7,\y4+0) to (\x4+9.7,\y4+2);
	\draw[green, line width=0.35mm] (\x4+10.3,\y4+0) to (\x4+10.3,\y4+2);
	\end{tikzpicture}
	\end{center}
	Since $\bB$ is locally idempotent complete, we know that the idempotent $\kappa_{g\circ f}$ admits a splitting and thus we can choose a 1-morphism $g\otimes_q f:A\to C$ and 2-morphisms $\varphi:g\circ f\to g\otimes_q f$ and $\psi:g\otimes_q f\to g\circ f$ such that $\varphi\cdot\psi=\id_{g\otimes_q f}$ and $\psi\cdot\varphi=\kappa_{g\circ f}$, which we write graphically in the following way:
	\tikzmath{\x1 = 1.5; \y1 =3; 
		\x2 = 0; \y2 =0; }
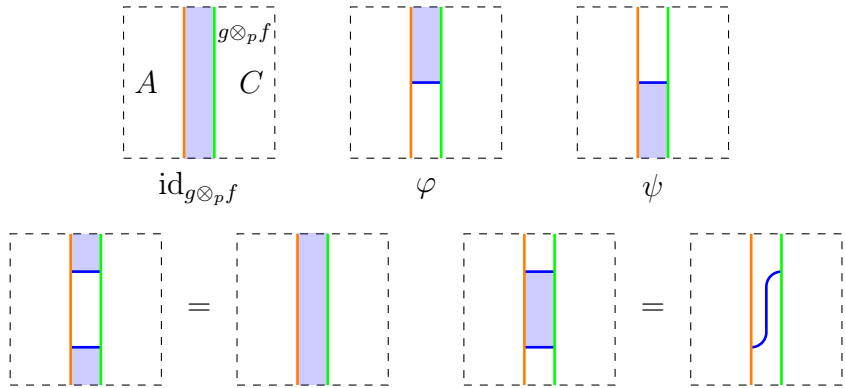
	\begin{center}
	\begin{tikzpicture}
	\draw[dashed] (\x1+0,\y1+0) --(\x1+2,\y1+0) --(\x1+2,\y1+2) --(\x1+0,\y1+2) --(\x1+0,\y1+0);
	\fill[blue, opacity = 0.2] (\x1+0.8,\y1+0) rectangle (\x1+1.2,\y1+2);
	\draw[orange, line width=0.35mm] (\x1+0.8,\y1+0) to (\x1+0.8,\y1+2);
	\draw[green, line width=0.35mm] (\x1+1.2,\y1+0) to (\x1+1.2,\y1+2);
	\node[right] at (\x1+1.1,\y1+1.65) {$_{g\otimes_p f}$};
	\node[right] at (\x1,\y1+1) {$A$};
	\node[left] at (\x1+2,\y1+1) {$C$};
	\node at (\x1+1,\y1-0.4) {$\id_{g \otimes_p f}$};
	
	\draw[dashed] (\x1+3,\y1+0) --(\x1+5,\y1+0) --(\x1+5,\y1+2) --(\x1+3,\y1+2) --(\x1+3,\y1+0);
	\fill[blue, opacity = 0.2] (\x1+3.8,\y1+1) rectangle (\x1+4.2,\y1+2);
	\draw[blue, line width=0.35mm] (\x1+3.8,\y1+1) to (\x1+4.2,\y1+1);
	\draw[orange, line width=0.35mm] (\x1+3.8,\y1+0) to (\x1+3.8,\y1+2);
	\draw[green, line width=0.35mm] (\x1+4.2,\y1+0) to (\x1+4.2,\y1+2);
	\node at (\x1+4,\y1-0.4) {$\phi$};
	
	\draw[dashed] (\x1+6,\y1+0) --(\x1+8,\y1+0) --(\x1+8,\y1+2) --(\x1+6,\y1+2) --(\x1+6,\y1+0);
	\fill[blue, opacity = 0.2] (\x1+6.8,\y1+0) rectangle (\x1+7.2,\y1+1);
	\draw[blue, line width=0.35mm] (\x1+6.8,\y1+1) to (\x1+7.2,\y1+1);
	\draw[orange, line width=0.35mm] (\x1+6.8,\y1+0) to (\x1+6.8,\y1+2);
	\draw[green, line width=0.35mm] (\x1+7.2,\y1+0) to (\x1+7.2,\y1+2);
	\node at (\x1+7,\y1-0.4) {$\psi$};

	\draw[dashed] (\x2+0,\y2+0) --(\x2+2,\y2+0) --(\x2+2,\y2+2) --(\x2+0,\y2+2) --(\x2+0,\y2+0);
	\fill[blue, opacity = 0.2] (\x2+0.8,\y2+0) rectangle (\x2+1.2,\y2+0.5);
	\fill[blue, opacity = 0.2] (\x2+0.8,\y2+1.5) rectangle (\x2+1.2,\y2+2);
	\draw[blue, line width=0.35mm] (\x2+0.8,\y2+0.5) to (\x2+1.2,\y2+0.5);
	\draw[blue, line width=0.35mm] (\x2+0.8,\y2+1.5) to (\x2+1.2,\y2+1.5);
	\draw[orange, line width=0.35mm] (\x2+0.8,\y2+0) to (\x2+0.8,\y2+2);
	\draw[green, line width=0.35mm] (\x2+1.2,\y2+0) to (\x2+1.2,\y2+2);

	\node at (\x2+2.5,\y2+1) {$=$};
	
	\draw[dashed] (\x2+3,\y2+0) --(\x2+5,\y2+0) --(\x2+5,\y2+2) --(\x2+3,\y2+2) --(\x2+3,\y2+0);
	\fill[blue, opacity = 0.2] (\x2+3.8,\y2+0) rectangle (\x2+4.2,\y2+2);
	\draw[orange, line width=0.35mm] (\x2+3.8,\y2+0) to (\x2+3.8,\y2+2);
	\draw[green, line width=0.35mm] (\x2+4.2,\y2+0) to (\x2+4.2,\y2+2);
	
	\draw[dashed] (\x2+6,\y2+0) --(\x2+8,\y2+0) --(\x2+8,\y2+2) --(\x2+6,\y2+2) --(\x2+6,\y2+0);
	\fill[blue, opacity = 0.2] (\x2+6.8,\y2+0.5) rectangle (\x2+7.2,\y2+1.5);
	\draw[blue, line width=0.35mm] (\x2+6.8,\y2+0.5) to (\x2+7.2,\y2+0.5);
	\draw[blue, line width=0.35mm] (\x2+6.8,\y2+1.5) to (\x2+7.2,\y2+1.5);
	\draw[orange, line width=0.35mm] (\x2+6.8,\y2+0) to (\x2+6.8,\y2+2);
	\draw[green, line width=0.35mm] (\x2+7.2,\y2+0) to (\x2+7.2,\y2+2);
	
	\node at (\x2+8.5,\y2+1) {$=$};
	
	\draw[dashed] (\x2+9,\y2+0) --(\x2+11,\y2+0) --(\x2+11,\y2+2) --(\x2+9,\y2+2) --(\x2+9,\y2+0);
	\draw[blue, line width=0.35mm] (\x2+9.8,\y2+0.5) to[out=0,in=270] (\x2+10,\y2+0.7) to (\x2+10,\y2+1.3) to[out=90,in=180] (\x2+10.2,\y2+1.5);
	\draw[orange, line width=0.35mm] (\x2+9.8,\y2+0) to (\x2+9.8,\y2+2);
	\draw[green, line width=0.35mm] (\x2+10.2,\y2+0) to (\x2+10.2,\y2+2);
	\end{tikzpicture}
	\end{center}
	We can now define right-$p$ (co)actions, and left-$r$ (co)actions on $g\otimes_q f$ as follows:
	\tikzmath{\x1 = 0; \y1 =2.5; 
		\x2 = 6; \y2 =2.5; 
		\x3 = 0; \y3 =0; 
		\x4 = 6; \y4 = 0; }
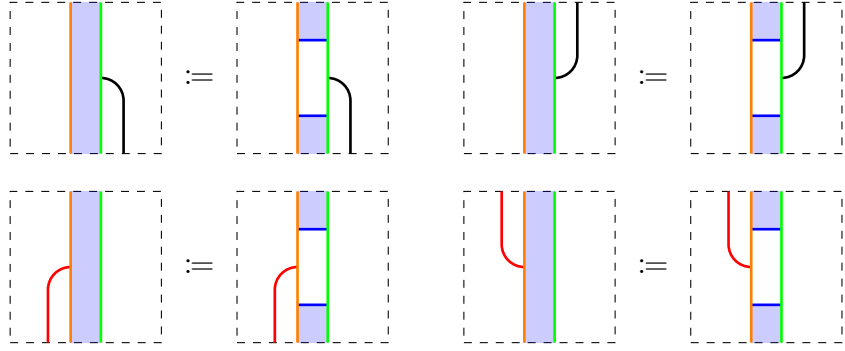
	\begin{center}
	\begin{tikzpicture}
	\draw[dashed] (\x1+0,\y1+0) --(\x1+2,\y1+0) --(\x1+2,\y1+2) --(\x1+0,\y1+2) --(\x1+0,\y1+0);
	\fill[blue, opacity = 0.2] (\x1+0.8,\y1+0) rectangle (\x1+1.2,\y1+2);
	\draw[line width=0.35mm] (\x1+1.5,\y1+0) to (\x1+1.5,\y1+0.7) to[out=90,in=0] (\x1+1.2,\y1+1);
	\draw[orange, line width=0.35mm] (\x1+0.8,\y1+0) to (\x1+0.8,\y1+2);
	\draw[green, line width=0.35mm] (\x1+1.2,\y1+0) to (\x1+1.2,\y1+2);
	
	\node at (\x1+2.5,\y1+1) {$\coloneqq$};
	
	\draw[dashed] (\x1+3,\y1+0) --(\x1+5,\y1+0) --(\x1+5,\y1+2) --(\x1+3,\y1+2) --(\x1+3,\y1+0);
	\fill[blue, opacity = 0.2] (\x1+3.8,\y1+0) rectangle (\x1+4.2,\y1+0.5);
	\fill[blue, opacity = 0.2] (\x1+3.8,\y1+1.5) rectangle (\x1+4.2,\y1+2);
	\draw[blue, line width=0.35mm] (\x1+3.8,\y1+0.5) to (\x1+4.2,\y1+0.5);
	\draw[blue, line width=0.35mm] (\x1+3.8,\y1+1.5) to (\x1+4.2,\y1+1.5);
	\draw[line width=0.35mm] (\x1+4.5,\y1+0) to (\x1+4.5,\y1+0.7) to[out=90,in=0] (\x1+4.2,\y1+1);
	\draw[orange, line width=0.35mm] (\x1+3.8,\y1+0) to (\x1+3.8,\y1+2);
	\draw[green, line width=0.35mm] (\x1+4.2,\y1+0) to (\x1+4.2,\y1+2);

	\draw[dashed] (\x2+0,\y2+0) --(\x2+2,\y2+0) --(\x2+2,\y2+2) --(\x2+0,\y2+2) --(\x2+0,\y2+0);
	\fill[blue, opacity = 0.2] (\x2+0.8,\y2+0) rectangle (\x2+1.2,\y2+2);
	\draw[line width=0.35mm] (\x2+1.5,\y2+2) to (\x2+1.5,\y2+1.3) to[out=270,in=0] (\x2+1.2,\y2+1);
	\draw[orange, line width=0.35mm] (\x2+0.8,\y2+0) to (\x2+0.8,\y2+2);
	\draw[green, line width=0.35mm] (\x2+1.2,\y2+0) to (\x2+1.2,\y2+2);
	
	\node at (\x2+2.5,\y2+1) {$\coloneqq$};
	
	\draw[dashed] (\x2+3,\y2+0) --(\x2+5,\y2+0) --(\x2+5,\y2+2) --(\x2+3,\y2+2) --(\x2+3,\y2+0);
	\fill[blue, opacity = 0.2] (\x2+3.8,\y2+0) rectangle (\x2+4.2,\y2+0.5);
	\fill[blue, opacity = 0.2] (\x2+3.8,\y2+1.5) rectangle (\x2+4.2,\y2+2);
	\draw[blue, line width=0.35mm] (\x2+3.8,\y2+0.5) to (\x2+4.2,\y2+0.5);
	\draw[blue, line width=0.35mm] (\x2+3.8,\y2+1.5) to (\x2+4.2,\y2+1.5);
	\draw[line width=0.35mm] (\x2+4.5,\y2+2) to (\x2+4.5,\y2+1.3) to[out=270,in=0] (\x2+4.2,\y2+1);
	\draw[orange, line width=0.35mm] (\x2+3.8,\y2+0) to (\x2+3.8,\y2+2);
	\draw[green, line width=0.35mm] (\x2+4.2,\y2+0) to (\x2+4.2,\y2+2);

	\draw[dashed] (\x3+0,\y3+0) --(\x3+2,\y3+0) --(\x3+2,\y3+2) --(\x3+0,\y3+2) --(\x3+0,\y3+0);
	\fill[blue, opacity = 0.2] (\x3+0.8,\y3+0) rectangle (\x3+1.2,\y3+2);
	\draw[red, line width=0.35mm] (\x3+0.5,\y3+0) to (\x3+0.5,\y3+0.7) to[out=90,in=180] (\x3+0.8,\y3+1);
	\draw[orange, line width=0.35mm] (\x3+0.8,\y3+0) to (\x3+0.8,\y3+2);
	\draw[green, line width=0.35mm] (\x3+1.2,\y3+0) to (\x3+1.2,\y3+2);
	
	\node at (\x3+2.5,\y3+1) {$\coloneqq$};
	
	\draw[dashed] (\x3+3,\y3+0) --(\x3+5,\y3+0) --(\x3+5,\y3+2) --(\x3+3,\y3+2) --(\x3+3,\y3+0);
	\fill[blue, opacity = 0.2] (\x3+3.8,\y3+0) rectangle (\x3+4.2,\y3+0.5);
	\fill[blue, opacity = 0.2] (\x3+3.8,\y3+1.5) rectangle (\x3+4.2,\y3+2);
	\draw[blue, line width=0.35mm] (\x3+3.8,\y3+0.5) to (\x3+4.2,\y3+0.5);
	\draw[blue, line width=0.35mm] (\x3+3.8,\y3+1.5) to (\x3+4.2,\y3+1.5);
	\draw[red, line width=0.35mm] (\x3+3.5,\y3+0) to (\x3+3.5,\y3+0.7) to[out=90,in=180] (\x3+3.8,\y3+1);
	\draw[orange, line width=0.35mm] (\x3+3.8,\y3+0) to (\x3+3.8,\y3+2);
	\draw[green, line width=0.35mm] (\x3+4.2,\y3+0) to (\x3+4.2,\y3+2);

	\draw[dashed] (\x4+0,\y4+0) --(\x4+2,\y4+0) --(\x4+2,\y4+2) --(\x4+0,\y4+2) --(\x4+0,\y4+0);
	\fill[blue, opacity = 0.2] (\x4+0.8,\y4+0) rectangle (\x4+1.2,\y4+2);
	\draw[red, line width=0.35mm] (\x4+0.5,\y4+2) to (\x4+0.5,\y4+1.3) to[out=270,in=180] (\x4+0.8,\y4+1);
	\draw[orange, line width=0.35mm] (\x4+0.8,\y4+0) to (\x4+0.8,\y4+2);
	\draw[green, line width=0.35mm] (\x4+1.2,\y4+0) to (\x4+1.2,\y4+2);
	
	\node at (\x4+2.5,\y4+1) {$\coloneqq$};
	
	\draw[dashed] (\x4+3,\y4+0) --(\x4+5,\y4+0) --(\x4+5,\y4+2) --(\x4+3,\y4+2) --(\x4+3,\y4+0);
	\fill[blue, opacity = 0.2] (\x4+3.8,\y4+0) rectangle (\x4+4.2,\y4+0.5);
	\fill[blue, opacity = 0.2] (\x4+3.8,\y4+1.5) rectangle (\x4+4.2,\y4+2);
	\draw[blue, line width=0.35mm] (\x4+3.8,\y4+0.5) to (\x4+4.2,\y4+0.5);
	\draw[blue, line width=0.35mm] (\x4+3.8,\y4+1.5) to (\x4+4.2,\y4+1.5);
	\draw[red, line width=0.35mm] (\x4+3.5,\y4+2) to (\x4+3.5,\y4+1.3) to[out=270,in=180] (\x4+3.8,\y4+1);
	\draw[orange, line width=0.35mm] (\x4+3.8,\y4+0) to (\x4+3.8,\y4+2);
	\draw[green, line width=0.35mm] (\x4+4.2,\y4+0) to (\x4+4.2,\y4+2);
	\end{tikzpicture}
	\end{center}
	This in fact turns it into a morphism $g\otimes_q f:A_p\to C_r$. Horizontal composition of 2-morphisms is defined in the following way:
	\begin{center}
	\tikzmath{\x1 = 0; \y1 =2.5; 
		\x2 = 0; \y2 =0;}
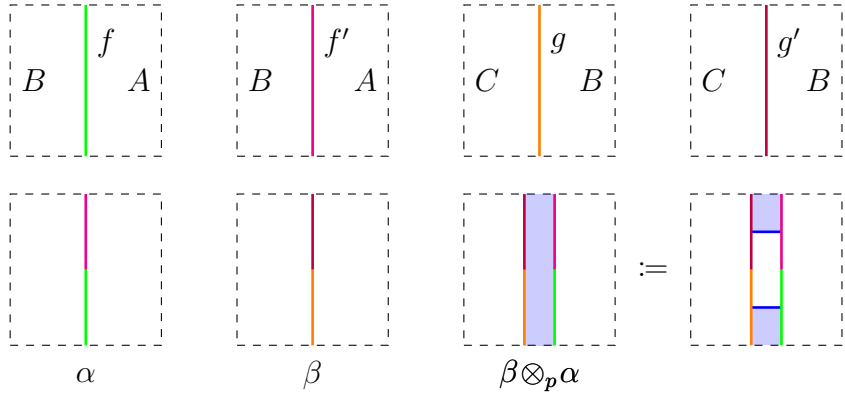
	\begin{tikzpicture}
	\draw[dashed] (\x1+0,\y1+0) --(\x1+2,\y1+0) --(\x1+2,\y1+2) --(\x1+0,\y1+2) --(\x1+0,\y1+0);
	\draw[green, line width=0.35mm] (\x1+1,\y1+0) to (\x1+1,\y1+2);
	\node[right] at (\x1+1,\y1+1.5) {$f$};
	\node[right] at (\x1,\y1+1) {$B$};
	\node[left] at (\x1+2,\y1+1) {$A$};
	
	\draw[dashed] (\x1+3,\y1+0) --(\x1+5,\y1+0) --(\x1+5,\y1+2) --(\x1+3,\y1+2) --(\x1+3,\y1+0);
	\draw[magenta, line width=0.35mm] (\x1+4,\y1+0) to (\x1+4,\y1+2);
	\node[right] at (\x1+4,\y1+1.5) {$f\p$};
	\node[right] at (\x1+3,\y1+1) {$B$};
	\node[left] at (\x1+5,\y1+1) {$A$};
	
	\draw[dashed] (\x1+6,\y1+0) --(\x1+8,\y1+0) --(\x1+8,\y1+2) --(\x1+6,\y1+2) --(\x1+6,\y1+0);
	\draw[orange, line width=0.35mm] (\x1+7,\y1+0) to (\x1+7,\y1+2);
	\node[right] at (\x1+7,\y1+1.5) {$g$};
	\node[right] at (\x1+6,\y1+1) {$C$};
	\node[left] at (\x1+8,\y1+1) {$B$};
	
	\draw[dashed] (\x1+9,\y1+0) --(\x1+11,\y1+0) --(\x1+11,\y1+2) --(\x1+9,\y1+2) --(\x1+9,\y1+0);
	\draw[purple, line width=0.35mm] (\x1+10,\y1+0) to (\x1+10,\y1+2);
	\node[right] at (\x1+10,\y1+1.5) {$g\p$};
	\node[right] at (\x1+9,\y1+1) {$C$};
	\node[left] at (\x1+11,\y1+1) {$B$};

	\draw[dashed] (\x2+0,\y2+0) --(\x2+2,\y2+0) --(\x2+2,\y2+2) --(\x2+0,\y2+2) --(\x2+0,\y2+0);
	\draw[green, line width=0.35mm] (\x2+1,\y2+0) to (\x2+1,\y2+1);
	\draw[magenta, line width=0.35mm] (\x2+1,\y2+1) to (\x2+1,\y2+2);
	\node at (\x2+1,\y2-0.4) {$\alpha$};
	
	\draw[dashed] (\x2+3,\y2+0) --(\x2+5,\y2+0) --(\x2+5,\y2+2) --(\x2+3,\y2+2) --(\x2+3,\y2+0);
	\draw[orange, line width=0.35mm] (\x2+4,\y2+0) to (\x2+4,\y2+1);
	\draw[purple, line width=0.35mm] (\x2+4,\y2+1) to (\x2+4,\y2+2);
	\node at (\x2+4,\y2-0.4) {$\beta$};
	
	\draw[dashed] (\x2+6,\y2+0) --(\x2+8,\y2+0) --(\x2+8,\y2+2) --(\x2+6,\y2+2) --(\x2+6,\y2+0);
	\fill[blue, opacity = 0.2] (\x2+6.8,\y2+0) rectangle (\x2+7.2,\y2+2);
	\draw[green, line width=0.35mm] (\x2+7.2,\y2+0) to (\x2+7.2,\y2+1);
	\draw[magenta, line width=0.35mm] (\x2+7.2,\y2+1) to (\x2+7.2,\y2+2);
	\draw[orange, line width=0.35mm] (\x2+6.8,\y2+0) to (\x2+6.8,\y2+1);
	\draw[purple, line width=0.35mm] (\x2+6.8,\y2+1) to (\x2+6.8,\y2+2);
	\node at (\x2+7,\y2-0.4) {$\beta \otimes_p \alpha$};	
	
	\node at (\x2+8.5,\y2+1) {$\coloneqq$};
	
	\draw[dashed] (\x2+9,\y2+0) --(\x2+11,\y2+0) --(\x2+11,\y2+2) --(\x2+9,\y2+2) --(\x2+9,\y2+0);
	\fill[blue, opacity = 0.2] (\x2+9.8,\y2+0) rectangle (\x2+10.2,\y2+0.5);
	\fill[blue, opacity = 0.2] (\x2+9.8,\y2+1.5) rectangle (\x2+10.2,\y2+2);
	\draw[blue, line width=0.35mm] (\x2+9.8,\y2+0.5) to (\x2+10.2,\y2+0.5);
	\draw[blue, line width=0.35mm] (\x2+9.8,\y2+1.5) to (\x2+10.2,\y2+1.5);
	\draw[green, line width=0.35mm] (\x2+10.2,\y2+0) to (\x2+10.2,\y2+1);
	\draw[magenta, line width=0.35mm] (\x2+10.2,\y2+1) to (\x2+10.2,\y2+2);
	\draw[orange, line width=0.35mm] (\x2+9.8,\y2+0) to (\x2+9.8,\y2+1);
	\draw[purple, line width=0.35mm] (\x2+9.8,\y2+1) to (\x2+9.8,\y2+2);
	\node at (\x2+7,\y2-0.4) {$\beta \otimes_p \alpha$};	
	\end{tikzpicture}
	\end{center}
	This turns composition into a functor
	\begin{align*}
	c_{A_p,B_q,C_r}:\Kar{\bB}(B_q,C_r)\times\Kar{\bB}(A_p,B_q)\to \Kar{\bB}(A_p,C_r).
	\end{align*}	
	It is important to note here that when defining the idempotent completion of a given bicategory $\bB$, we have to make a choice for each pair of morphisms to be able to define this functor. All possible choices will lead to equivalent bicategories, but we have to make these choices none the less. 
	
	Let $D_s$ be another object in $\Kar{\bB}$ and $h:C_r\to D_s$ another 1-morphism. The 1-morphisms $(h\otimes_r g)\otimes_q f:A_p\to D_s$ and  $h\otimes_r (g\otimes_q f):A_p\to D_s$ are now both splittings of the idempotent given by
	
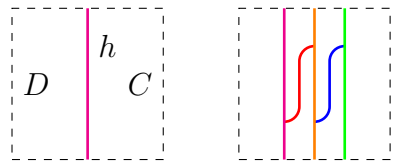
\begin{center}
	\begin{tikzpicture}
	\draw[dashed] (0,0) --(2,0) --(2,2) --(0,2) --(0,0);
	\draw[magenta, line width=0.35mm] (1,0) to (1,2);
	\node[right] at (1,1.5) {$h$};
	\node[right] at (0,1) {$D$};
	\node[left] at (2,1) {$C$};
	
	\draw[dashed] (3,0) --(5,0) --(5,2) --(3,2) --(3,0);
	\draw[red, line width=0.35mm] (3.6,0.5) to[out=0,in=270] (3.8,0.7) to (3.8,1.3) to[out=90,in=180] (4,1.5);
	\draw[blue, line width=0.35mm] (4,0.5) to[out=0,in=270] (4.2,0.7) to (4.2,1.3) to[out=90,in=180] (4.4,1.5);
	\draw[magenta, line width=0.35mm] (3.6,0) to (3.6,2);
	\draw[orange, line width=0.35mm] (4,0) to (4,2);
	\draw[green, line width=0.35mm] (4.4,0) to (4.4,2);
	\end{tikzpicture}
	\end{center}
	and thus we have a unique isomorphism $a_{h,g,f}:(h\otimes_r g)\otimes_q f\to h\otimes_r (g\otimes_q f)$ and in general we get a natural isomorphism
	\begin{align*}
	a_{A_p,B_q,C_r,D_s} : & c_{A_p,B_q,D_s}(c_{B_q,C_r,D_s}\times \id_{\Kar{\bB}(A_p,B_q)}) \\
	& \to c_{A_p,C_r,D_s}(\id_{\Kar{\bB}(C_r,D_s)}\times c_{A_p,B_q,C_r}).
	\end{align*}
	\item The identity 1-morphism on $A_p$ is given by $(p,m_p,\Delta_p,m_p,\Delta_p)$. Let $f:A_p\to B_q$ and $g:C_r\to A_p$ be 1-morphism. We now have the following:
	\tikzmath{\x1 = 1.5; \y1 =2.5; 
		\x2 = 0; \y2 =0;}
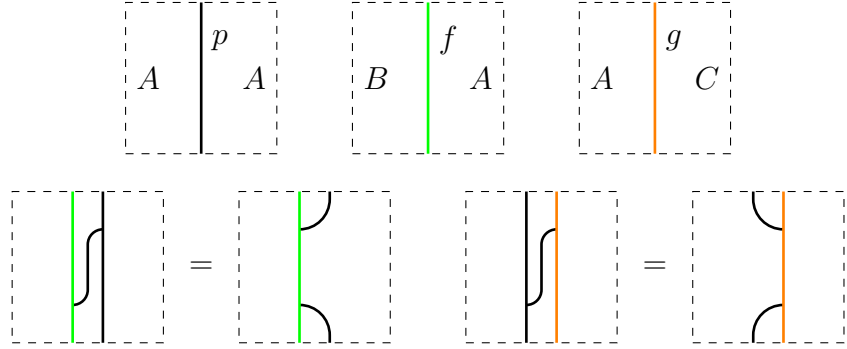
	\begin{center}
	\begin{tikzpicture}
	\draw[dashed] (\x1+0,\y1+0) --(\x1+2,\y1+0) --(\x1+2,\y1+2) --(\x1+0,\y1+2) --(\x1+0,\y1+0);
	\draw[line width=0.35mm] (\x1+1,\y1+0) to (\x1+1,\y1+2);
	\node[right] at (\x1+1,\y1+1.5) {$p$};
	\node[right] at (\x1,\y1+1) {$A$};
	\node[left] at (\x1+2,\y1+1) {$A$};
	
	\draw[dashed] (\x1+3,\y1+0) --(\x1+5,\y1+0) --(\x1+5,\y1+2) --(\x1+3,\y1+2) --(\x1+3,\y1+0);
	\draw[green, line width=0.35mm] (\x1+4,\y1+0) to (\x1+4,\y1+2);
	\node[right] at (\x1+4,\y1+1.5) {$f$};
	\node[right] at (\x1+3,\y1+1) {$B$};
	\node[left] at (\x1+5,\y1+1) {$A$};
	
	\draw[dashed] (\x1+6,\y1+0) --(\x1+8,\y1+0) --(\x1+8,\y1+2) --(\x1+6,\y1+2) --(\x1+6,\y1+0);
	\draw[orange, line width=0.35mm] (\x1+7,\y1+0) to (\x1+7,\y1+2);
	\node[right] at (\x1+7,\y1+1.5) {$g$};
	\node[right] at (\x1+6,\y1+1) {$A$};
	\node[left] at (\x1+8,\y1+1) {$C$};

	\draw[dashed] (\x2+0,\y2+0) --(\x2+2,\y2+0) --(\x2+2,\y2+2) --(\x2+0,\y2+2) --(\x2+0,\y2+0);
	\draw[line width=0.35mm] (\x2+0.8,\y2+0.5) to[out=0,in=270] (\x2+1,\y2+0.7) to (\x2+1,\y2+1.3) to[out=90,in=180] (\x2+1.2,\y2+1.5);
	\draw[green, line width=0.35mm] (\x2+0.8,\y2+0) to (\x2+0.8,\y2+2);
	\draw[line width=0.35mm] (\x2+1.2,\y2+0) to (\x2+1.2,\y2+2);
	
	\node at (\x2+2.5,\y2+1) {$=$};
	
	\draw[dashed] (\x2+3,\y2+0) --(\x2+5,\y2+0) --(\x2+5,\y2+2) --(\x2+3,\y2+2) --(\x2+3,\y2+0);
	\draw[line width=0.35mm] (\x2+4.2,\y2+0) to (\x2+4.2,\y2+0.1) to[out=90,in=0] (\x2+3.8,\y2+0.5);
	\draw[line width=0.35mm] (\x2+4.2,\y2+2) to (\x2+4.2,\y2+1.9) to[out=270,in=0] (\x2+3.8,\y2+1.5);
	\draw[green, line width=0.35mm] (\x2+3.8,\y2+0) to (\x2+3.8,\y2+2);
	
	\draw[dashed] (\x2+6,\y2+0) --(\x2+8,\y2+0) --(\x2+8,\y2+2) --(\x2+6,\y2+2) --(\x2+6,\y2+0);
	\draw[line width=0.35mm] (\x2+6.8,\y2+0.5) to[out=0,in=270] (\x2+7,\y2+0.7) to (\x2+7,\y2+1.3) to[out=90,in=180] (\x2+7.2,\y2+1.5);
	\draw[line width=0.35mm] (\x2+6.8,\y2+0) to (\x2+6.8,\y2+2);
	\draw[orange, line width=0.35mm] (\x2+7.2,\y2+0) to (\x2+7.2,\y2+2);
	
	\node at (\x2+8.5,\y2+1) {$=$};
	
	\draw[dashed] (\x2+9,\y2+0) --(\x2+11,\y2+0) --(\x2+11,\y2+2) --(\x2+9,\y2+2) --(\x2+9,\y2+0);
	\draw[line width=0.35mm] (\x2+9.8,\y2+0) to (\x2+9.8,\y2+0.1) to[out=90,in=180] (\x2+10.2,\y2+0.5);
	\draw[line width=0.35mm] (\x2+9.8,\y2+2) to (\x2+9.8,\y2+1.9) to[out=270,in=180] (\x2+10.2,\y2+1.5);
	\draw[orange, line width=0.35mm] (\x2+10.2,\y2+0) to (\x2+10.2,\y2+2);
	\end{tikzpicture}
	\end{center}
	Since splittings are unique up to unique isomorphism, we have isomorphisms $l_g: p\otimes_p g\to g$ and $r_f:f\otimes_p p \to f$. These form natural isomorphism
	\begin{align*}
	& l_{C_r,A_p}:c_{C_r,A_p,A_p}(p,-)\to \id_{\Kar{\bB}(C_r,A_p)} \text{ and} \\
	& r_{A_p,B_q}:c_{A_p,A_p,B_q}(-,p)\to \id_{\Kar{\bB}(A_p,B_q)}.
	\end{align*}
	\item Since both the associator and left and right unitors are given by unique isomorphisms between colimits, it follows immediately that the triangle and pentagon axiom both have to hold.
\end{itemize}
\end{definition}

For $\Kar{\bB}$ to be the completion of $\bB$ with respect to splitting 2-idempotents, we want $\Kar{\bB}$ itself to be 2-idempotent complete, which we will show in the following.

\begin{proposition} \label{kar2com}
For every locally idempotent complete bicategory $\bB$, the bicategory $\Kar{\bB}$ is 2-idempotent complete.
\end{proposition}

\begin{proof}
We first need to show that $\Kar{\bB}$ is locally idempotent complete. Let $A_p$ and $B_q$ be objects in $\Kar{\bB}$, $(h,\lt_h,\rho_h,\rt_h,\lambda_h):A_p\to B_q$ a 1-morphism and $e:h\to h$ an idempotent 2-morphism. Since $\bB$ is locally idempotent complete, we know that the idempotent $e$ splits into 2-morphisms $f:h\to s$ and $g:s\to h$ such that $gf=e$ and $fg=\id_s$ where $s:A\to B$ is a 1-morphism in $\bB$. Note that, a priori, $s$ is not a morphism in $\Kar{\bB}$ but we can turn $s$ into a morphism $s:A_p\to B_q$ with the following left and right (co)action:

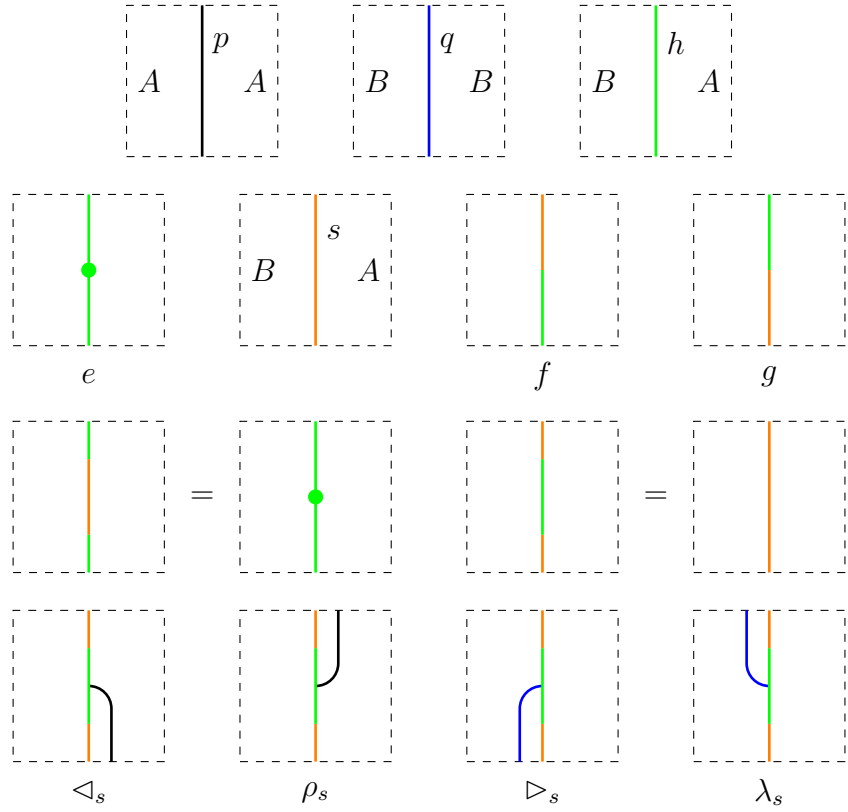
\tikzmath{\x1 = 1.5; \y1 =8; 
	\x2 = 0; \y2 =5.5;
	\x3 = 0; \y3 =2.5;
	\x4 = 0; \y4 =0;}
\begin{center}
\begin{tikzpicture}
\draw[dashed] (\x1+0,\y1+0) --(\x1+2,\y1+0) --(\x1+2,\y1+2) --(\x1+0,\y1+2) --(\x1+0,\y1+0);
\draw[line width=0.35mm] (\x1+1,\y1+0) to (\x1+1,\y1+2);
\node[right] at (\x1+1,\y1+1.5) {$p$};
\node[right] at (\x1,\y1+1) {$A$};
\node[left] at (\x1+2,\y1+1) {$A$};
	
\draw[dashed] (\x1+3,\y1+0) --(\x1+5,\y1+0) --(\x1+5,\y1+2) --(\x1+3,\y1+2) --(\x1+3,\y1+0);
\draw[blue, line width=0.35mm] (\x1+4,\y1+0) to (\x1+4,\y1+2);
\node[right] at (\x1+4,\y1+1.5) {$q$};
\node[right] at (\x1+3,\y1+1) {$B$};
\node[left] at (\x1+5,\y1+1) {$B$};
	
\draw[dashed] (\x1+6,\y1+0) --(\x1+8,\y1+0) --(\x1+8,\y1+2) --(\x1+6,\y1+2) --(\x1+6,\y1+0);
\draw[green, line width=0.35mm] (\x1+7,\y1+0) to (\x1+7,\y1+2);
\node[right] at (\x1+7,\y1+1.5) {$h$};
\node[right] at (\x1+6,\y1+1) {$B$};
\node[left] at (\x1+8,\y1+1) {$A$};

\draw[dashed] (\x2+0,\y2+0) --(\x2+2,\y2+0) --(\x2+2,\y2+2) --(\x2+0,\y2+2) --(\x2+0,\y2+0);
\draw[green, line width=0.35mm] (\x2+1,\y2+0) to (\x2+1,\y2+2);
\fill[green] (\x2+1,\y2+1) circle (0.1);
\node at (\x2+1,\y2-0.4) {$e$};
	
\draw[dashed] (\x2+3,\y2+0) --(\x2+5,\y2+0) --(\x2+5,\y2+2) --(\x2+3,\y2+2) --(\x2+3,\y2+0);
\draw[orange, line width=0.35mm] (\x2+4,\y2+0) to (\x2+4,\y2+2);
\node[right] at (\x2+4,\y2+1.5) {$s$};
\node[right] at (\x2+3,\y2+1) {$B$};
\node[left] at (\x2+5,\y2+1) {$A$};
	
\draw[dashed] (\x2+6,\y2+0) --(\x2+8,\y2+0) --(\x2+8,\y2+2) --(\x2+6,\y2+2) --(\x2+6,\y2+0);
\draw[green, line width=0.35mm] (\x2+7,\y2+0) to (\x2+7,\y2+1);
\draw[orange, line width=0.35mm] (\x2+7,\y2+1) to (\x2+7,\y2+2);
\node at (\x2+7,\y2-0.4) {$f$};

\draw[dashed] (\x2+9,\y2+0) --(\x2+11,\y2+0) --(\x2+11,\y2+2) --(\x2+9,\y2+2) --(\x2+9,\y2+0);
\draw[orange, line width=0.35mm] (\x2+10,\y2+0) to (\x2+10,\y2+1);
\draw[green, line width=0.35mm] (\x2+10,\y2+1) to (\x2+10,\y2+2);
\node at (\x2+10,\y2-0.4) {$g$};

\draw[dashed] (\x3+0,\y3+0) --(\x3+2,\y3+0) --(\x3+2,\y3+2) --(\x3+0,\y3+2) --(\x3+0,\y3+0);
\draw[green, line width=0.35mm] (\x3+1,\y3+0) to (\x3+1,\y3+0.5);
\draw[orange, line width=0.35mm] (\x3+1,\y3+0.5) to (\x3+1,\y3+1.5);
\draw[green, line width=0.35mm] (\x3+1,\y3+1.5) to (\x3+1,\y3+2);

\node at (\x3+2.5,\y3+1) {$=$};
	
\draw[dashed] (\x3+3,\y3+0) --(\x3+5,\y3+0) --(\x3+5,\y3+2) --(\x3+3,\y3+2) --(\x3+3,\y3+0);
\draw[green, line width=0.35mm] (\x3+4,\y3+0) to (\x3+4,\y3+2);
\fill[green] (\x3+4,\y3+1) circle (0.1);
	
\draw[dashed] (\x3+6,\y3+0) --(\x3+8,\y3+0) --(\x3+8,\y3+2) --(\x3+6,\y3+2) --(\x3+6,\y3+0);
\draw[orange, line width=0.35mm] (\x3+7,\y3+0) to (\x3+7,\y3+0.5);
\draw[green, line width=0.35mm] (\x3+7,\y3+0.5) to (\x3+7,\y3+1.5);
\draw[orange, line width=0.35mm] (\x3+7,\y3+1.5) to (\x3+7,\y3+2);

\node at (\x3+8.5,\y3+1) {$=$};

\draw[dashed] (\x3+9,\y3+0) --(\x3+11,\y3+0) --(\x3+11,\y3+2) --(\x3+9,\y3+2) --(\x3+9,\y3+0);
\draw[orange, line width=0.35mm] (\x3+10,\y3+0) to (\x3+10,\y3+2);

\draw[dashed] (\x4+0,\y4+0) --(\x4+2,\y4+0) --(\x4+2,\y4+2) --(\x4+0,\y4+2) --(\x4+0,\y4+0);
\draw[line width=0.35mm] (\x4+1.3,\y4+0) to (\x4+1.3,\y4+0.7) to[out=90,in=0] (\x4+1,\y4+1);
\draw[orange, line width=0.35mm] (\x4+1,\y4+0) to (\x4+1,\y4+0.5);
\draw[green, line width=0.35mm] (\x4+1,\y4+0.5) to (\x4+1,\y4+1.5);
\draw[orange, line width=0.35mm] (\x4+1,\y4+1.5) to (\x4+1,\y4+2);
\node at (\x4+1,\y4-0.4) {$\lt_s$};

\draw[dashed] (\x4+3,\y4+0) --(\x4+5,\y4+0) --(\x4+5,\y4+2) --(\x4+3,\y4+2) --(\x4+3,\y4+0);
\draw[line width=0.35mm] (\x4+4.3,\y4+2) to (\x4+4.3,\y4+1.3) to[out=270,in=0] (\x4+4,\y4+1);
\draw[orange, line width=0.35mm] (\x4+4,\y4+0) to (\x4+4,\y4+0.5);
\draw[green, line width=0.35mm] (\x4+4,\y4+0.5) to (\x4+4,\y4+1.5);
\draw[orange, line width=0.35mm] (\x4+4,\y4+1.5) to (\x4+4,\y4+2);
\node at (\x4+4,\y4-0.4) {$\rho_s$};
	
\draw[dashed] (\x4+6,\y4+0) --(\x4+8,\y4+0) --(\x4+8,\y4+2) --(\x4+6,\y4+2) --(\x4+6,\y4+0);
\draw[blue, line width=0.35mm] (\x4+6.7,\y4+0) to (\x4+6.7,\y4+0.7) to[out=90,in=180] (\x4+7,\y4+1);
\draw[orange, line width=0.35mm] (\x4+7,\y4+0) to (\x4+7,\y4+0.5);
\draw[green, line width=0.35mm] (\x4+7,\y4+0.5) to (\x4+7,\y4+1.5);
\draw[orange, line width=0.35mm] (\x4+7,\y4+1.5) to (\x4+7,\y4+2);
\node at (\x4+7,\y4-0.4) {$\rt_s$};

\draw[dashed] (\x4+9,\y4+0) --(\x4+11,\y4+0) --(\x4+11,\y4+2) --(\x4+9,\y4+2) --(\x4+9,\y4+0);
\draw[blue, line width=0.35mm] (\x4+9.7,\y4+2) to (\x4+9.7,\y4+1.3) to[out=270,in=180] (\x4+10,\y4+1);
\draw[orange, line width=0.35mm] (\x4+10,\y4+0) to (\x4+10,\y4+0.5);
\draw[green, line width=0.35mm] (\x4+10,\y4+0.5) to (\x4+10,\y4+1.5);
\draw[orange, line width=0.35mm] (\x4+10,\y4+1.5) to (\x4+10,\y4+2);
\node at (\x4+10,\y4-0.4) {$\lambda_s$};
\end{tikzpicture}
\end{center}

With this definition of $s$, $f$ and $g$ also become 2-morphisms in $\Kar{\bB}$ and therefore split the idempotent $e$. Thus $\Kar{\bB}$ is locally idempotent complete.

Now, let $(A,p,m,\Delta)$ be an object in $\Kar{\bB}$ and $(e,\lt,\rho,\rt,\lambda)$ a 2-idempotent on $A_p$, which means we have morphisms $m^\otimes_e:e\otimes_p e\to e$ and $\Delta^\otimes_e:e\to e\otimes_p e$ such that
\begin{align*}
(\id_e\otimes_p m^\otimes_e)\cdot(\Delta^\otimes_e\otimes_p \id_e) & =(m^\otimes_e\otimes_p \id_e)\cdot(\id_e\otimes_p\Delta^\otimes_e)=\Delta^\otimes_e\cdot m^\otimes_e \text{ and} \\
m^\otimes_e\cdot\Delta^\otimes_e & =\id_e.
\end{align*}
We want to show that this 2-idempotent splits. From the definition of $e\otimes_p e$ we have 2-morphisms $\varphi:e\circ e\to e\otimes_p e$ and $\psi:e\otimes_p e\to e\circ e$ such that $\varphi\cdot\psi=\id_{e\otimes_p e}$ and $\psi\cdot\varphi=\kappa_{e\circ e}$.

Combining these we get morphisms $m_e=m_e^\otimes\cdot\varphi:e\circ e\to e$ and $\Delta_e=\psi\cdot\Delta_e^\otimes:e\to e\circ e$. Graphically, we can represent these in the following way:
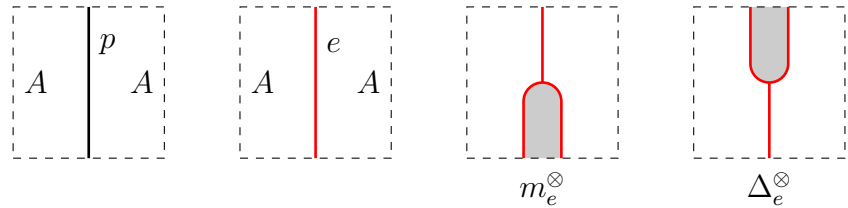
\tikzmath{\x1 = 0; \y1 =3; 
	\x2 = 0; \y2 =0;}
\begin{center}
\begin{tikzpicture}
\draw[dashed] (\x1+0,\y1+0) --(\x1+2,\y1+0) --(\x1+2,\y1+2) --(\x1+0,\y1+2) --(\x1+0,\y1+0);
\draw[line width=0.35mm] (\x1+1,\y1+0) to (\x1+1,\y1+2);
\node[right] at (\x1+1,\y1+1.5) {$p$};
\node[right] at (\x1,\y1+1) {$A$};
\node[left] at (\x1+2,\y1+1) {$A$};
	
\draw[dashed] (\x1+3,\y1+0) --(\x1+5,\y1+0) --(\x1+5,\y1+2) --(\x1+3,\y1+2) --(\x1+3,\y1+0);
\draw[red, line width=0.35mm] (\x1+4,\y1+0) to (\x1+4,\y1+2);
\node[right] at (\x1+4,\y1+1.5) {$e$};
\node[right] at (\x1+3,\y1+1) {$A$};
\node[left] at (\x1+5,\y1+1) {$A$};
	
\draw[dashed] (\x1+6,\y1+0) --(\x1+8,\y1+0) --(\x1+8,\y1+2) --(\x1+6,\y1+2) --(\x1+6,\y1+0);
\filldraw[red, line width=0.35mm, fill = black, fill opacity = 0.2] (\x1+6.75,\y1+0) to (\x1+6.75,\y1+0.75) to[out=90,in=180] (\x1+7,\y1+1) to[out=0,in=90] (\x1+7.25,\y1+0.75) to (\x1+7.25,\y1+0);
\draw[red, line width=0.35mm] (\x1+7,\y1+1) to (\x1+7,\y1+2);
\node at (\x1+7,\y1-0.4) {$m_e^\otimes$};
	
\draw[dashed] (\x1+9,\y1+0) --(\x1+11,\y1+0) --(\x1+11,\y1+2) --(\x1+9,\y1+2) --(\x1+9,\y1+0);
\filldraw[red, line width=0.35mm, fill = black, fill opacity = 0.2] (\x1+9.75,\y1+2) to (\x1+9.75,\y1+1.25) to[out=270,in=180] (\x1+10,\y1+1) to[out=0,in=270] (\x1+10.25,\y1+1.25) to (\x1+10.25,\y1+2);
\draw[red, line width=0.35mm] (\x1+10,\y1+0) to (\x1+10,\y1+1);
\node at (\x1+10,\y1-0.4) {$\Delta_e^\otimes$};
\end{tikzpicture}
\end{center}

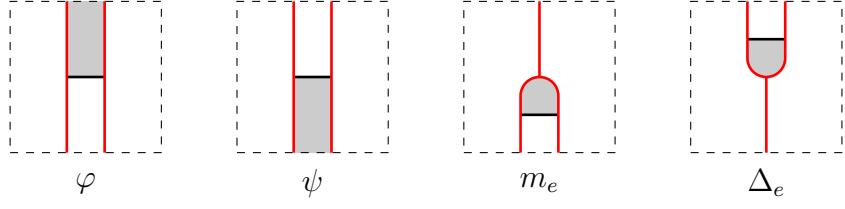
\begin{center}
\begin{tikzpicture}
\draw[dashed] (\x2+0,\y2+0) --(\x2+2,\y2+0) --(\x2+2,\y2+2) --(\x2+0,\y2+2) --(\x2+0,\y2+0);
\draw[line width=0.35mm] (\x2+0.75,\y2+1) to (\x2+1.25,\y2+1);
\fill[black, opacity=0.2] (\x2+0.75,\y2+1) rectangle (\x2+1.25,\y2+2);
\draw[red, line width=0.35mm] (\x2+0.75,\y2+0) to (\x2+0.75,\y2+2);
\draw[red, line width=0.35mm] (\x2+1.25,\y2+0) to (\x2+1.25,\y2+2);
\node at (\x2+1,\y2-0.4) {$\phi$};
	
\draw[dashed] (\x2+3,\y2+0) --(\x2+5,\y2+0) --(\x2+5,\y2+2) --(\x2+3,\y2+2) --(\x2+3,\y2+0);
\draw[line width=0.35mm] (\x2+3.75,\y2+1) to (\x2+4.25,\y2+1);
\fill[black, opacity=0.2] (\x2+3.75,\y2+0) rectangle (\x2+4.25,\y2+1);
\draw[red, line width=0.35mm] (\x2+3.75,\y2+0) to (\x2+3.75,\y2+2);
\draw[red, line width=0.35mm] (\x2+4.25,\y2+0) to (\x2+4.25,\y2+2);
\node at (\x2+4,\y2-0.4) {$\psi$};
	
\draw[dashed] (\x2+6,\y2+0) --(\x2+8,\y2+0) --(\x2+8,\y2+2) --(\x2+6,\y2+2) --(\x2+6,\y2+0);
\draw[line width=0.35mm] (\x2+6.75,\y2+0.5) to (\x2+7.25,\y2+0.5);
\filldraw[red, line width=0.35mm, fill = black, fill opacity = 0.2] (\x2+6.75,\y2+0.5) to (\x2+6.75,\y2+0.75) to[out=90,in=180] (\x2+7,\y2+1) to[out=0,in=90] (\x2+7.25,\y2+0.75) to (\x2+7.25,\y2+0.5);
\draw[red, line width=0.35mm] (\x2+6.75,\y2+0) to (\x2+6.75,\y2+0.5);
\draw[red, line width=0.35mm] (\x2+7.25,\y2+0) to (\x2+7.25,\y2+0.5);
\draw[red, line width=0.35mm] (\x2+7,\y2+1) to (\x2+7,\y2+2);
\node at (\x2+7,\y2-0.4) {$m_e$};
	
\draw[dashed] (\x2+9,\y2+0) --(\x2+11,\y2+0) --(\x2+11,\y2+2) --(\x2+9,\y2+2) --(\x2+9,\y2+0);
\draw[line width=0.35mm] (\x2+9.75,\y2+1.5) to (\x2+10.25,\y2+1.5);
\filldraw[red, line width=0.35mm, fill = black, fill opacity = 0.2] (\x2+9.75,\y2+1.5) to (\x2+9.75,\y2+1.25) to[out=270,in=180] (\x2+10,\y2+1) to[out=0,in=270] (\x2+10.25,\y2+1.25) to (\x2+10.25,\y2+1.5);
\draw[red, line width=0.35mm] (\x2+9.75,\y2+1.5) to (\x2+9.75,\y2+2);
\draw[red, line width=0.35mm] (\x2+10.25,\y2+1.5) to (\x2+10.25,\y2+2);
\draw[red, line width=0.35mm] (\x2+10,\y2+0) to (\x2+10,\y2+1);
\node at (\x2+10,\y2-0.4) {$\Delta_e$};
\end{tikzpicture}
\end{center}

This defines a 2-idempotent in $\bB$ as we will show graphically:
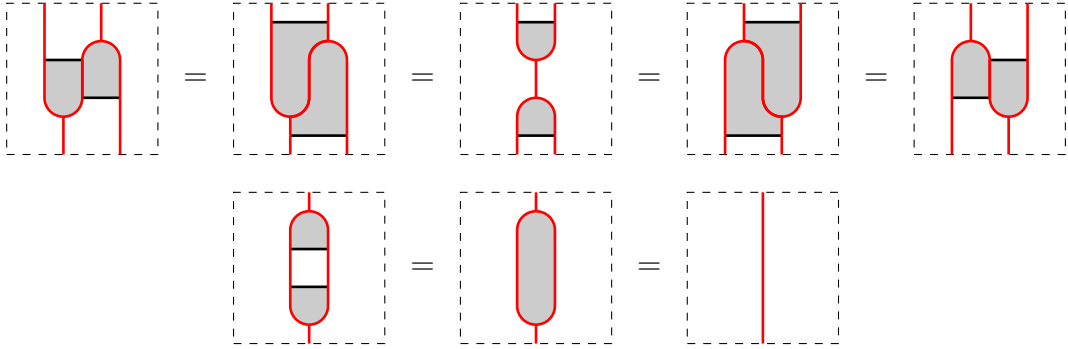
\tikzmath{\x1 = 0; \y1 =2.5; 
	\x2 = 3; \y2 =0;}
\begin{center}
\begin{tikzpicture}
\draw[dashed] (\x1+0,\y1+0) --(\x1+2,\y1+0) --(\x1+2,\y1+2) --(\x1+0,\y1+2) --(\x1+0,\y1+0);
\draw[line width=0.35mm] (\x1+0.5,\y1+1.25) to (\x1+1,\y1+1.25);
\draw[line width=0.35mm] (\x1+1,\y1+0.75) to (\x1+1.5,\y1+0.75);
\filldraw[red, line width=0.35mm, fill = black, fill opacity = 0.2] (\x1+0.5,\y1+1.25) to (\x1+0.5,\y1+0.75) to[out=270,in=180] (\x1+0.75,\y1+0.5) to[out=0,in=270] (\x1+1,\y1+0.75) to (\x1+1,\y1+1.25);
\filldraw[red, line width=0.35mm, fill = black, fill opacity = 0.2] (\x1+1,\y1+0.75) to (\x1+1,\y1+1.25) to[out=90,in=180] (\x1+1.25,\y1+1.5) to[out=0,in=90] (\x1+1.5,\y1+1.25) to (\x1+1.5,\y1+0.75);
\draw[red, line width=0.35mm] (\x1+0.5,\y1+2) to (\x1+0.5,\y1+1.25);
\draw[red, line width=0.35mm] (\x1+0.75,\y1+0) to (\x1+0.75,\y1+0.5);
\draw[red, line width=0.35mm] (\x1+1.25,\y1+2) to (\x1+1.25,\y1+1.5);
\draw[red, line width=0.35mm] (\x1+1.5,\y1+0) to (\x1+1.5,\y1+0.75);

\node at (\x1+2.5,\y1+1) {$=$};
	
\draw[dashed] (\x1+3,\y1+0) --(\x1+5,\y1+0) --(\x1+5,\y1+2) --(\x1+3,\y1+2) --(\x1+3,\y1+0);
\draw[line width=0.35mm] (\x1+3.5,\y1+1.75) to (\x1+4.25,\y1+1.75);
\draw[line width=0.35mm] (\x1+3.75,\y1+0.25) to (\x1+4.5,\y1+0.25);
\filldraw[red, line width=0.35mm, fill = black, fill opacity = 0.2] (\x1+3.5,\y1+1.75) to (\x1+3.5,\y1+0.75) to[out=270,in=180] (\x1+3.75,\y1+0.5) to[out=0,in=270] (\x1+4,\y1+0.75) to (\x1+4,\y1+1.25) to[out=90,in=180] (\x1+4.25,\y1+1.5) to (\x1+4.25,\y1+1.75);
\filldraw[red, line width=0.35mm, fill = black, fill opacity = 0.2] (\x1+3.75,\y1+0.25) to (\x1+3.75,\y1+0.5) to[out=0,in=270] (\x1+4,\y1+0.75) to (\x1+4,\y1+1.25) to[out=90,in=180] (\x1+4.25,\y1+1.5) to[out=0,in=90] (\x1+4.5,\y1+1.25) to (\x1+4.5,\y1+0.25);
\draw[red, line width=0.35mm] (\x1+3.75,\y1+0) to (\x1+3.75,\y1+0.25);
\draw[red, line width=0.35mm] (\x1+4.5,\y1+0) to (\x1+4.5,\y1+0.25);
\draw[red, line width=0.35mm] (\x1+3.5,\y1+2) to (\x1+3.5,\y1+1.75);
\draw[red, line width=0.35mm] (\x1+4.25,\y1+2) to (\x1+4.25,\y1+1.75);

\node at (\x1+5.5,\y1+1) {$=$};
	
\draw[dashed] (\x1+6,\y1+0) --(\x1+8,\y1+0) --(\x1+8,\y1+2) --(\x1+6,\y1+2) --(\x1+6,\y1+0);
\draw[line width=0.35mm] (\x1+6.75,\y1+0.25) to (\x1+7.25,\y1+0.25);
\draw[line width=0.35mm] (\x1+6.75,\y1+1.75) to (\x1+7.25,\y1+1.75);
\filldraw[red, line width=0.35mm, fill = black, fill opacity = 0.2] (\x1+6.75,\y1+0.25) to (\x1+6.75,\y1+0.5) to[out=90,in=180] (\x1+7,\y1+0.75) to[out=0,in=90] (\x1+7.25,\y1+0.5) to (\x1+7.25,\y1+0.25);
\filldraw[red, line width=0.35mm, fill = black, fill opacity = 0.2] (\x1+6.75,\y1+1.75) to (\x1+6.75,\y1+1.5) to[out=270,in=180] (\x1+7,\y1+1.25) to[out=0,in=270] (\x1+7.25,\y1+1.5) to (\x1+7.25,\y1+1.75);
\draw[red, line width=0.35mm] (\x1+6.75,\y1+0) to (\x1+6.75,\y1+0.25);
\draw[red, line width=0.35mm] (\x1+7.25,\y1+0) to (\x1+7.25,\y1+0.25);
\draw[red, line width=0.35mm] (\x1+6.75,\y1+1.75) to (\x1+6.75,\y1+2);
\draw[red, line width=0.35mm] (\x1+7.25,\y1+1.75) to (\x1+7.25,\y1+2);
\draw[red, line width=0.35mm] (\x1+7,\y1+0.75) to (\x1+7,\y1+1.25);

\node at (\x1+8.5,\y1+1) {$=$};

\draw[dashed] (\x1+9,\y1+0) --(\x1+11,\y1+0) --(\x1+11,\y1+2) --(\x1+9,\y1+2) --(\x1+9,\y1+0);
\draw[line width=0.35mm] (\x1+10.5,\y1+1.75) to (\x1+9.75,\y1+1.75);
\draw[line width=0.35mm] (\x1+10.25,\y1+0.25) to (\x1+9.5,\y1+0.25);
\filldraw[red, line width=0.35mm, fill = black, fill opacity = 0.2] (\x1+10.5,\y1+1.75) to (\x1+10.5,\y1+0.75) to[out=270,in=0] (\x1+10.25,\y1+0.5) to[out=180,in=270] (\x1+10,\y1+0.75) to (\x1+10,\y1+1.25) to[out=90,in=0] (\x1+9.75,\y1+1.5) to (\x1+9.75,\y1+1.75);
\filldraw[red, line width=0.35mm, fill = black, fill opacity = 0.2] (\x1+10.25,\y1+0.25) to (\x1+10.25,\y1+0.5) to[out=180,in=270] (\x1+10,\y1+0.75) to (\x1+10,\y1+1.25) to[out=90,in=0] (\x1+9.75,\y1+1.5) to[out=180,in=90] (\x1+9.5,\y1+1.25) to (\x1+9.5,\y1+0.25);
\draw[red, line width=0.35mm] (\x1+10.25,\y1+0) to (\x1+10.25,\y1+0.25);
\draw[red, line width=0.35mm] (\x1+9.5,\y1+0) to (\x1+9.5,\y1+0.25);
\draw[red, line width=0.35mm] (\x1+10.5,\y1+2) to (\x1+10.5,\y1+1.75);
\draw[red, line width=0.35mm] (\x1+9.75,\y1+2) to (\x1+9.75,\y1+1.75);

\node at (\x1+11.5,\y1+1) {$=$};

\draw[dashed] (\x1+12,\y1+0) --(\x1+14,\y1+0) --(\x1+14,\y1+2) --(\x1+12,\y1+2) --(\x1+12,\y1+0);
\draw[line width=0.35mm] (\x1+13.5,\y1+1.25) to (\x1+13,\y1+1.25);
\draw[line width=0.35mm] (\x1+13,\y1+0.75) to (\x1+12.5,\y1+0.75);
\filldraw[red, line width=0.35mm, fill = black, fill opacity = 0.2] (\x1+13.5,\y1+1.25) to (\x1+13.5,\y1+0.75) to[out=270,in=0] (\x1+13.25,\y1+0.5) to[out=180,in=270] (\x1+13,\y1+0.75) to (\x1+13,\y1+1.25);
\filldraw[red, line width=0.35mm, fill = black, fill opacity = 0.2] (\x1+13,\y1+0.75) to (\x1+13,\y1+1.25) to[out=90,in=0] (\x1+12.75,\y1+1.5) to[out=180,in=90] (\x1+12.5,\y1+1.25) to (\x1+12.5,\y1+0.75);
\draw[red, line width=0.35mm] (\x1+13.5,\y1+2) to (\x1+13.5,\y1+1.25);
\draw[red, line width=0.35mm] (\x1+13.25,\y1+0) to (\x1+13.25,\y1+0.5);
\draw[red, line width=0.35mm] (\x1+12.75,\y1+2) to (\x1+12.75,\y1+1.5);
\draw[red, line width=0.35mm] (\x1+12.5,\y1+0) to (\x1+12.5,\y1+0.75);

\draw[dashed] (\x2+0,\y2+0) --(\x2+2,\y2+0) --(\x2+2,\y2+2) --(\x2+0,\y2+2) --(\x2+0,\y2+0);
\draw[line width=0.35mm] (\x2+0.75,\y2+0.75) to (\x2+1.25,\y2+0.75);
\draw[line width=0.35mm] (\x2+0.75,\y2+1.25) to (\x2+1.25,\y2+1.25);
\filldraw[red, line width=0.35mm, fill = black, fill opacity = 0.2] (\x2+0.75,\y2+0.75) to (\x2+0.75,\y2+0.5) to[out=270,in=180] (\x2+1,\y2+0.25) to[out=0,in=270] (\x2+1.25,\y2+0.5) to (\x2+1.25,\y2+0.75);
\filldraw[red, line width=0.35mm, fill = black, fill opacity = 0.2] (\x2+0.75,\y2+1.25) to (\x2+0.75,\y2+1.5) to[out=90,in=180] (\x2+1,\y2+1.75) to[out=0,in=90] (\x2+1.25,\y2+1.5) to (\x2+1.25,\y2+1.25);
\draw[red, line width=0.35mm] (\x2+1,\y2+0) to (\x2+1,\y2+0.25);
\draw[red, line width=0.35mm] (\x2+0.75,\y2+0.75) to (\x2+0.75,\y2+1.25);
\draw[red, line width=0.35mm] (\x2+1.25,\y2+0.75) to (\x2+1.25,\y2+1.25);
\draw[red, line width=0.35mm] (\x2+1,\y2+1.75) to (\x2+1,\y2+2);

\node at (\x2+2.5,\y2+1) {$=$};
	
\draw[dashed] (\x2+3,\y2+0) --(\x2+5,\y2+0) --(\x2+5,\y2+2) --(\x2+3,\y2+2) --(\x2+3,\y2+0);
\filldraw[red, line width=0.35mm, fill = black, fill opacity = 0.2] (\x2+3.75,\y2+0.5) to[out=270,in=180] (\x2+4,\y2+0.25) to[out=0,in=270] (\x2+4.25,\y2+0.5) to (\x2+4.25,\y2+1.5) to[out=90,in=0] (\x2+4,\y2+1.75) to[out=180,in=90] (\x2+3.75,\y2+1.5) to (\x2+3.75,\y2+0.5);
\draw[red, line width=0.35mm] (\x2+4,\y2+0) to (\x2+4,\y2+0.25);
\draw[red, line width=0.35mm] (\x2+4,\y2+1.75) to (\x2+4,\y2+2);

\node at (\x2+5.5,\y2+1) {$=$};
	
\draw[dashed] (\x2+6,\y2+0) --(\x2+8,\y2+0) --(\x2+8,\y2+2) --(\x2+6,\y2+2) --(\x2+6,\y2+0);
\draw[red, line width=0.35mm] (\x2+7,\y2+0) to (\x2+7,\y2+2);
\end{tikzpicture}
\end{center}

Therefore $(A,e,m_e,\Delta_e)$ is an object in $\Kar{\bB}$ and furthermore we can now define morphisms $(e,\lt,\rho,m_e,\Delta_e):A_p\to A_e$ and $(e,m_e\Delta_e,\rt,\lambda):A_e\to A_p$. Composing them in $\Kar{\bB}$, we get $e\otimes_e e\iso e:A_p\to A_p$ and $e\otimes_p e:A_e\to A_e$. As we have morphisms $m_e^\otimes: e\otimes_p e \to e=\id_{A_e}$ and $\Delta_e^\otimes: \id_{A_e}=e\to e\otimes_p e$, we have shown that $(e,\lt,\rho,\rt,\lambda)$ splits in $\Kar{\bB}$ and thus $\Kar{\bB}$ is 2-idempotent complete.
\end{proof}

Since any 2-idempotent $(A,p,m,\Delta)$ in $\bB$ defines a 2-idempotent $(A_{\id},p,m,\Delta)$ in $\Kar{\bB}$, we also get the following statement.

\begin{remark} \label{x3}
For any object $A_p$ in $\Kar{\bB}$, $A_p$ is a splitting of the 2-idempotent $p$ on $A_{\id}$.
\end{remark}

We will now show that $\Kar{\bB}$ is a completion of $\bB$ in the sense that $\bB$ embeds into $\Kar{\bB}$ and is equivalent to it if $\bB$ was already 2-idempotent complete.

\begin{proposition} \label{kar2emb}
For every locally idempotent complete bicategory $\bB$, there exists a fully faithful pseudofunctor $\iota_\bB:\bB\to\Kar{\bB}$. If $\bB$ is furthermore 2-idempotent complete, this functor is an equivalence.
\end{proposition}

\begin{proof}
We define $\iota_\bB$ to map the object $A$ in $\bB$ onto $(A,\id_A,\id_{\id_A},\id_{\id_A})$ in $\Kar{\bB}$. It maps 1-morphisms $f:A\to B$ in $\bB$ onto the the morphism $(f,\id_f,\id_f,\id_f,\id_f):A_{\id}\to B_{\id}$ and 2-morphisms $\phi:f\to g$ in $\bB$ onto $\phi:f\to g$ in $\Kar{\bB}$.

We now note that, for objects $A$, $B$ in $\bB$, every morphism in $\Kar{\bB}$ between $A_{\id}$ and $B_{\id}$ is of the form $(f,\id_f,\id_f,\id_f,\id_f)$ where $f$ is an arbitrary morphism $f:A\to B$ in $\bB$. Since a 2-morphism between 1-morphisms of the form  $(f,\id_f,\id_f,\id_f,\id_f)$ and $(g,\id_g,\id_g,\id_g,\id_g)$ is also just an arbitrary 2-morphism $\varphi:f\to g$, we see that $\iota_\bB$ is fully faithful.

Now assume that $\bB$ is 2-idempotent complete and let $(A,p,m,\Delta)$ be an object in $\Kar{\bB}$. We want to show that $\iota_\bB$ is essentially surjective, i.e., there exists some $B$ in $\bB$ such that $A_p\equiv \iota_\bB B$. We know that the 2-idempotent $(A,p,m,\Delta)$ splits in $\bB$, which means there exists an object $B$ in $\bB$, 1-morphisms $f:A\to B$ and $g:B\to A$, 2-morphisms $\varphi:f\circ g\to \id_B$, $\psi:\id_B \to f\circ g$ and an isomorphism $\gamma:g\circ f\to p$ such that $m=\gamma\cdot(\id_g \circ \varphi \circ \id_f)\cdot(\gamma^{-1}\circ\gamma^{-1})$ and $\Delta=(\gamma\circ\gamma)\cdot(\id_g \circ \psi \circ \id_f)\cdot\gamma^{-1}$.

We can now define the morphisms
\begin{align*}
& (f,(\phi\circ\id_f)\cdot(\id_f\circ \gamma^{-1}),(\id_f\circ\gamma)\cdot(\psi\circ\id_f),\id_f,\id_f):A_p\to B_{\id} \text{ and} \\
& (g,\id_g,\id_g,(\id_g\circ\phi)\cdot(\gamma^{-1}\circ\id_g),(\gamma\circ\id_g)\cdot(\id_g\circ\psi)):B_{\id}\to A_p
\end{align*}
and, with these, we now have $g\otimes_{\id_B} f\iso g\circ f\iso p=\id_{A_p}$ and $f\otimes_p g\iso\id_B$. Thus it follows that $A_p\equiv B_{\id}$ and $\iota_\bB:\bB\to\Kar{\bB}$ is an equivalence of bicategories.
\end{proof}

Lastly, we want to show that the idempotent completion $\Kar{\bB}$ is universal among all possible idempotent completions of $\bB$, which is why we are able to call it \textit{the} idempotent completion of $\bB$. What we mean by universal is that for each pseudofunctor $F$ from $\bB$ into an arbitrary 2-idempotent complete bicategory $\bC$, there is a pseudofunctor $F^\prime:\Kar{\bB}\to\bC$ such that $F^\prime\iota_\bB\cong F$.

\begin{theorem} \label{lscol}
For each locally idempotent complete bicategory $\bB$ and 2-idempotent complete bicategory $\bC$, we have an equivalence of bicategories
\begin{align*}
\Bicat(\Kar{\bB},\bC)\simeq\Bicat(\bB,\bC)
\end{align*}
induced by precomposing with $\iota_\bB$.
\end{theorem}

\begin{proof}
First, we want to show that for every pseudofunctor $F:\bB\to\bC$, there exists a pseudofunctor $F\p:\Kar{\bB}\to\bC$ such that $F\p\iota_\bB\equiv F$, i.e., precomposition with $\iota_\bB$ is essentially surjective. For a given pseudofunctor $F:\bB\to \bC$ we can define the pseu\-do\-functor $\Kar{F}:\Kar{\bB}\to\Kar{\bC}$ which maps a 2-idempotent $(A,p,m,\Delta)$ onto the 2-idempotent $(F(A),F(p),F(m),F(\Delta))$ and acts analogously on 1-morphisms and 2-morphisms. The following diagram now commutes up to equivalence
\begin{center}
\begin{tikzcd}
\bB \arrow[r, "F"] \arrow[d, "\iota_\bB"] & \bC \arrow[d, "\iota_\bC"] \\
\Kar{\bB} \arrow[r, "\Kar{F}"]            & \Kar{\bC}
\end{tikzcd}
\end{center}
Since $\bC$ is a 2-idempotent complete bicategory, $\iota_\bC$ is an equivalence and we can choose an inverse $\iota_\bC\inv$. We now define $F\p=\iota_\bC\inv\Kar{F}$ and have 
\begin{align*}
F\p\iota_\bB=\iota_\bC\inv\Kar{F}\iota_\bB\equiv \iota_\bC\inv\iota_\bC F\equiv F.
\end{align*}

Secondly, we need to show that precomposition with $\iota_\bB$ is essentially surjective on Hom-categories. Let $F,G:\Kar{\bB}\to\bC$ be pseudofunctors. We want to show that for each strong transformation $\phi:F\iota_\bB\to G\iota_\bB$, we can find a strong transformation $\phi\p:F\to G$ such that $\phi\p\iota_\bB\iso \phi$. For a given $\phi:F\iota_\bB\to G\iota_\bB$, we define such a $\phi\p$ by defining an idempotent and then setting $\phi\p_{A_p}$ to be its splitting. This idempotent $\gamma_{A_p}:G(p)\phi_A F(p)\to G(p)\phi_A F(p)$ is given via
\begin{center}
\begin{tikzpicture}
\node at (0,0) {$F(A_{\id})$};
\node at (4,0) {$F(A_{\id})$};
\node at (2,2) {$F(A_p)$};
\node at (0,-3) {$G(A_{\id})$};
\node at (4,-3) {$G(A_{\id})$};
\node at (2,-5) {$G(A_p)$};
\node at (-0.2,1.5) {$F(p)$};
\node at (4.2,1.5) {$F(p)$};
\node at (2,-0.3) {$F(p)$};
\node at (-0.3,-1.5) {$\phi_A$};
\node at (4.3,-1.5) {$\phi_A$};
\node at (2,-2.7) {$G(p)$};
\node at (-0.2,-4.5) {$G(p)$};
\node at (4.2,-4.5) {$G(p)$};
\node at (2.6,0.65) {$F(m_p)$};
\node at (1.7,-1.3) {$\phi_p$};
\node at (1.4,-3.75) {$G(\Delta_p)$};
\draw[->, out=180, in=90] (1.25,2) to (0,0.25);
\draw[->, out=0, in=90] (2.75,2) to (4,0.25);
\draw[->] (0.75,0) to (3.25,0);
\draw[->] (0,-0.25) to (0,-2.75);
\draw[->] (4,-0.25) to (4,-2.75);
\draw[->] (0.75,-3) to (3.25,-3);
\draw[->, out=-90, in=180] (0,-3.25) to (1.25,-5);
\draw[->, out=-90, in=0] (4,-3.25) to (2.75,-5);
\draw[implies-, double equal sign distance] (2.5,1.25) to (1.5,0.75);
\draw[implies-, double equal sign distance] (2.5,-1) to (1.5,-2);
\draw[implies-, double equal sign distance] (2.5,-3.75) to (1.5,-4.25);
\end{tikzpicture}
\end{center}
and we now have a splitting consisting of $\phi\p_{A_p}:F(A_p)\to G(A_p)$ along with morphisms $\alpha_{A_p}:G(p)\phi_A F(p)\to \phi\p_{A_p}$ and $\beta_{A_p}:\phi\p_{A_p}\to G(p)\phi_A F(p)$ which satisfy $\alpha_{A_p} \beta_{A_p} =\id_{\phi\p_{A_p}}$ and $\beta_{A_p} \alpha_{A_p}=\gamma_{A_p}$. We define $\phi\p_f$ on a morphism $f:A_p\to B_q$ via
\begin{center}
\begin{tikzpicture}
\node at (0,0) {$F(A_p)$};
\node at (3,0) {$F(B_q)$};
\node at (0,-2) {$F(A_{\id})$};
\node at (3,-2) {$F(B_{\id})$};
\node at (0,-4) {$G(A_{\id})$};
\node at (3,-4) {$G(B_{\id})$};
\node at (0,-6) {$G(A_p)$};
\node at (3,-6) {$G(B_q)$};
\node at (1.5,0.3) {$F(f)$};
\node at (1.5,-1.7) {$F(f)$};
\node at (1.5,-3.7) {$G(f)$};
\node at (1.5,-5.7) {$G(f)$};
\node at (0.5,-1) {$F(p)$};
\node at (0.4,-3) {$\phi_A$};
\node at (0.5,-5) {$G(p)$};
\node at (2.5,-1) {$F(q)$};
\node at (2.6,-3) {$\phi_B$};
\node at (2.5,-5) {$G(q)$};
\node at (-1.7,-3) {$\phi\p_{A_p}$};
\node at (4.7,-3) {$\phi\p_{B_q}$};
\node at (-0.6,-3) {$\beta_{A_p}$};
\node at (3.6,-3) {$\alpha_{B_q}$};
\node at (1.3,-2.7) {$\phi_f$};
\draw[->] (0.6,0) to (2.4,0);
\draw[->] (0.75,-2) to (2.25,-2);
\draw[->] (0.75,-4) to (2.25,-4);
\draw[->] (0.6,-6) to (2.4,-6);
\draw[->] (0,-0.25) to (0,-1.75);
\draw[->] (3,-0.25) to (3,-1.75);
\draw[->] (0,-2.25) to (0,-3.75);
\draw[->] (3,-2.25) to (3,-3.75);
\draw[->] (0,-4.25) to (0,-5.75);
\draw[->] (3,-4.25) to (3,-5.75);
\draw[->, out=-120, in =120] (-0.5,-0.25) to (-0.5,-5.75);
\draw[->, out=-60, in =60] (3.5,-0.25) to (3.5,-5.75);
\draw[-implies, double equal sign distance] (-1,-3.4) to (-0.3,-3.4);
\draw[-implies, double equal sign distance] (3.3,-3.4) to (4,-3.4);
\draw[-implies, double equal sign distance] (1,-1.25) to (2,-0.75);
\draw[-implies, double equal sign distance] (1,-3.25) to (2,-2.75);
\draw[-implies, double equal sign distance] (1,-5.25) to (2,-4.75);
\end{tikzpicture}
\end{center}
where the top morphism is given by the image of
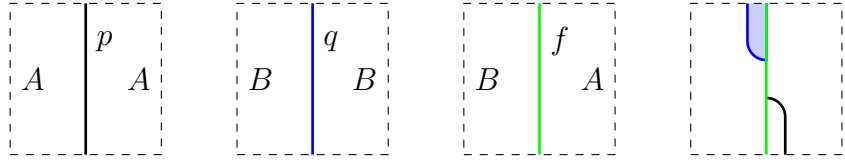
\begin{center}
\begin{tikzpicture}
\draw[dashed] (0,0) --(2,0) --(2,2) --(0,2) --(0,0);
\draw[line width=0.35mm] (1,0) to (1,2);
\node[right] at (1,1.5) {$p$};
\node[right] at (0,1) {$A$};
\node[left] at (2,1) {$A$};
	
\draw[dashed] (3,0) --(5,0) --(5,2) --(3,2) --(3,0);
\draw[blue, line width=0.35mm] (4,0) to (4,2);
\node[right] at (4,1.5) {$q$};
\node[right] at (3,1) {$B$};
\node[left] at (5,1) {$B$};
	
\draw[dashed] (6,0) --(8,0) --(8,2) --(6,2) --(6,0);
\draw[green, line width=0.35mm] (7,0) to (7,2);
\node[right] at (7,1.5) {$f$};
\node[right] at (6,1) {$B$};
\node[left] at (8,1) {$A$};

\draw[dashed] (9,0) --(11,0) --(11,2) --(9,2) --(9,0);
\filldraw[blue, line width=0.35mm, fill = blue, fill opacity = 0.2] (9.75,2) to (9.75,1.5) to[out=270,in=180] (10,1.25) to (10,2);
\draw[line width=0.35mm] (10,0.75) to[out=0,in=90] (10.25,0.5) to (10.25,0);
\draw[green, line width=0.35mm] (10,0) to (10,2);
\end{tikzpicture}
\end{center}
under $F$ and the bottom morphism is given by the image of
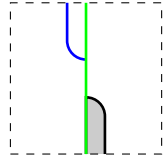
\begin{center}
\begin{tikzpicture}
\draw[dashed] (0,0) --(2,0) --(2,2) --(0,2) --(0,0);
\filldraw[black, line width=0.35mm, fill = black, fill opacity = 0.2] (1,0) to (1,0.75) to[out=0,in=90] (1.25,0.5) to (1.25,0);
\draw[blue, line width=0.35mm] (0.75,2) to (0.75,1.5) to[out=270,in=180] (1,1.25);
\draw[green, line width=0.35mm] (1,0) to (1,2);
\end{tikzpicture}
\end{center}
under $G$. This turns $\phi\p:G\to F$ into a strong transformation. We now need to show that $\phi\p\iota_\bB\iso\phi$, i.e., $\phi\p_{A_{\id}}\iso \phi_A$. If we look at the idempotent $\gamma_{A_{\id}}$ we see that it splits via $\phi_A$ and thus $\phi\p_{A_{\id}}\iso \phi_A$.

Lastly, we need to show that precomposition with $\iota_\bB$ is fully faithful on Hom-categories. Let $F,G:\Kar{\bB}\to\bC$ be pseudofunctors and $\phi,\psi:F\to G$ strong transformations. We want to show that we have an isomorphism
\begin{align*}
\Bicat(\Kar{\bB},\bC)(F,G)(\phi,\psi)\iso \Bicat(\bB,\bC)(F\iota_\bB,G\iota_\bB)(\phi\iota_\bB,\psi\iota_\bB).
\end{align*}
Let $\Gamma:\phi\iota_\bB\to \psi\iota_\bB$ be a modification. We can construct a modification $\Gamma\p:\phi\to\psi$ with components $\Gamma\p_{A_p}:\phi_{A_p}\to\psi_{A_p}$ via the following.
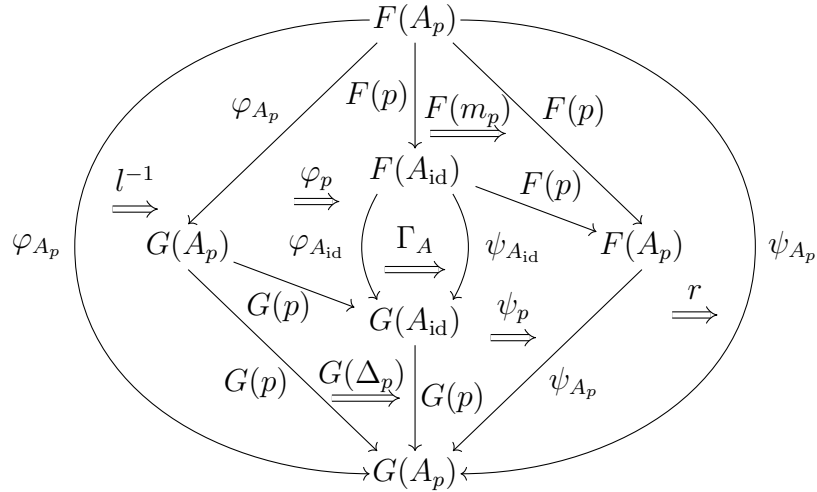
\begin{center}
\begin{tikzpicture}
\node at (0,0) {$F(A_p)$};
\node at (0,-2) {$F(A_{\id})$};
\node at (0,-4) {$G(A_{\id})$};
\node at (0,-6) {$G(A_p)$};
\node at (-3,-3) {$G(A_p)$};
\node at (3,-3) {$F(A_p)$};
\node at (-5,-3) {$\phi_{A_p}$};
\node at (5,-3) {$\psi_{A_p}$};
\node at (-2.1,-1.2) {$\phi_{A_p}$};
\node at (-2.1,-4.8) {$G(p)$};
\node at (2.1,-1.2) {$F(p)$};
\node at (2.1,-4.8) {$\psi_{A_p}$};
\node at (-0.5,-1) {$F(p)$};
\node at (0.5,-5) {$G(p)$};
\node at (1.8,-2.2) {$F(p)$};
\node at (-1.8,-3.8) {$G(p)$};
\node at (-1.3,-3) {$\phi_{A_{\id}}$};
\node at (1.3,-3) {$\psi_{A_{\id}}$};
\node at (0,-2.9) {$\Gamma_A$};
\node at (-1.3,-2.1) {$\phi_p$};
\node at (1.3,-3.8) {$\psi_p$};
\node at (-0.7,-4.7) {$G(\Delta_p)$};
\node at (0.7,-1.2) {$F(m_p)$};
\node at (-3.7,-2.1) {$l\inv$};
\node at (3.7,-3.6) {$r$};
\draw[->] (-0.6,0) to [out=180, in=90] (-4.5,-3)
	to [->, out=-90, in=180] (-0.6,-6);
\draw[->] (-0.5,-0.3) to (-3,-2.7);
\draw[->] (-3,-3.3) to (-0.5,-5.7);
\draw[->] (-2.4,-3.2) to (-0.8,-3.8);
\draw[->] (0,-0.3) to (0,-1.7);
\draw[->, out=-120, in=120] (-0.5,-2.3) to (-0.5,-3.7);
\draw[->, out=-60, in=60] (0.5,-2.3) to (0.5,-3.7);
\draw[->] (0,-4.3) to (0,-5.7);
\draw[->] (0.8,-2.2) to (2.4,-2.8);
\draw[->] (0.5,-0.3) to (3,-2.7);
\draw[->] (3,-3.3) to (0.5,-5.7);
\draw[->] (0.6,0) to [out=0,in=90] (4.5,-3)
	to [out=-90,in=0] (0.6,-6);
\draw[-implies, double equal sign distance] (-0.4,-3.3) to (0.4,-3.3);
\draw[-implies, double equal sign distance] (-1.6,-2.4) to (-1,-2.4);
\draw[-implies, double equal sign distance] (1,-4.2) to (1.6,-4.2);
\draw[-implies, double equal sign distance] (-1.1,-5) to (-0.2,-5);
\draw[-implies, double equal sign distance] (0.2,-1.5) to (1.2,-1.5);
\draw[-implies, double equal sign distance] (-4,-2.5) to (-3.4,-2.5);
\draw[-implies, double equal sign distance] (3.4,-3.9) to (4,-3.9);
\end{tikzpicture}
\end{center}
This defines a modification $\Gamma\p:F\to G$ and we can see that this construction defines an inverse on the level of modifications. Thus precomposition with $\iota_\bB$ is fully faithful on Hom-categories and we have an equivalence
\begin{align*}
\Bicat(\Kar{\bB},\bC)\equiv \Bicat(\bB,\bC).
\end{align*} 
\end{proof}

\begin{remark}
Without going into technical details, theorem \ref{lscol} is a consequence of the fact that idempotent completion should form a left adjoint trifunctor $\Kartwo:\Bicatlic\to\Bicatic$ from the tricategory of locally idempotent complete bicategories, pseudofunctors, strong transformations and modifications into the full subtricategory of 2-idempotent complete bicategories, analogously to theorem \ref{1adj}.
\end{remark}

\section{Cauchy Completion} \label{sec5}

We are now going to define the Cauchy completion of a bicategory, which is its completion with respect to absolute weighted colimits. We will see, that it is universal among all bicategories that have this property. Before we get to this, we will define a bicategorical version of weighted colimits and show that a 2-idempotent splitting in a locally idempotent complete bicategory is a weighted colimit. Since 2-idempotent splittings are preserved by every pseudofunctor, they in fact define absolute weighted colimits.

\subsection{Weighted Colimits in Bicategories}

In a general enriched context, weighted colimits were already described in \cite{Kel2}. Since one can think of bicategories as categories \textit{weakly} enriched in $\Cat$, there are some choices one can make when defining weighted colimits.

\begin{definition}{(Weighted Colimit)}
Let $\bJ$ and $\bB$ be bicategories. Given a presheaf $W:\bJ\op\to\Cat$, which we will call \textit{weight}, and another pseudofunctor $F:\bJ\to\bB$, the \textit{colimit of $F$ weighted by $W$} is given in the following way: We can form the pseudofunctor
\begin{align*}
\Psh{\bJ}(W,\bB(F,-)):\bB\to\Cat.
\end{align*}
If this pseudofunctor is representable, i.e., there exists an object $\colim^W F$ and an equivalence of pseudofunctors $\phi:\bB(\colim^W F,-)\to\Psh{\bJ}(W,\bB(F,-))$, we call $\colim^W F$ together with $\phi$ the \textit{colimit of $F$ weighted by $W$} or just the \textit{weighted colimit of $F$} and we call $\phi$ a \textit{universal} strong transformation.

By the Yoneda Lemma, the data of this strong transformation is equivalent to the object $\phi_{\colim^W F}(\id_{\colim^W F})$ in $\Psh{\bJ}(W,\bB(F-,\colim^W F))$, i.e., a strong transformation $\lambda:W\to \bB(F-,\colim^W F)$ such that precomposition with $\lambda$ defines an equivalence
\begin{align*}
\lambda^*:\bB(\colim^W F,-)\to\Psh{\bJ}(W,\bB(F,-)).
\end{align*}
\end{definition}

\begin{definition}{(Cocomplete Bicategory)}
A bicategory $\bB$ is called \textit{cocomplete} if for every weight $W$ and every pseudofunctor $F$ into $\bB$ the weighted colimit of $F$ by $W$ exists.
\end{definition}

Since we want to study absolute weighted colimits, we need a precise definition of what it means for a weighted  colimit to be preserved by a pseudofunctor.

\begin{definition}{(Preservation of Weighted Colimits)}
Let $\bC$ be another bicategory and $G:\bB\to\bC$ a pseudofunctor. We say $G$ \textit{preserves} a colimit if precomposition with $G\lambda:W\to \bC(GF-,G\colim^W F)$ defines an equivalence
\begin{align*}
(G\lambda)^*:\bC(G\colim^W F,-)\to\Psh{\bJ}(W,\bC(GF,-)).
\end{align*}
\end{definition}

\begin{definition}{(Absolute Weighted Colimit)}
A weighted colimit is called \textit{absolute} if it is preserved by every pseudofunctor.
\end{definition}

\begin{definition}{(Absolute Weight)}
A weight $W:\bJ\op\to\Cat$ is called \textit{absolute} if for every bicategory $\bB$ and every pseudofunctor $F:\bJ\to\bB$, the colimit of $F$ weighted by $W$ is absolute if it exists.
\end{definition}

The theory of weighted colimits still holds true in the context of categories weakly enriched over categories satisfying a property $P$.

\begin{definition}{(Weighted Colimits in Locally $P$ Bicategories)} \label{lPcolim}
Let $\bB$ be a locally $P$ bicategory, $\bJ$ an arbitrary bicategory and $F:\bJ\to\bB$ a pseudofunctor. A \textit{$P$-weight} is a presheaf $W:\bJ\op\to\Cat_P$. The \textit{$P$-weighted colimit} of $F$ weighted by $W$ now consists of a strong transformation $\lambda:W\to\bB(F-,\colim_P^W F)$ such that
\begin{align*}
\lambda^*:\bB(\colim_P^W F,b)\to\PshP{\bJ}(W,\bB(F,b))
\end{align*}
defines an equivalence of categories for all objects $b$ in $\bB$. A locally $P$ bicategory is called \textit{$P$-cocomplete} if every $P$-weighted colimit exists.
\end{definition}

In general every $P$-weighted colimit in a locally $P$ bicategory can be regarded as a weighted colimit in an ordinary bicategory by simply regarding the $P$-weight $W:\bJ\op\to\Cat_P$ as an ordinary weight $W:\bJ\op\to\Cat$. But by restricting ourselves to locally $P$ bicategories and $P$-weights, we will be able to find $P$-weights which are absolute and which would not be absolute as ordinary weights. When it is clear from context, we will simply refer to $P$-weights, $P$-weighted colimits, and $P$-cocompleteness as weights, weighted colimits, and cocompleteness. 

In many ordinary categories we can give explicit formulas for calculating certain colimits. By explicitly calculating a colimit we usually also gain explicit formulas for the universal property of a colimit which can be used to understand how the colimit interacts with other objects. In the case of the 2-category $\Cat$, we also find an explicit formula for a colimit analogous to the one in the 1-category $\Set$. By explicitly constructing a weighted colimit for an arbitrary bicategory $\bJ$, weight $W:\bJ\op\to\Cat$ and pseudofunctor $F:\bJ\to\Cat$, we also show that $\Cat$ is cocomplete.

\begin{theorem}
Let $\bB$ be a bicategory, the 2-categories $\Cat$ and $\Psh{\bB}$ are cocomplete.
\end{theorem}

The explicit construction and detailed proofs can be found in \ref{app1}. Abstractly, this was already proven in \cite{Kel} but by taking the colimit of a two-sided bar construction as outlined in \citep[12.]{S}, we can derive an explicit formula for weighted colimits in $\Cat$. An equivalent formula is also given in \cite{Lam} but derived differently. The cocompleteness of $\Psh{\bB}$ follows by point-wise computing weighted colimits.

\subsection{2-Idempotent Splittings as Absolute Weighted Colimits}

We briefly recall definition \ref{walkin} of the free walking 2-idempotent splitting $\spd_2$. It has two objects $X$ and $Y$ and a 2-idempotent on $X$ which splits via $Y$ given by morphisms $f:X\to Y$, $g:Y\to X$, $\phi:f g\to \id_Y$ and $\psi:\id_Y\to f g$ such that $\phi\psi=\id_{\id_Y}$. The free walking 2-idempotent $\clb_2$ is the full subbicategory of $\spd_2$ on $X$ and has a fully faithful embedding $\iota:\clb_2\to\spd_2$.

\begin{lemma} \label{spd2idcomp}
The free walking 2-idempotent splitting is locally idempotent complete.
\end{lemma}

\begin{proof}
We have to show that all Hom-categories of $\spd_2$ are idempotent complete. Since there are only 4 of them, we can check this by hand. The objects of $\spd_2(Y,Y)$ are generated by $fg:Y\to Y$. The only non-identity idempotent on $fg$ is $\psi\phi$ which splits trivially via $\phi$ and $\psi$. This shows that $\spd_2(Y,Y)$ is idempotent complete.

We have an equivalence of categories $f^*:\spd_2(Y,Y)\to\spd_2(X,Y)$, since all 1-morphisms $h:X\to Y$ can be written as $h\p \circ f$ for some $h\p:Y\to Y$, and all 2-morphisms $\alpha:h\p \circ f\to k\p \circ f$ are of the form $\alpha\p \circ \id_f$, for some $\alpha\p:h\p \to k\p:Y\to Y$. This implies that $\spd_2(X,Y)$ is idempotent complete. We get an analogous equivalence and result for $\spd_2(Y,X)$.

For $\spd_2(X,X)$ we do not get an equivalence but we note that all 1-morphisms $h:X\to X$ except $\id_X$ can be written as $g\circ h\p\circ f$ for some $h\p:Y\to Y$. Since the only 2-morphism on $\id_X$ is the identity, and there are no 2-morphisms between $\id_X$ and any other 1-morphism, idempotent completeness of $\spd_2(Y,Y)$ still implies the idempotent completeness of $\spd_2(X,X)$. Thus $\spd_2$ is locally idempotent complete.
\end{proof}

\begin{proposition}
Let $\bB$ be a locally idempotent complete bicategory and $(A,p,m,\Delta)$ a 2-idempotent which corresponds to a pseudofunctor $F:\clb_2\to\bB$. The colimit of $F$ weighted by $\spd_2(\iota-,Y):\clb_2\op\to\Catic$ determines a splitting of $A_p$ and every splitting determines a colimit of $F$ weighted by $\spd_2(\iota-,Y)$.
\end{proposition}

\begin{proof}
First we note that from this point onward we are interested in weighted colimits in locally idempotent complete bicategories as laid out in definition \ref{lPcolim}. We will shorten "idempotent complete" to ic, whenever appropriate when talking about weighted colimits in locally idempotent complete bicategories.

We note that lemma \ref{spd2idcomp} implies that the weight $\spd_2(\iota-,Y):\clb_2\op\to\Catic$ takes values in idempotent complete categories and defines an ic-weight. Assume that the colimit of $F$ weighted by $\spd_2(\iota-,Y)$ exists. The colimit of $F$ weighted by $\spd_2(\iota-,X)\equiv \clb_2(-,X)$ exists trivially and is equivalent to $FX$ by the Yoneda lemma. Proposition \ref{ext} tells us that the existence of these weighted colimits guarantees the existence of an extension $F\p:\spd_2\to\bB$ of $F$ along $\iota:\clb_2\to \spd_2$ such that $F\p\iota \equiv F$. By remark \ref{ExtIsSplit}, this shows that the colimit of $F$ weighted by $\spd_2(\iota-,Y)$ determines a splitting of $A_p$.

Before we can finish this proof, we briefly examine the bicategory $\Bicat(\clb_2\op,\Catic)$ to better understand colimits weighted by $\spd_2(\iota-,Y)$. For a given monoidal category $\C$, we have an equivalence
\begin{align*}
\Bicat((\B\C)\op,\Cat)\equiv\rAct{\C}
\end{align*}
where $\B\C$ is the delooping of $\C$ and $\rAct{\C}$ is the bicategory of categories with a right $\C$-action,
given by evaluating at the one object of $\B\C$. This is a categorification of the fact that, for a monoid $M$, there is an equivalence between $\Cat(\B M\op, \Set)$ and the category of sets with a right $M$-action. Analogously, its proof consists of correctly de- and reassembling the data of the pseudofunctors in question. Further details on actions of monoidal categories on categories can be found in \cite{CG}.

Since $\clb_2$ is a one-object bicategory, we have an equivalence
\begin{align*}
\Bicat(\clb_2\op,\Catic)\equiv\rAct{\clb_2(X,X)}.
\end{align*}
Thus, a strong transformation $\kappa:\spd_2(\iota -, Y)\to \bB(F-,C):\clb_2\op\to\Catic$ corresponds to a functor $\kappa_X:\spd_2(X, Y)\to \bB(A,C)$ preserving the action of $\clb_2(X,X)$. We can now show that $\kappa$ is uniquely up to unique isomorphism determined by $\kappa_X(f)$ together with a right $p$-module structure on it. The following diagram commutes up to isomorphism.
\begin{center}
\begin{tikzcd}
{\spd_2(X,Y)\times \clb_2(X,X)} \arrow[r, "\kappa_X\times\id"] \arrow[d, "\id\times\iota"] & {\bB(A,C)\times\clb_2(X,X)} \arrow[d, "\id\times F"] \\
{\spd_2(X,Y)\times \spd_2(X,X)} \arrow[d, "\circ"]                                         & {\bB(A,C)\times\bB(A,A)} \arrow[d, "\circ"]          \\
{\spd_2(X,Y)} \arrow[r, "\kappa_X"]                                                        & {\bB(A,C)}                                          
\end{tikzcd}
\end{center}
We therefore have $\kappa_X(f\circ h)\iso\kappa_X(f) \circ F(h)$ and $\kappa_X(\id_f\circ \alpha)\iso \id_{\kappa_X(f)}\circ F(\alpha)$ for $h$ a 1-morphism and $\alpha$ a 2-morphism in $\bB(A,A)$. The 2-morphisms in $\spd_2(X,Y)$ which cannot be written as $\id_f \circ \alpha$ are all generated by $\phi^{\circ n}\circ \id_f$ and $\psi^{\circ n}\circ \id_f$. Though, we note that
\begin{align*}
\kappa_X(\phi^{\circ 2}\circ \id_f)\iso \kappa_X((\phi\circ \id_f)\cdot (\id_f\circ\id_g\circ\phi\circ\id_f))\iso \kappa_X(\phi\circ \id_f)\cdot \left(\id_{\kappa_X(f)}\circ F(\id_g\circ \phi\circ \id_f)\right),
\end{align*}
i.e., all of these are up to isomorphism determined by $\kappa_X(\phi\circ \id_f):f\circ p\to f$ and $\kappa_X(\psi\circ \id_f):f\to f\circ p$ which need to satisfy $\kappa_X(\phi\circ \id_f)\cdot\kappa_X(\psi\circ \id_f)=\id_{\kappa_X(f)}$, i.e., a right $p$-action on $\kappa_X(f)$.

Now let $(A,B,f,g,\phi,\psi,\gamma)$ be a splitting of $A_p$, determining an extension $F\p$ of $F$. We can define a strong transformation
\begin{align*}
\lambda:\spd_2(\iota-,Y)\to \bB(F-,B)
\end{align*}
by defining $\lambda_X:\spd_2(X,Y)\to \bB(A,B)$ to be $F\p$. We now need to show that it is universal, i.e.,
\begin{align*}
\lambda^*:\bB(B,b)\to \Bicat(\clb_2,\Catic)(\spd_2(\iota-,Y),\bB(F-,b))
\end{align*}
defines an equivalence of categories for each object $b$ in $\bB$. Let $\kappa:\spd_2(\iota-,Y)\to \bB(F-,b)$ be another strong transformation which is determined by the right $p$-module $\kappa_X(f):A\to b$. We can now form the morphism $\kappa_X(f)\otimes_p g:B\to b$ and will now show that this defines an inverse to $\lambda^*$. We have
\begin{align*}
(\kappa_X(f)\otimes_p g)\lambda_X(f)=(\kappa_X(f)\otimes_p g)f \iso \kappa_X(f)\otimes_p (gf)\iso \kappa_X(f)\otimes_p p\iso \kappa_X(f)
\end{align*}
and since the composition with $\lambda_X(f)$ is a composite of $\id_B$-modules, this isomorphism also preserves the right $p$-module structure and $(\kappa_X(f)\otimes_p g)\lambda \iso \kappa$. Furthermore, for a morphism $h:B\to b$, we have
\begin{align*}
\left(h\lambda_X(f)\right) \otimes_p g\iso (hf)\otimes_p g \iso h(f\otimes_p g) \iso h\id_B\iso h.
\end{align*}
Thus $\lambda$ is universal.
\end{proof}

\begin{corollary} \label{x4}
Any two splittings of the same 2-idempotent are equivalent in a locally idempotent complete bicategory.
\end{corollary}

\begin{corollary}
$\spd_2(\iota-,Y):\clb_2\op\to\Catic$ is an absolute ic-weight.
\end{corollary}

\begin{corollary}
Let $\bB$ be a bicategory. The bicategories $\Catic$, $\Catic\op$ and $\Pshic{\bB}$ are 2-idempotent complete.
\end{corollary}

\begin{proof}
Corollaries \ref{Catlic} and \ref{Pshlic} tell us that $\Catic$, $\Catic\op$ and $\Pshic{\bB}$ are all locally idempotent complete. By theorem \ref{x2} and corollary \ref{Pshicc} we know that $\Catic$ and $\Pshic{\bB}$ are ic-cocomplete. Since splittings of a 2-idempotents are given by an ic-weighted colimit, they are also 2-idempotent complete.

One can show, that the data of a 2-idempotent in $\Catic\op$ is the same data as that of a 2-idempotent in $\Catic$ and furthermore a splitting of this idempotent in $\Catic\op$ defines a splitting of it in $\Catic$ and vice versa. This implies that $\Catic\op$ is also 2-idempotent complete.
\end{proof}

\subsection{Cauchy Completion of a Bicategory}

Finally, we can define the bicategorical analogue of Cauchy completion.

\begin{definition}{(Cauchy Completion)}
We call an object $A$ in a locally idempotent complete bicategory $\bB$ \textit{atomic} if the pseudofunctor $\bB(A,-):\bB\to\Catic$ is ic-cocontinuous, i.e., preserves all ic-weighted colimits.

Let $\bB$ be a locally idempotent complete bicategory. The \textit{Cauchy completion} of $\bB$ is defined to be the full subbicategory of atomic objects in $\Pshic{\bB}$, i.e., an object in this bicategory is a pseudofunctor $S:\bB\op\to\Catic$ such that $\Pshic{\bB}(S,-):\Pshic{\bB}\to\Catic$ is ic-cocontinuous. We will denote this bicategory as $\Pshicat{\bB}$.
\end{definition}

We first want to show that this Cauchy completion is complete under absolute ic-weighted colimits, then that any locally idempotent complete bicategory $\bB$ embeds into its Cauchy completion and that this embedding is an equivalence if $\bB$ was already complete under absolute ic-weighted colimits and finally that the Cauchy completion is universal under all completions with respect to absolute ic-weighted colimits. For sake of brevity, we shorten "absolute ic-weighted colimit" to "icAWC". To prove these statements, we will make great use of bicategorical coends as written in \cite[7.1]{Lor}. 

\begin{remark}
We will frequently use the formula
\begin{align*}
\colimic^W F \equiv \int^{j:\bJ} Wj \times Fj
\end{align*}
to calculate weighted colimits using coends. This equation is not explicitly written anywhere in \cite{Lor} for bicategorical coends but it follows just the same as its 1-categorical counterpart.
\end{remark}

\begin{proposition} \label{repato2}
Representable presheaves are atomic.
\end{proposition} 

\begin{proof}
Let $b$ be an object in $\bB$, we will now show that $\bB(-,b)$ is atomic. Let $F:\bJ\to\Pshic{\bB}$ be a pseudofunctor and $W:\bJ\op\to\Catic$ a weight with colimit $\colimic^W F$, i.e., we have a universal strong transformation
\begin{align*}
\lambda:W\to \Pshic{\bB}(F-,\colimic^W F).
\end{align*}
and we will now see that $\Pshic{\bB}(\bB(-,b),-)$ preserves this colimit. Applying this functor to $\lambda$ yields a strong transformation
\begin{align*}
\lambda_* :W \to \Catic(\Pshic{\bB}(\bB(-,b),F-),\Pshic{\bB}(\bB(-,b),\colimic^W F)).
\end{align*}
Via the Yoneda lemma, the strong transformation $\lambda_*$ corresponds to the strong transformation $\lambda_b:W\to \Catic(F(-)(b),\colimic^W F(b))$. Since weighted colimits in $\Pshic{\bB}$ are computed point-wise, we know that $\lambda_b$ has to define a weighted colimit which implies that $\lambda_*$ is universal as well. Thus the pseudofunctor $\Pshic{\bB}(\bB(-,b),-)$ preserves colimits and $\bB(b,-)$ is atomic.
\end{proof}

\begin{proposition}\label{absato2}
Let $\bJ$ be a locally idempotent complete bicategory. An ic-weight $W:\bJ\op\to\Catic$ is absolute if and only if is atomic as an object in $\Pshic{\bJ}$.
\end{proposition}

\begin{proof}
To start, let $W:\bJ\op\to\Catic$ be an absolute weight. To show that $W$ is atomic, let $V:\bI\op\to\Catic$ be another weight and $F:\bI\to\Pshic{\bJ}$ a pseudofunctor with colimit $\colimic^V F$. We have
\begin{align*}
\Pshic{\bJ}(W,\colimic^V F) &\equiv \Pshic{\bJ}(\colimic^W \yo_{\bJ}, \colimic^V F) \\
&\equiv \colimic^W \Pshic{\bJ}(\yo_{\bJ}, \colimic^V F) \\
&\equiv \colimic^W \colimic^V \Pshic{\bJ}(\yo_{\bJ},F) \\
&\equiv \colimic^V \colimic^W \Pshic{\bJ}(\yo_{\bJ},F) \\
&\equiv \colimic^V \Pshic{\bJ}(W,F).
\end{align*}
For the first equivalence, we use the density formula, for the second the absoluteness of $W$, for the third we use Proposition \ref{repato2}, for the fourth we use the fact that $\colimic^V:\Bicatlic(\bI,\Catic)\to\Catic$ defines a functor, and lastly the absoluteness of $W$ again. This shows that $\Pshic{\bJ}(W,-):\Pshic{\bJ}\to\Catic$ is ic-cocontinuous and therefore $W$ is atomic.

We note that by moving $\colim^W$ to the outside, it defines a weighted colimit in $\Catic\op$, i.e., a weighted limit in $\Catic$ since the left argument of the Hom-functor is contravariant.

Now let $W$ be atomic. Let $F:\bJ\to\bB$ be a functor. The colimit $\colimic^W F$ satisfies the following:
\begin{align*}
\bB(\colimic^W F, b) &\equiv \Pshic{\bJ}(W,\bB(F-,b)) \equiv \Pshic{\bJ}(W,\colimic^{\bB(F-,b)}\yo) \\
&\equiv \colimic^{\bB(F-,b)}\Pshic{\bJ}(W,\yo),
\end{align*}  
for an arbitrary $b$ in $\bB$, where we used the atomicity of $W$ to pull out the colimit. For brevity, we use $W\co$ to denote $\Pshic{\bJ}(W,\yo):\bJ\to\Catic$. Now let $G:\bB\to\bC$ be an arbitrary functor, and $c$ an object in $\bC$. Using coends, we have
\begin{align*}
\bC(G\colimic^W F, c) &\equiv \int^{b:\bB} \bC(Gb,c)\times \bB(\colimic^W F,b) \\
&\equiv \int^{b:\bB} \bC(Gb,c)\times \int^{j:\bJ} \bB(Fj,b) \times W\co (j) \\
&\equiv \int^{b:\bB} \int^{j:\bJ} \bC(Gb,c)\times \bB(Fj,b) \times W\co (j) \\
&\equiv \int^{j:\bJ} \int^{b:\bB} \bC(Gb,c)\times \bB(Fj,b) \times W\co (j) \\
&\equiv \int^{j:\bJ} \bC(GFj,c) \times W\co (j) \equiv \bC(\colimic^W GF, c)
\end{align*}
where we used the Yoneda lemma as stated in \cite[7.1.2 (7.28)]{Lor} for the first and fifth equivalences. By the Yoneda lemma, we now have $G\colimic^W F\equiv \colimic^W GF$. This shows that colimits weighted by $W$ are preserved by arbitrary functors, thus $W$ is an absolute weight. 
\end{proof}

\begin{proposition} \label{cau2com}
Let $\bB$ be a locally idempotent complete bicategory. Its Cauchy completion $\Pshicat{\bB}$ is complete under icAWCs.
\end{proposition}

\begin{proof}
Let $W:\bJ\op\to\Catic$ be an absolute weight and $F:\bJ\to\Pshicat{\bB}$ a pseudofunctor. A priori, the colimit of $F$ weighted by $W$ does not need to exist in $\Pshicat{\bB}$ but since $\Pshic{\bB}$ is ic-cocomplete, it will exist there. We now need to show that this colimit defines an atomic object. For this we will use Proposition \ref{absato2}, i.e., we will show that this colimit is an absolute weight. Let $G:\bB\to\bC$ be a pseudofunctor such that its colimit weighted by $\colimic^W F$ exists. We have the following equivalence: 
\begin{align*}
\colimic^{\colimic^W F} G &\equiv \int^{b:\bB} \colimic^W Fb \times Gb \equiv \int^{b:\bB} \int^{j:\bJ} Wj \times F(j)(b) \times Gb \\
&\equiv \int^{j:\bJ} Wj \times \int^{b : \bB} F(j)(b) \times Gb \equiv \int^{j:\bJ} Wj \times \colimic^{Fj} G  \\
&\equiv \colimic^W \colimic^F G
\end{align*}
Now let $H:\bC\to\bD$ be an arbitrary functor. Since $W$ is absolute and $F$ takes values in absolute weights, we have
\begin{align*}
H\left(\colimic^W \colimic^F G\right) \equiv  \colimic^W H\left(\colimic^F G\right) \equiv \colimic^W \colimic^F HG.
\end{align*}
Together, this yields
\begin{align*}
H\left(\colimic^{\colimic^W F} G\right) \iso \colimic^{\colimic^W F} HG.
\end{align*}
Therefore, $\colimic^W F$ defines an absolute weight showing that $\Pshicat{\C}$ is complete under icAWCs.
\end{proof}

\begin{proposition} \label{cau2emb}
For every locally idempotent bicategory $\bB$, the Yoneda embedding $\yo_\bB:\bB\to\Pshic{\bB}$ takes values in $\Pshicat{\bB}$ and thus defines a fully faithful pseudofunctor $\yo_\bB:\bB\to\Pshicat{\bB}$. If $\bB$ is furthermore complete under icAWCs, this pseudofunctor is an equivalence.
\end{proposition}

\begin{proof}
Let $b$ be an object in $\bB$.  It follows from proposition \ref{repato2} that $\yo_\bB(b)=\bB(-,b)$ is atomic. Thus $\yo_\bB:\bB\to\Pshicat{\bB}$ defines a fully faithful pseudofunctor.

Now assume that $\bB$ is complete under icAWCs. Let $S$ be an atomic presheaf on $\bB$. By Proposition \ref{absato2}, $S$ is an absolute weight. By our assumption, the colimit of $\id_{\bB}:\bB\to\bB$ weighted by $S:\bB\op\to\Catic$ exists in $\bB$. We now have that
\begin{align*}
\yo_{\bB} \, \colimic^S \id_{\bB} \equiv \colimic^S \yo_{\bB} \equiv S
\end{align*}
where the first equivalence follows by absoluteness and the second by the density formula. Therefore the pseudofunctor $\yo_{\bB}:\bB\to\Pshicat{\bB}$ is essentially surjective and thus an equivalence of bicategories.
\end{proof}

To prove the universal property of the Cauchy completion, we first show the following lemma.

\begin{lemma} \label{EW2}
For any two locally idempotent complete bicategories $\bB$, $\bC$, we have an equivalence of categories
\begin{align*}
\Bicat(\bB,\Pshicat{\bC}) \equiv \Bicat(\Pshicat{\bB},\Pshicat{\bC})
\end{align*}
\end{lemma}

\begin{remark}
At its core, this proposition is a variant of the Eilenberg-Watts theorem named after the near simultaneous results of Eilenberg \cite{Eil} and  Watts \cite{Wat}. A more common version states that there is an equivalence of categories
\begin{align*}
\textup{Prof}(\C,\D) \equiv \textup{Func}^{\textup{coc}}(\Psh{\C},\Psh{\D})
\end{align*}
between the category of profunctors from $\C$ to $\D$, i.e., functors of the form $\D\op\times\C\to\Set$, and the category of cocontinuous functors from $\Psh{\C}$ to $\Psh{\D}$, for categories $\C$ and $\D$.
\end{remark}

\begin{proof}
To begin, we construct the pseudofunctor  
\begin{align*}
\Phi : \Bicat(\bB,\Pshicat{\bC}) \to \Bicat(\Pshicat{\bB},\Pshicat{\bC})
\end{align*}
which maps a pseudofunctor $F:\bB\to\Pshicat{\bC}$ to the pseudofunctor $\Phi_F:\Pshicat{\bB}\to\Pshicat{\bC}$ which is defines as
\begin{align*}
\Phi_F(S)(c)=\int^{b:\bB} Sb \times F(b)(c). 
\end{align*}
It directly follows that this construction is functorial in $F$, $S$ and $c$, i.e., that $\Phi$, $\Phi_F$ and $\Phi_F(S)$ define pseudofunctors but we still need to verify that $\Phi_F(S)$ is an atomic presheaf. We do this via proposition \ref{absato2}, by checking that it is an absolute weight. Let $G:\bC\to\bD$ be an arbitrary pseudofunctor such that the colimit $\colimic^{\Phi_F(S)} G$ exists. We now have
\begin{align*}
\colimic^{\Phi_F(S)} G &\equiv \int^{c:\bC} \Phi_F(S)(c) \times Gc \equiv \int^{c:\bC} \int^{b:\bB} Sb \times F(b)(c) \times Gc \\
&\equiv \int^{b:\bB} Sb \times \int^{c:\bC} F(b)(c) \times Gc \equiv \int^{b:\bB} Sb \times \colimic^{Fb} G \\
&\equiv \colimic^S \colimic^F G. 
\end{align*}
It now follows, that for an arbitrary pseudofunctor $H$, we have
\begin{align*}
H\left(\colimic^{\Phi_F(S)} G\right) &\equiv H\left(\colimic^S \colimic^F G\right) \equiv \colimic^S H\left(\colimic^F G\right) \\
&\equiv \colimic^S \colimic^F HG \equiv \colimic^{\Phi_F(S)} HG
\end{align*}
since $S$ is an absolute weight and $F$ takes values in absolute weights, showing that $\Phi_F(S)$ is an atomic presheaf.

Next, the functor
\begin{align*}
F: \Bicat(\Pshicat{\bB},\Pshat{\bC}) \to \Bicat(\bB,\Pshat{\bC})
\end{align*}
which maps a pseudofunctor $\Phi:\Pshicat{\bB}\to\Pshicat{\bC}$ to a pseudofunctor $F_\Phi:\bB\to\Pshat{\bC}$, is defined by precomposing with $\yo_{\bB}:\bB\to\Pshicat{\bB}$.

We can now show that for any given $F:\bB\to\Pshat{\bC}$ we have $F_{\Phi_F}\equiv F$ and for any given $\Phi:\Pshat{\bB}\to\Pshat{\bC}$ we have $\Phi_{F_{\Phi}}\equiv \Phi$. We will start with the first. Let $b\p$ in $\bB$, and $c$ in $\bC$ be objects, we have
\begin{align*}
F_{\Phi_F}(b\p)(c) &= \Phi_F(\yo_{\bB}(b\p))(c) = \int^{b:\bB} \yo_{\bB}(b\p)(b) \times F(b)(c) \\
&= \int^{b:\bB} \bB(b,b\p) \times F(b)(c) \equiv F(b\p)(c)  
\end{align*}
where in the last step we used the Yoneda lemma as stated in \cite[7.1.2 (7.28)]{Lor}. Now let $S:\bB\op\to\Catic$ be an atomic presheaf and $c$ in $\bC$ an object. We have
\begin{align*}
\Phi_{F_{\Phi}}(S)(c) &= \int^{b:\bB} Sb \times F_{\Phi}(b)(c) = \int^{b:\bB} Sb \times \Phi(\yo_{\bB}(b))(c) \\
&\equiv (\colimic^S \Phi(\yo_{\bB}))(c) \equiv \Phi(\colimic^S\yo_{\bB})(c) \iso \Phi(S)(c)
\end{align*}
where we used that $S$ is an absolute weight.

Since both these equivalences are natural, this proves the desired equivalence 
\begin{align*}
\Bicat(\bB,\Pshicat{\bC}) \equiv \Bicat(\Pshicat{\bB},\Pshicat{\bC}).
\end{align*}
\end{proof}

Finally we can show that the Cauchy completion is universal among all completions with respect to absolute colimits in the sense that every functor $F:\bB\to\bC$ from an arbitrary locally idempotent complete bicategory into a bicategory complete under icAWCs factors through the Yoneda embedding $\yo_{\bB}:\bB\to\Pshicat{\bB}$.

\begin{corollary} \label{cau2equ}
For every locally idempotent bicategory $\bB$ and 2-idempotent complete bicategory $\bC$, there is an equivalence
\begin{align*}
\Bicat(\Pshicat{\bB},\bC)\equiv \Bicat(\bB,\bC)
\end{align*} 
given by precomposing with $\yo_\bB$.
\end{corollary}

\begin{proof}
This immediately follows from lemma \ref{EW2} together with proposition \ref{cau2emb}. 
\end{proof}

\section{Their Equivalence} \label{sec6}

Finally we can show that these two constructions yield equivalent bicategories. To prove the bicategorical analogue of lemma \ref{ret1ato} we first need to prove the following lemmas.

\begin{lemma} \label{splitequiv}
A 2-retract of an equivalences is an equivalence, i.e., given two split 2-idempotents, i.e., pseudofunctors $F,G:\spd_2\to\bB$ and a strong transformation $\alpha:F\to G$, if $\alpha_X$ is an equivalence then so is $\alpha_Y$.
\end{lemma}

\begin{proof}
Let $\alpha_X^{-1}$ be an inverse to $\alpha_X$ up to isomorphism. The inverse of $\alpha_Y$ will be defined as $F(f)\alpha_X^{-1}\otimes_{G(p)}G(g)$. For this, we need to define a right $G(p)$-module structure on $F(f)\alpha_X^{-1}$ where we use $p$ to denote $gf:X\to X$ in $\spd_2$.

The isomorphism $\alpha_p:G(p)\alpha_X\to\alpha_X F(p)$ yields an isomorphism $\alpha_X^{-1}G(p)\iso F(p)\alpha_X^{-1}$. With this, we now have a right $G(p)$-action
\begin{align*}
F(f)\alpha_X^{-1}G(p)\iso F(f)F(p)\alpha_X^{-1}\iso F(fp)\alpha_X^{-1}\to F(f)\alpha_X^{-1}
\end{align*}
where the last arrow is given by $F(\phi\circ\id_f)\circ\id_{\alpha_X^{-1}}$. Analogously, we can define a right $G(p)$-coaction. It follows from the naturality of $\alpha$ that this defines a right $G(p)$-module structure.

We can now define $F(f)\alpha_X^{-1}\otimes_{G(p)}G(g)$ and show that it is inverse to $\alpha_Y$. We have
\begin{align*}
\alpha_Y F(f)\alpha_X^{-1}\otimes_{G(p)}G(g) \iso G(f)\alpha_X \alpha_X^{-1}\otimes_{G(p)}G(g) & \iso G(f)\otimes_{G(p)}G(g)\iso\id_{GY} \text{  and} \\ 
F(f)\alpha_X^{-1}\otimes_{G(p)}G(g)\alpha_Y \iso F(f)\alpha_X^{-1}\otimes_{G(p)}\alpha_X F(g) & \iso F(f)\alpha_X^{-1}\alpha_X\otimes_{F(p)}F(g) \\
& \iso F(f)\otimes_{F(p)}F(g)\iso \id_{FY}
\end{align*}
where we used the fact that $f\otimes_p g\iso \id_Y$ at the end of both calculations as well as the fact that $\alpha_X$ freely translates between $F(p)$- and $G(p)$-module structures. This proves that $\alpha_Y$ is an equivalence.
\end{proof}

\begin{lemma} \label{univnat}
Let $F:\bJ\to\bB$ be a pseudofunctor and $W:\bJ\op\to\Catic$ a weight such that $\colim^W F$ exists. We then have a canonical strong transformation
\begin{align*}
u:\colim^W \bB(-,F) \to \bB(-,\colim^W F):\bB\op\to\Catic.
\end{align*}
\end{lemma}

\begin{proof}
First, we note that $\colim^W \bB(-,F)$ exists since $\Pshic{\bB}$ is cocomplete.
 
On objects we can define $u$ in the following way. For any $b$ in $\bB$, we can define a strong transformation $W\to\Catic(\bB(b,F),\bB(b,\colim^W F))$ given by postcomposing with the universal strong transformation $\lambda:W\to \bB(F,\colim^W F)$. This yields a functor $u_b: \colim^W \bB(b,F) \to \bB(b,\colim^W F)$ by universality. These functors assemble into the desired strong transformation $u$ with its component 2-isomorphisms given by universality.
\end{proof}

\begin{corollary} \label{x7}
A 2-retract of an atomic object is atomic.
\end{corollary}

\begin{proof}
Let $A$ be an atomic object in $\bB$ and $B$ a retract of $A$, i.e., we have a pseu\-do\-functor $F:\spd_2\to\bB$ such that $FX\equiv A$ and $FY \equiv B$. It follows from its definition that $\spd_2$ is equivalent to its opposite $\spd_2\op$. This implies that we can think of $F$ as a pseudofunctor $\spd_2\to\bB\op$.

Given a weight $W:\bJ\to\Catic$ and pseudofunctor $G:\bJ\to\bB$ such that $\colim^W F$ exists, by lemma \ref{univnat}, we have a strong transformation $u:\colim^W \bB(-,G) \to \bB(-,\colim^W G):\bB\op\to\Catic$. By precomposing with $F$ we then get a strong transformation $u_F:\colim^W \bB(F-,G) \to \bB(F-,\colim^W G):\spd_2\to\Catic$.

Since $A$ is assumed to be atomic, $u_{FX}\equiv u_A$ is an equivalence. By lemma \ref{splitequiv}, this then implies that $u_{FY}\equiv u_B$ is an equivalence as well. This tells us that $\bB(B,-):\bB\to\Catic$ preserves arbitrary colimits and thus $B$ is atomic as well.
\end{proof}

\begin{proposition} \label{2retato}
Every atomic presheaf is a 2-retract of a representable presheaf.
\end{proposition}

\begin{proof}
We can express $S$ as the colimit of $\yo_\bB:\bB\to\Pshic{\bB}$ weighted by $S:\bB\op\to\Catic$ with a strong transformation $\lambda:S\to \Pshic{\bB}(\yo_\bB-,S)$ given by the Yoneda lemma. Since $S$ is atomic we can now apply the functor $\Pshic{\bB}(S,-)$ to this colimit and we have universal strong transformation
\begin{align*}
\lambda_*:S\to \Catic(\Pshic{\bB}(S,\yo_\bB-),\Pshic{\bB}(S,S))
\end{align*}
for the colimit of $\Pshic{\bB}(S,\yo_\bB-):\bB \to\Catic$ weighted by $S:\bB\op\to\Catic$. Since this is now a weighted colimit in $\Catic$ we can explicitely define another colimit using theorem \ref{x1}. We get a universal strong transformation
\begin{align*}
\kappa: \, & S\to \Cat(\Pshic{\bB}(S,\yo_\bB-),\colim^S \Pshic{\bB}(S,\yo_\bB-)) \\
\kappa_b(a): \, & \Pshic{\bB}(S,\bB(-,b))\to \colim^S \Pshic{\bB}(S,\yo_\bB-) \\
 & \kappa_b(a)(\alpha)=(a,\alpha)
\end{align*}
for the colimit in $\Cat$. Using theorem \ref{x2}, we can now form the colimit in $\Catic$ by taking the idempotent completion. Let $\K$ denote the idempotent completion of $\colim^S \Pshic{\bB}(S,\yo_\bB-)$ with embedding $\iota:\colim^S \Pshic{\bB}(S,\yo_\bB-)\to \K$. We now have a colimit strong transformation
\begin{align*}
\tilde{\kappa}:S\to \Catic(\Pshic{\bB}(S,\yo_\bB-),\K)
\end{align*}
given by $\tilde{\kappa}=\iota_*\kappa$. Since both $\K$ and $\Pshic{\bB}(S,S)$ are colimits of $\Pshic{\bB}(S,\yo_\bB-)$ weighted by $S$, we have an equivalence of categories $\Phi:\K\to\Pshic{\bB}(S,S)$ such that $\Phi_*\tilde{\kappa}\iso \lambda_*$. $\Phi$ is defined by the following: Let $(a,\alpha)$ with $a$ in $Sb$ and $\alpha:S\to\bB(-,b)$ be an object in $\colim^S \Pshic{\bB}(S,\yo_\bB-)$. We now have
\begin{align*}
\Phi\iota(a,\alpha)=\lambda_b(a)\alpha:S\to S.
\end{align*}
This defines $\Phi$ up to isomorphism since by corollary \ref{kar1equ} for any category $\C$ and idempotent complete category $\D$, precomposition with $\iota_\C:\C\to\Kar{\C}$ defines an equivalence
\begin{align*}
\iota_\C^*:\Cat(\Kar{\C},\D)\to\Cat(\C,\D).
\end{align*}
$\Phi$ now has the desired property since
\begin{align*}
(\Phi_*\tilde{\kappa})_b(a)(\alpha)=\Phi\iota\kappa_b(a)(\alpha)=\Phi\iota(a,\alpha)=\lambda_b(a)\alpha=\left(\lambda_*\right)_b(a)(\alpha).
\end{align*}
We know that $\Phi$ has to be an equivalence, so there exists an object $k$ in $\K$ such that $\Phi k\iso\id_S$ via an isomorphism $\gamma:\Phi k \to\id_S$. Since $k$ lives in the idempotent completion of $\colim^S \Pshic{\bB}(S,\yo_\bB-)$, we know due to remark \ref{eobsplit} that there exists an object $(a,\alpha)$ and an idempotent $p:(a,\alpha)\to (a,\alpha)$ in $\colim^S \Pshic{\bB}(S,\yo_\bB-)$ and morphisms $\tilde{\phi}:\iota(a,\alpha)\to k$ and $\tilde{\psi}:k\to \iota(a,\alpha)$ in $\K$ such that $\tilde{\phi}\tilde{\psi}=\id_k$ and $\tilde{\psi}\tilde{\phi}=\iota(p)$. We now have
\begin{align*}
\phi& =\gamma\Phi(\tilde{\phi}):\lambda_b(a)\alpha\to \id_S \text{ and} \\
\psi& =\Phi(\tilde{\psi})\gamma\inv:\id_S\to \lambda_b(a)\alpha
\end{align*}
such that $\phi\psi=\id_{\id_S}$. Therefore $S$ is a 2-idempotent splitting of the 2-idempotent $\alpha\lambda_b(a):\bB(-,b)\to\bB(-,b)$.
\end{proof}

Finally, we can now rigorously prove that the idempotent completion of a locally idempotent bicategory is equivalent to its Cauchy completion.

\begin{theorem} \label{x8}
Let $\bB$ be a locally idempotent complete category. The peudofunctor given by the composition
\begin{center}
\begin{tikzcd}
\Kar{\bB} \arrow[r, "\yo_{\Kar{\bB}}"] & \Pshic{\Kar{\bB}} \arrow[r, "\iota_\bB^*"] & \Pshic{\bB}
\end{tikzcd}
\end{center}
defines an equivalence of bicategories $\Kar{\bB}\equiv \Pshicat{\bB}$.
\end{theorem}

\begin{proof}
First, we want to show that the pseudofunctor takes values in atomic objects, i.e., for every object $A_p$ in $\Kar{\bB}$, the presheaf $\Kar{\bB}(\iota_\bB,A_p):\bB\op\to\Catic$ is atomic. Remark \ref{x3} states that $A_p$ is a splitting of the 2-idempotent $p$ on $A_{\id}$. By absoluteness of 2-splittings, we have that $\Kar{\bB}(\iota_\bB,A_p)$ is a splitting of an 2-idempotent on $\Kar{\bB}(\iota_\bB,A_{\id})=\Kar{\bB}(\iota_\bB,\iota_\bB A)$. Since $\iota_\bB$ is fully faithful, we have that $\Kar{\bB}(\iota_\bB,\iota_\bB A)\equiv\bB(-,A)$. This means that $\Kar{\bB}(\iota_\bB,A_p)$ is a retract of a representable presheaf and by lemma \ref{x7}, it therefore is atomic.

Next, we will show that the pseudofunctor is fully faithful. Since the Yoneda embedding is fully faithful and by theorem \ref{lscol} with $\bC=\Catic\op$, precomposition with $\iota_\bB$ is fully faithful, their composition must also be fully faithful.

Lastly, we check that the pseudofunctor is essentially surjective. Let $S$ be an atomic idempotent complete presheaf on $\bB$. By proposition \ref{2retato}, there exists an object $A$ in $\bB$ and a 2-idempotent $p$ on $A$ such that $S$ is a splitting of the 2-idempotent $p^*$ on $\bB(-,A)$. Remark \ref{x3} states that $A_p$ is a splitting of the 2-idempotent $p$ on $A_{\id}=\iota_\bB A$. By absoluteness of splittings $\Kar{\bB}(\iota_\bB,A_p)$ is a splitting of the 2-idempotent $p^*$ on $\Kar{\bB}(\iota_\bB,\iota_\bB A)\equiv \bB(-,A)$. Since splittings of 2-idempotents are unique up to equivalence, it follows that $S\equiv \Kar{\bB}(\iota_\bB,A_p)$.
\end{proof}

We can now also see that 2-idempotent completeness and completeness under absolute weighted colimits are the same concept in a locally idempotent complete bicategory.

\begin{corollary} \label{x10}
A locally idempotent complete bicategory is 2-idempotent complete if and only if it is complete under icAWCs.
\end{corollary}

\begin{proof}
Let $\bB$ be a 2-idempotent complete bicategory. By proposition \ref{kar2emb} and theorem \ref{x8}, we now have $\bB\equiv \Kar{\bB}\equiv \Pshicat{\bB}$. Since $\Pshicat{\bB}$ is complete under icAWCs, $\bB$ must also be. The opposite direction follows analogously by proposition \ref{cau2emb}.
\end{proof}
\appendix
\section{Appendix}

\subsection{Adjoint Pseudofunctors} \label{2adj}

There are many different ways of lifting the concept of adjoint functors into the realm of bicategories, many of which can be found in \cite[Chapter I,7]{G}. Our definition differs slightly from any of those listed there, \cite[Theorem 4.3.11.]{RV} shows that within our setting these notions are equivalent.

\begin{definition}{(Adjoint Pseudofunctors)}
Let $\bB$ and $\bC$ be bicategories. Two pseu\-do\-functors $F:\bB\to\bC$ and $G:\bC\to\bB$ are called \textit{adjoint} if there exist strong transformations $\eta:\id_\bB\to GF$, $\epsilon:FG\to\id_\bC$ and invertible modifications $\Gamma:\id_F \to (\epsilon F)\circ (F\eta)$ and $\Lambda:(G\epsilon)\circ (\eta G)\to \id_G$. We say that $F$ is \textit{left adjoint} to $G$, $G$ is \textit{right adjoint} to $F$ and the pseudofunctors $F$ and $G$ form an adjunction. We write $F\dashv G$.
\end{definition}

\begin{proposition} \label{2adjequiv}
Adjoint functors induce an equivalence of categories
\begin{align*}
\bC(Fb,c)\equiv \bB(b,Gc)
\end{align*}
for each pair of objects $b$ in $\bB$ and $c$ in $\bC$.
\end{proposition}

\begin{proof}
Let $f:Fb\to c$ be a 1-morphism in $\bB$. We can define a morphism $f\es:b\to Gc$ via
\begin{center}
\begin{tikzcd}
b \arrow[r, "\eta_b"] & GFb \arrow[r, "Gf"] & Gc.
\end{tikzcd}
\end{center}
Let $g:Fb\to c$ be another 1-morphism and $\theta:f\to g$ a 2-morphism. We can define a 2-morphism $\theta\es:f\es\to g\es$ via
\begin{center}
\begin{tikzpicture}
\node at (0,0) {$b$};
\node at (2,0) {$GFb$};
\node at (5,0) {$Gc.$};
\node at (0.8,0.2) {$\eta_b$};
\node at (3.5,0.8) {$Gf$};
\node at (3.5,-0.8) {$Gg$};
\node at (3.7,0) {$G\theta$};
\draw[->] (0.25,0) to (1.4,0);
\draw[->, out=30, in=150] (2.4,0.2) to (4.7,0.2);
\draw[->, out=-30, in=-150] (2.4,-0.2) to (4.7,-0.2);
\draw[-implies, double equal sign distance] (3.2,-0.3) to (3.2,0.3);
\end{tikzpicture}
\end{center}
One can check that this defines a functor $(-)\es:\bC(Fb,c)\to \bB(b,Gc)$. For a 1-morphism $f:b\to Gc$, we can analogously define a morphism $f\is:Fb\to c$ via
\begin{center}
\begin{tikzcd}
Fb \arrow[r, "Ff"] & FGc \arrow[r, "\epsilon_C"] & c.
\end{tikzcd}
\end{center}
This also defines a functor $(-)\is:\bB(b,Gc)\to \bC(Fb,c)$. We will now show that these two functors form an equivalence of categories. We can construct a natural transformation $\id_{\bC(Fb,c)} \to ((-)\es)\is$ with components given by the diagram
\begin{center}
\begin{tikzpicture}
\node at (0,0) {$Fb$};
\node at (3,0) {$FGFb$};
\node at (6,0) {$FGc$};
\node at (3,-3) {$Fb$};
\node at (6,-3) {$c,$};
\node at (1.5,0.3) {$F\eta_b$};
\node at (4.5,0.3) {$FGf$};
\node at (1.3,-1.7) {$\id_{Fb}$};
\node at (3.4,-1.5) {$\epsilon_{Fb}$};
\node at (6.3,-1.5) {$\epsilon_c$};
\node at (4.5,-3.3) {$f$};
\node at (1.9,-0.6) {$\Gamma_b$};
\node at (4.8,-1.7) {$\epsilon_f$};
\draw[->] (0.3,0) to (2.4,0);
\draw[->] (3.6,0) to (5.5,0);
\draw[->] (0.3,-0.2) to (2.7,-2.8);
\draw[->] (3,-0.2) to (3,-2.8);
\draw[->] (6,-0.2) to (6,-2.8);
\draw[->] (3.3,-3) to (5.8,-3);
\draw[-implies, double equal sign distance] (1.8,-1.2) to (2.5,-0.5);
\draw[-implies, double equal sign distance] (4.2,-1.8) to (5,-1);
\end{tikzpicture}
\end{center}
i.e., we have an isomorphism $f\to (f\es)\is$ given by
\begin{center}
\begin{tikzcd}
f \arrow[r, "r\inv"] & f\circ \id_{Fb} \arrow[r, "f\circ \Gamma_b"] & f\circ \epsilon_{Fb} \circ F\eta_B \arrow[r, "\epsilon_f\circ F\eta_b"] & \epsilon_C\circ FGf\circ F\eta_b=(f\es)\is.
\end{tikzcd}
\end{center}
That these isomorphism form a natural isomorphism, follows from the equality of the following two diagrams.
\begin{center}
\begin{tikzpicture}
\node at (0,0) {$Fb$};
\node at (3,0) {$FGFb$};
\node at (6,0) {$FGc$};
\node at (3,-3) {$Fb$};
\node at (6,-3) {$c$};
\node at (1.5,0.3) {$F\eta_b$};
\node at (4.5,0.7) {$FGg$};
\node at (1.3,-1.7) {$\id_{Fb}$};
\node at (3.4,-1.5) {$\epsilon_{Fb}$};
\node at (6.3,-1.5) {$\epsilon_c$};
\node at (4.5,-2.2) {$g$};
\node at (4.5,-3.8) {$f$};
\node at (1.9,-0.6) {$\Gamma_b$};
\node at (4.8,-1.2) {$\epsilon_f$};
\node at (4.65,-3) {$\theta$};
\draw[->] (0.3,0) to (2.4,0);
\draw[->, out=30, in=150] (3.6,0.2) to (5.5,0.2);
\draw[->] (0.3,-0.2) to (2.7,-2.8);
\draw[->] (3,-0.2) to (3,-2.8);
\draw[->] (6,-0.2) to (6,-2.8);
\draw[->, out=30, in=150] (3.3,-2.8) to (5.8,-2.8);
\draw[->, out=-30, in=-150] (3.3,-3.2) to (5.8,-3.2);
\draw[-implies, double equal sign distance] (1.8,-1.2) to (2.5,-0.5);
\draw[-implies, double equal sign distance] (4.2,-1.3) to (5,-0.5);
\draw[-implies, double equal sign distance] (4.35,-3.3) to (4.35,-2.7);
\node at (7.5,-1.5) {$=$};
\node at (8,0) {$Fb$};
\node at (11,0) {$FGFb$};
\node at (14,0) {$FGc$};
\node at (11,-3) {$Fb$};
\node at (14,-3) {$c$};
\node at (9.5,0.3) {$F\eta_b$};
\node at (12.5,-0.7) {$FGf$};
\node at (12.5,0.7) {$FGg$};
\node at (9.3,-1.7) {$\id_{Fb}$};
\node at (11.4,-1.5) {$\epsilon_{Fb}$};
\node at (14.3,-1.5) {$\epsilon_c$};
\node at (12.5,-3.8) {$f$};
\node at (9.9,-0.6) {$\Gamma_b$};
\node at (12.8,-2.2) {$\epsilon_f$};
\node at (12.7,0) {$FG\theta$};
\draw[->] (8.3,0) to (10.4,0);
\draw[->, out=30, in=150] (11.6,0.2) to (13.5,0.2);
\draw[->, out=-30, in=-150] (11.6,-0.2) to (13.5,-0.2);
\draw[->] (8.3,-0.2) to (10.7,-2.8);
\draw[->] (11,-0.2) to (11,-2.8);
\draw[->] (14,-0.2) to (14,-2.8);
\draw[->, out=-30, in=-150] (11.3,-3.2) to (13.8,-3.2);
\draw[-implies, double equal sign distance] (9.8,-1.2) to (10.5,-0.5);
\draw[-implies, double equal sign distance] (12.2,-2.3) to (13,-1.5);
\draw[-implies, double equal sign distance] (12.15,-0.3) to (12.15,0.3);
\end{tikzpicture}
\end{center}
We thus have a natural isomorphism $\id_{\bC(Fb,c)} \iso ((-)\es)\is$ and analogously also a natural isomorphism $((-)\is)\es\iso \id_{\bB(b,Gc)}$. Therefore they form an equivalence of categories.
\end{proof}

The statement that left adjoint functors preserve colimits also holds true in the bicategorical case.

\begin{proposition} \label{adjpres}
Left adjoint pseudofunctors preserve weighted colimits.
\end{proposition}

\begin{proof}
Let $\bJ,\bB$ be bicategories, $W:\bJ\op\to\Cat$ a weight, $F:\bJ\to\bB$ pseudofunctor and $\lambda:W\to \bB(F-,\colim^W F)$ a universal strong transformation. Now let $\bC$ be another bicategory and $G:\bB\to\bC$ and $H:\bC\to\bB$ adjoint pseudofunctors with strong transformations $\eta:\id_\bB\to HG$ and $\epsilon:GH\to \id_{\bC}$. We will show that the strong transformation $G\lambda:W\to \bC(GF-,G\colim^W F)$ is universal by showing that the following diagram commutes up to isomorphism for all objects $c$ in $\bC$
\begin{center}
\begin{tikzcd}
{\bC(G\colim^W F,c)} \arrow[r, "(G\lambda)^*"] \arrow[d, "(-)\es"] & {\Bicat(\bJ\op,\Cat)(W,\bC(GF-,c))} \arrow[d, "(-)\es"] \\
{\bB(\colim^W F,Hc)} \arrow[r, "\lambda^*"]                        & {\Bicat(\bJ\op,\Cat)(W,\bB(F-,Hc))}                    
\end{tikzcd}
\end{center}
where $(-)\es$ is defined pointwise on strong transformation. All of the functors except $(G\lambda)^*$ are known to be equivalences of categories, so we only need to show that it commutes up to isomorphism. Let $f:G\colim^W F\to c$ be a 1-morphism in $\bC$. We have an isomorphism $(f_*G\lambda)\es\iso f\es_*\lambda$ given by the diagram
\begin{center}
\begin{tikzpicture}
\node at (0,0) {$Fj$};
\node at (4,0) {$\colim^W F$};
\node at (0,-3) {$HGFj$};
\node at (4,-3) {$HG\colim^W F$};
\node at (8,-3) {$Hc$};
\node at (2,0.3) {$\lambda_j(a)$};
\node at (0.35,-1.5) {$\eta_{Fj}$};
\node at (4.8,-1.5) {$\eta_{\colim^W F}$};
\node at (1.7,-3.3) {$HG\lambda_j(a)$};
\node at (6.3,-3.3) {$Hf$};
\node at (6.5,-1.3) {$f\es$};
\node at (2.5,-1.7) {$\eta_{\lambda_j(a)}$};
\draw[->] (0.3,0) to (3.1,0);
\draw[->] (0,-0.3) to (0,-2.75);
\draw[->] (4,-0.3) to (4,-2.75);
\draw[->] (0.7,-3) to (2.8,-3);
\draw[->] (4.8,-0.2) to (7.7,-2.8);
\draw[->] (5.25,-3) to (7.7,-3);
\draw[-implies, double equal sign distance] (1.5,-1.8) to (2.3,-1);
\end{tikzpicture}
\end{center}
where $j$ is an object in $\bJ$ and $a$ an object in $Wj$. We can also read off of this diagram that this isomorphism is natural. Thus $G$ preserves the weighted colimit.
\end{proof}

\subsection{Weighted Colimits in \textbf{Cat}} \label{app1}

\begin{construction} \label{const}
Let $\bJ$ be a bicategory, $W:\bJ\op\to\Cat$ a weight and $F:\bJ\to\Cat$ a pseudofunctor. First we choose a set of objects $\Ob\bJ$ for $\bJ$. We look at the diagram
\begin{center}
\vspace*{-\baselineskip}
\hspace*{-24pt}
\begin{tikzcd}
{\coprod_{j,j\p,j\pp\in\Ob\bJ}Wj\pp\times \bJ(j\p,j\pp)\times\bJ(j,j\p)\times Fj} \arrow[r, "d^0_2"{yshift=-2}, shift left=3] \arrow[r, "d^1_2"{yshift=-2}, shift left =-1] \arrow[r, "d^2_2"{yshift=-2}, shift left=-5] & {\coprod_{j,j\p\in\Ob\bJ}Wj\p\times\bJ(j,j\p)\times Fj} \arrow[r, "d^0_1"{yshift=-2}, shift left=3] \arrow[r, "d^1_1"{yshift=-2}, shift left=-5] & \coprod_{j \in \Ob\bJ}Wj\times Fj \arrow[l, "s^0_0"'{yshift=-2}, shift left=1]
\end{tikzcd}
\end{center}
where the maps are defined in the following way
\begin{align*}
d_2^0: & (a,g,f,x) \mapsto (a,g,F(f)x), \\
d_2^1: & (a,g,f,x) \mapsto (a,gf,x), \\
d_2^2: & (a,g,f,x) \mapsto (W(g)a,f,x), \\
d_1^0: & (a,f,x) \mapsto (a,F(f)x), \\
d_1^1: & (a,f,x) \mapsto (W(f)a,x), \\
s_0^0: & (a,x) \mapsto (a,\id_j,x).
\end{align*}

We now define the category $\colim^W F$ by taking the category $\coprod_{j \in \Ob\bJ}Wj\times Fj$ and freely adding isomorphism determined by $d^0_1$ and $d^1_1$ that are subject to relations given by $s^0_0$, $d^0_2$, $d^1_2$ and $d^2_2$.

For each pair of objects $j,j\p$ in $\bJ$ and each object $(a,f,x)$ in $Wj\p\times\bJ(j,j\p)\times Fj$ we freely add an isomorphism 
\begin{align*}
\gamma_{a,f,x}: (a,F(f)x)\to (W(f)a,x)
\end{align*}
which assemble into a natural transformation
\begin{align*}
\gamma_{-,-,-}: (-,F(-)-) \to (W(-)-,-): Wj\p\times \bJ(j,j\p) \times Fj \to \colim^W F.
\end{align*}

These natural transformations have to satisfy the following two conditions. For all $j,j\p,j\pp\in\Ob\bJ$, $x$ in $Fj$, $f:j\to j\p$, $g:j\p\to j\pp$ and $a$ in $Wj\pp$, the diagram
\begin{center}
\begin{tikzcd}
{(a,F(g)F(f)x)} \arrow[d, "{(\id_a,F^2\id_x)}"] \arrow[r, "{\gamma_{a,g,F(f)x}}"] & {(W(g)a,F(f)x)} \arrow[r, "{\gamma_{W(g)a,f,x}}"] & {(W(f)W(g)a,x)} \arrow[d, "{(W^2\id_a,\id_x)}"] \\
{(a,F(gf)x)} \arrow[rr, "{\gamma_{a,gf,x}}"]                               &                                                   & {(W(gf)a,x)}                            
\end{tikzcd}
\end{center}
has to commute, and for all $j\in\Ob\bJ$, $x$ in $Fj$ and $a$ in $Wj$, the diagram
\begin{center}
\begin{tikzcd}
{(a,\id_{Fj}x)} \arrow[r, "="] \arrow[d, "{(\id_a,F^0\id_x)}"] & {(a,x)} \arrow[r, "="] & {(\id_{Wj}a,x)} \arrow[d, "{(W^0\id_a,\id_x)}"] \\
{(a,F(\id_j)x)} \arrow[rr, "{\gamma_{a,\id_j,x}}"]      &                        & {(W(\id_j)a,x)}                         
\end{tikzcd}
\end{center}
has to commute. $F^2$ and $F^0$ denote the coherence data of the pseudofunctor $F$ and analogously for $W$. We call these two diagrams the naturality and unitality conditions of $\gamma$. $\colim^W F$ is now the category obtained by adding these ismorphism subject to the given relations freely to $\coprod_{j \in \Ob\bJ}Wj\times Fj$. We now also define the functors
\begin{align*}
\gamma_j:W(j)\times F(j) \to \colim^W F
\end{align*}
which are given by inclusions since $W(j)\times F(j)$ is a subcategory of $\colim^W F$.

We can now define a strong transformation
\begin{align*}
\lambda:W\to \Cat(F-,\colim^W F):\bJ\op\to\Cat.
\end{align*}
For each object $j$ in $\bJ$, we have a 1-morphism $\lambda_j:Wj\to \Cat(Fj,\colim^W F)$ given by $\lambda_j(a)(x)=\gamma_j(a,x)=(a,x)$
for each $a$ in $Wj$ and $x$ in $Fj$. Since $\gamma_j$ is a functor, both $\lambda_j$ and $\lambda_j(a)$ also define functors.

For each pair of objects $j,j\p$ in $\bJ$, we also have a natural transformation
\begin{align*}
\lambda: \lambda_{j\p}^*\Cat(F-,\colim^W F)\to (\lambda_j)_*W:\bJ(j,j\p)\to \Cat(Wj\p,\Cat(Fj,\colim^W F))
\end{align*}
with component morphisms  $\lambda_f: F(f)^*\lambda_{j\p} \to \lambda_j W(f)$ which are defined by
\begin{align*}
\lambda_{f,a,x} =\gamma_{a,f,x}: \left(F(f)^*\lambda_{j\p}(a)\right)(x) & = \lambda_{j\p}(a)(F(f)x) = (a,F(f)x) \\
& \to (W(f)a,x)=\lambda_j(W(f)a)(x)
\end{align*}
for each $a$ in $Wj\p$ and $x$ in $Fj$. The naturality of $\lambda$ and $\lambda_f$ and $\lambda_{f,a}$ all follow from the naturality of $\gamma$. Lastly we need to check that $\lambda$ satisfies the naturality and unitality conditions, this follows since $\gamma$ needs to satisfy its own naturality and unitality conditions.
\end{construction}

\begin{theorem} \label{x1}
Given a bicategory $\bJ$, a weight $W:\bJ\op\to\Cat$ and a pseudofunctor $F:\bJ\to\Cat$, construction \ref{const} defines a colimit of $F$ weighted by $W$.
\end{theorem}

\begin{proof}
We next need to check that precomposition with $\lambda: W \to \Cat(F-,\colim^W F)$ defines an equivalence of pseudofunctors
\begin{align*}
\Cat(\colim^W F,-)\to \Psh{\bJ}(W,\Cat(F,-)).
\end{align*}
For this it suffices to check that it defines an equivalence of categories
\begin{align*}
\Cat(\colim^W F,\C)\to \Psh{\bJ}(W,\Cat(F-,\C))
\end{align*}
for each category $\C$. As it turns out, a functor $\colim^W F\to \C$ is determined by the same data as a strong transformation $W\to \Cat(F-,\C)$ just presented in different ways. This will lead to an isomorphism of categories.

A strong transformation $\epsilon:W\to\Cat(F-,\C)$ is given by a $\bJ$-indexed family of functors $\epsilon_j:Wj\to\Cat(Fj,\C)$ and for each morphism $f:j\to j\p$ in $\bJ$ a natural isomorphism $\epsilon_f:F(f)^*\epsilon_{j\p}\to\epsilon_j W(f):Wj\p\to\Cat(Fj,\C)$ such that the $\epsilon_f$ are natural in $f$ and satisfy the naturality and unitality conditions.

Using the fact that for categories $\C_1,\C_2,\D$ we have an isomorphism of categories
\begin{align*}
\Cat(\C_1\times\C_2,\D)\iso\Cat(\C_1,\Cat(\C_2,\D))
\end{align*} 
given by currying, know that the family $(\epsilon_j)_{j\in\Ob\bJ}$ is equivalent to a $\bJ$-indexed family of functors $\tilde{\epsilon}_j:Wj\times Fj\to \C$ and the $\epsilon_f$ turn into natural transformations 
\begin{align*}
\tilde{\epsilon}_f:\tilde{\epsilon}_{j\p}(-,F(f)-)\to \tilde{\epsilon}_j(W(f)-,-):Wj\p\times Fj\to\C
\end{align*}
Naturality in $f$ now means that they assemble into a natural transformation
\begin{align*}
\tilde{\epsilon}:\tilde{\epsilon}_{j\p}(-,F(-)-)\to \tilde{\epsilon}_j(W(-)-,-):Wj\p\times \bJ(j,j\p)\times Fj\to\C.
\end{align*}
The naturality condition tranlates into the following commutative diagram
\begin{center}
\begin{tikzcd}
{\tilde{\epsilon}_{j\pp}(a,F(g)F(f)x)} \arrow[d, "{\tilde{\epsilon}_{j\pp}(\id_a,F^2\id_x)}"] \arrow[r, "{\tilde{\epsilon}_{a,g,F(f)x}}"] & {\tilde{\epsilon}_{j\p}(W(g)a,F(f)x)} \arrow[r, "{\tilde{\epsilon}_{W(g)a,f,x}}"] & {\tilde{\epsilon}_j(W(f)W(g)a,x)} \arrow[d, "{\tilde{\epsilon}_j(W^2\id_a,\id_x)}"] \\
{\tilde{\epsilon}_{j\pp}(a,F(gf)x)} \arrow[rr, "{\tilde{\epsilon}_{a,gf,x}}"]                               &                                                   & {\tilde{\epsilon}_j(W(gf)a,x)}                            
\end{tikzcd}
\end{center}
which has to hold for all objects $j,j\p,j\pp$ in $\bJ$, morphisms $f:j\to j\p$ and $g:j\p\to j\pp$ and objects $a$ in $Wj\pp$ and $x$ in $Fj$. The unitality condition tranlates into the diagram
\begin{center}
\begin{tikzcd}
{\tilde{\epsilon}_j(a,\id_{Fj}x)} \arrow[r, "="] \arrow[d, "{\tilde{\epsilon}_j(\id_a,F^0\id_x)}"] & {\tilde{\epsilon}_j(a,x)} \arrow[r, "="] & {\tilde{\epsilon}_j(\id_{Wj}a,x)} \arrow[d, "{\tilde{\epsilon}_j(W^0\id_a,\id_x)}"] \\
{\tilde{\epsilon}_j(a,F(\id_j)x)} \arrow[rr, "{\tilde{\epsilon}_{a,\id_j,x}}"]      &                        & {\tilde{\epsilon}_j(W(\id_j)a,x)}                         
\end{tikzcd}
\end{center}
which has to commute for all objects $j$ in $\bJ$, $a$ in $Wj$ and $x$ in $Fj$. We can already see that these are exactly the conditions we asked of $\gamma$ when constructing $\colim^W F$. We will now look at the data that a functor $G:\colim^W F\to \C$ consists of. Since $\colim^W F$ was constructed by adding isomorphisms to the category $\coprod_{j \in \Ob\bJ}Wj\times Fj$, the functor $G$ consists a family of functors $G_j:Wj\times Fj\to \C$ that also has to map the components $\gamma$ on isomorphisms in $\C$ that have to satisfy the same properties as $\gamma$. So we have a natural transformation 
\begin{align*}
G(\gamma_{-,-,-}):G(-,F(-)-)\to G(W(-)-,-):Wj\p\times \bJ(j,j\p)\times Fj\to\C
\end{align*}
and the components of this natural transformation also have to satisfy the unitality and naturality conditions. So we indeed see that a functor $\colim^W F\to \C$ consists of the same data as a strong transformation $W\to\Cat(F-,\C)$ and the presentation of this data only differs by currying. This currying is also induced by precomposition with $\lambda$ since given a functor $G:\colim^W F\to \C$, the strong transformation $G_*\lambda:W\to\Cat(F-,\C)$ is defined by
\begin{align*}
(G_*\lambda)_j(a)(x)=(G_*\lambda_j)(a)(x)=(G\lambda_j(a))(x)=G(\lambda_j(a)(x))=G(a,x)
\end{align*}
for objects $j$ in $\bJ$, $a$ in $Wj$ and $x$ in $Fj$, and by
\begin{align*}
(G_*\lambda)_{f,a,x}=G\lambda_{f,a,x}=G\gamma_{a,f,x}
\end{align*}
for objects $j,j\p$ in $\bJ$, a morphism $f:j\to j\p$ and objects $a$ in $Wj\p$ and $x$ in $Fj$. To show that this induces an isomorphism of categories
\begin{align*}
\Cat(\colim^W F,\C)\iso \Psh{\bJ}(W,\Cat(F-,\C))
\end{align*}
we also have to show that a natural transformation between such functors also only differs by currying from a modification between such strong transformations. Let $\epsilon,\eta:W\to \Cat(F-,\C)$ be two strong transformations and let $\Gamma:\epsilon\to\eta$ be a modification between them. $\Gamma$ consists of a $\bJ$-indexed family of natural transformations
\begin{align*}
\Gamma_j:\epsilon_j\to\eta_j:Wj\to\Cat(Fj,\C)
\end{align*}
which since $\Gamma$ is a modification has to satisfy a certain property. Using currying, this the natural transformations $\Gamma_j$ correspond to natural transformations
\begin{align*}
\tilde{\Gamma}_j:\tilde{\epsilon}\to\tilde{\eta}:Wj\times Fj\to \C
\end{align*}
and the property they need to satisfy can be written as the diagram
\begin{center}
\begin{tikzcd}
{\tilde{\epsilon}_{j\p}(a,F(f)x)} \arrow[r, "{\tilde{\epsilon}_{a,f,x}}"] \arrow[d, "{\tilde{\Gamma}_{j\p,a,F(f)x}}"] & {\tilde{\epsilon}_{j}(W(f)a,x)} \arrow[d, "{\tilde{\Gamma}_{j,W(f)a,x}}"] \\
{\tilde{\eta}_{j\p}(a,F(f)x)} \arrow[r, "{\tilde{\eta}_{a,f,x}}"]                                                     & {\tilde{\eta}_{j}(W(f)a,x)}                                              
\end{tikzcd}
\end{center}
which has to commute for all objects $j,j\p$ in $\bJ$, morphisms $f:j\to j\p$ and objects $a$ in $Wj\p$ and $x$ in $Fj$. Now let $G,H:\colim^W F\to \C$ be two functors and let $\phi:G\to H$ be a natural transformations. We've already seen that due to the structure of $\colim^W F$ the two functors $G,H$ consists of $\bJ$-indexed families of functors $G_j,H_j:Wj\times Fj\to \C$ together with the natural transformations $G(\gamma)$ and $H(\gamma)$. A natural transformation $\phi:G\to H$ now consist of a $\bJ$-indexed family of natural transformations $\phi_j:G_j\to H_j$ which also have to be natural with respect to $\gamma$. This can be represented by the diagram
\begin{center}
\begin{tikzcd}
{G_{j\p}(a,F(f)x)} \arrow[r, "{G(\gamma_{a,f,x})}"] \arrow[d, "{\phi_{j\p,(a,F(f)x)}}"] & {G_{j}(W(f)a,x)} \arrow[d, "{\phi_{j,(W(f)a,x)}}"] \\
{H_{j\p}(a,F(f)x)} \arrow[r, "{H(\gamma_{a,f,x})}"]                                     & {H_{j}(W(f)a,x)}                                  
\end{tikzcd}
\end{center}
which has to commute for all objects $j,j\p$ in $\bJ$, morphisms $f:j\to j\p$ and objects $a$ in $Wj\p$ and $x$ in $Fj$. So now we also see that a natural transformation $\phi:G\to H$ between functors $G,H:\colim^W F\to \C$ consists of the same data as a modification $\Gamma:\epsilon\to\eta$ between strong transformations $\epsilon,\eta:W\to\Cat(F-,\C)$ and the presentation of this data also only differs by currying.

Thus, we have shown that precomposition with $\lambda:W\to\Cat(F-,\colim^W F)$, which does the same as currying, actually defines an isomorphism
\begin{align*}
\lambda^*:\Cat(\colim^W F,\C)\to \Psh{\bJ}(W,\Cat(F-,\C)).
\end{align*}
for each category $\C$. Thus $\colim^W F$ together with $\lambda$ forms a colimit of $F$ weighted by $W$.
\end{proof}

\begin{corollary} \label{catcc}
The 2-category $\Cat$ of categories is cocomplete, i.e., for every bicategory $\bJ$, every weight $W:\bJ\op\to\Cat$ and every pseudofunctor $F:\bJ\to\Cat$ the colimit of $F$ weighted by $W$ exists.
\end{corollary}

Another bicategory that interest us is the bicategory of presheaves over a given bicategory.

\begin{theorem} \label{PshCocomp}
Let $\bB$ be a bicategory. The bicategory $\Psh{\bB}$ is cocomplete.
\end{theorem}

\begin{proof}
Let $\bJ$ be a bicategory, $W:\bJ\op\to\Cat$ a weight and $F:\bJ\to\Psh{\bB}$ a pseudofunctor. We can construct a colimit of $F$ weighted by $W$ in the following way. Let $A$ be an object in $\bB$. We now have the pseudofunctor $F(-)(A):\bJ\to\Cat$ and since $\Cat$ is cocomplete the colimit of $F(-)(A)$ weighted by $W$ exists, i.e., we have a universal strong transformation $\lambda_A:W\to \Cat(F(-)(A),\colim^W F(-)(A))$.

We now define the presheaf $\colim^W F:\bB\op\to\Cat$ via $\colim^W F(A)=\colim^W F(-)(A)$. Let $B$ be another object in $\bB$ and $f:A\to B$ a morphism. We can now define the strong transformation $F(-)(f)^*\lambda_A:W\to \Cat(F(-)(B),\colim^W F(A))$ which has components $(F(-)(f)^*\lambda_A)_j(a)=\lambda_{A,j}(a)F(j)(f)$
\begin{center}
\begin{tikzcd}
F(j)(B) \arrow[r, "F(j)(f)"] & F(j)(A) \arrow[r, "{\lambda_{A,j}(a)}"] & \colim^W F(A)
\end{tikzcd}
\end{center}
for each $j$ in $\bJ$ and $a$ in $Wj$. By the universal property of the colimit, we get a morphism 
\begin{align*}
\colim^W F(f):\colim^W F(B)\to\colim^W F(A)
\end{align*}
and an isomorphism $\lambda_f:\colim^W F(f)_*\lambda_B \to F(-)(f)^*\lambda_A$ which we can represent by the following diagram
\begin{center}
\begin{tikzcd}
F(j)(B) \arrow[r, "F(j)(f)"] \arrow[d, "{\lambda_{B,j}(a)}"']                        & F(j)(A) \arrow[d, "{\lambda_{A,j}(a)}"] \\
\colim^W F(B) \arrow[r, "\colim^WF(f)"'] \arrow[ru, "{\lambda_{f,j,a}}", Rightarrow] & \colim^W F(A)                          
\end{tikzcd}
\end{center}
For a second 1-morphism $g:A\to B$ and a 2-morphism $\theta:f\to g$, we get a morphism of strong transformations
\begin{align*}
F(-)(\theta)^*\lambda_A:F(-)(f)^*\lambda_A\to F(-)(g)^*\lambda_A
\end{align*}
which gives us a morphism $\colim^W F(\theta):\colim^W F(f)\to \colim^W F(g)$. This can be represented as
\begin{center}
\begin{tikzpicture}
\node at (0,0) {$F(j)(B)$};
\node at (4,0) {$F(j)(A)$};
\node at (0,-3) {$\colim^W F(B)$};
\node at (4,-3) {$\colim^W F(A)$};
\node at (2,1.1) {$F(j)(g)$};
\node at (2,-1.1) {$F(j)(f)$};
\node at (2,-4) {$\colim^W F(f)$};
\node at (-0.7,-1.5) {$\lambda_{B,j}(a)$};
\node at (4.7,-1.5) {$\lambda_{A,j}(a)$};
\node at (2.2,0) {$F(j)(\theta)$};
\node at (1.4,-1.8) {$\lambda_{f,j,a}$};
\draw[->, out=30, in=150] (0.25,0.25) to (3.75,0.25);
\draw[->, out=-30, in=-150] (0.25,-0.25) to (3.75,-0.25);
\draw[->, out=-30, in=-150] (0.5,-3.25) to (3.5,-3.25);
\draw[->] (0,-0.3) to (0,-2.7);
\draw[->] (4,-0.3) to (4,-2.7);
\draw[-implies, double equal sign distance] (1.3,-0.3) to (1.3,0.3);
\draw[-implies, double equal sign distance] (1.3,-2.5) to (2.7,-1.7);
\node at (6,-1.5) {$=$};
\node at (8,0) {$F(j)(B)$};
\node at (12,0) {$F(j)(A)$};
\node at (8,-3) {$\colim^W F(B)$};
\node at (12,-3) {$\colim^W F(A)$};
\node at (10,1) {$F(j)(g)$};
\node at (10,-2) {$\colim^W F(g)$};
\node at (10,-4) {$\colim^W F(f)$};
\node at (7.3,-1.5) {$\lambda_{B,j}(a)$};
\node at (12.7,-1.5) {$\lambda_{A,j}(a)$};
\node at (10.1,-3) {\tiny{$\colim^W F(\theta)$}};
\node at (9.4,-0.6) {$\lambda_{g,j,a}$};
\draw[->, out=30, in=150] (8.25,0.25) to (11.75,0.25);
\draw[->, out=30, in=150] (8.5,-2.75) to (11.5,-2.75);
\draw[->, out=-30, in=-150] (8.5,-3.25) to (11.5,-3.25);
\draw[->] (8,-0.3) to (8,-2.7);
\draw[->] (12,-0.3) to (12,-2.7);
\draw[-implies, double equal sign distance] (9.25,-3.4) to (9.25,-2.6);
\draw[-implies, double equal sign distance] (9.3,-1.3) to (10.7,-0.5);
\end{tikzpicture}
\end{center}
The presheaf $\colim^W F$ is now well defined along with a strong transformation $\lambda:W\to \Psh{\bB}(F,\colim^W F)$ with components $\lambda_{j,A}(a)=\lambda_{A,j}(a)$. We must now show that $\lambda$ defines a universal strong transformation. Let $\kappa:W\to \Psh{\bB}(F,S)$ be another strong transformation for a presheaf $S$. For an object $A$ in $\bB$, we can then define the natural transformation $\kappa_A:\Psh{\bB}(F(-)(A),S(A))$ and by the universal property of the colimit, we get a morphism $\phi_A:\colim^W F(A)\to S(A)$ such that $\kappa_A\iso (\phi_A)_*\lambda_A$. The $\phi_A$ assemble into a morphism of presheafs $\phi:\colim^W F\to S$ and we get an isomorphism $\kappa\iso \phi_*\lambda$. This shows that $\lambda$ defines a colimit.
\end{proof}

\subsection{Weighted Colimits in \texorpdfstring{\textbf{Cat}\textsubscript{ic}}{Catic}} \label{app2}

We also want to show that $\Catic$ is cocomplete as a locally idempotent complete bicategory. For this we will first have to show that it actually is locally idempotent complete.

\begin{lemma}
For a category $\C$ and a idempotent complete category $\D$, the category $\Cat(\C,\D)$ of functors and natural transformations is idempotent complete.
\end{lemma}

\begin{proof}
Let $F:\C\to\D$ be a functor and $p:F\to F$ an idempotent natural transformation, i.e., for every object $c$ in $\C$ the morphism $p_c:Fc\to Fc$ is an idempotent. Since these are morphisms in $\D$, we can choose a splitting for every idempotent $p_c$. We now choose for every $c$ in $\C$ an object $Sc$ in $\D$ and morphisms $f_c:Fc\to Sc$ and $g_c:Sc\to Fc$ such that $g_c f_c=p_c$ and $f_c g_c=\id_{Sc}$. We can now turn $S$ into a functor. Let $h:c\to d$ be a morphism in $\C$. We define $S(h)=f_d F(h) g_c$. We now have $S(\id_c)=f_c F(\id_c) g_c=f_c g_c=\id_{Sc}$ and for another morphism $k:d\to e$, we have
\begin{align*}
S(k) S(h) & \, =f_e F(k) g_d f_d F(h) g_c=f_e F(k) p_d F(h) g_c=f_e F(k) F(h) p_c g_c \\
& = f_e F(kh) g_c f_c g_c= f_e F(kh) g_c \id_{Sc} = S(kh)
\end{align*}
which shows that $S$ defines a functor. We lastly need to show that the $f_c$ and $g_c$ define natural transformations $f:F\to S$ and $g:S \to F$. For this we need to check that the following two diagrams commute.
\begin{center}
\begin{tikzcd}
Fc \arrow[r, "F(h)"] \arrow[d, "f_c"] & Fd \arrow[d, "f_d"] & Sc \arrow[r, "S(h)"] \arrow[d, "g_c"] & Sd \arrow[d, "g_d"] \\
Sc \arrow[r, "S(h)"]                  & Sd                  & Fc \arrow[r, "F(h)"]                  & Fd                 
\end{tikzcd}
\end{center}
We can show
\begin{align*}
S(h) f_c = f_d F(h) g_c f_c=f_d F(h) p_c = f_d p_d F(h) = f_d g_d f_d F(h)= \id_{Sd} f_d F(h) = f_d F(h) 
\end{align*}
and thus $f:F\to S$ defines a natural transformation and
\begin{align*}
g_d S(h) = g_d f_d F(h)g_c = p_d F(h) g_c = F(h) p_c g_c = F(h) g_c f_c g_c = F(h) g_c \id_{Sc} = F(h) g_c
\end{align*}
and thus $g:S\to F$ defines a natural transformation. The natural transformations $f$ and $g$ now satisfy $gf=p$ and $fg=\id_{S}$ and thus $p$ splits.
\end{proof}

\begin{corollary} \label{Catlic}
The bicategories $\Catic$ and $\Catic\op$ are locally idempotent complete.
\end{corollary}

\begin{theorem} \label{x2}
The locally idempotent complete bicategory $\Catic$ of idempotent complete categories is ic-cocomplete, i.e., for every bicategory $\bJ$, every ic-weight $W:\bJ\op\to\Catic$ and every pseudofunctor $F:\bJ\to\Catic$ the colimit of $F$ weighted by $W$ exists.
\end{theorem}

\begin{proof}
Let $\bJ$ be a bicategory, $W:\bJ\op\to\Catic$ a weight and $F:\bJ\to\Catic$ a pseudofunctor. We can regard both $W$ and $F$ as pseudofunctors taking values in $\Cat$. By theorem \ref{x1}, we now know that we can construct a category $\C$ along with a strong transformation $\lambda:W\to\Cat(F-,\C)$ such that we have a natural equivalence
\begin{align*}
\lambda^* :\Cat(\C,-)\to \Psh{\bJ}(W,\Cat(F,-)).
\end{align*}
We know by theorem \ref{1adj}, that we have a left adjoint pseudofunctor $\Kartwo:\Cat\to\Catic$ and thus, by proposition \ref{adjpres}, we have a natural equivalence
\begin{align*}
(\Kartwo\lambda)^*:\Catic(\Kar{\C},-)\to\Psh{\bJ}(W,\Catic(\Kartwo F,-))
\end{align*}
where $\Kartwo\lambda$ is defined by
\begin{align*}
(\Kartwo\lambda)_j(a)=\Kar{\lambda_j(a)}:\Kar{Fj}\to\Kar{\C}
\end{align*}
for each $j$ in $\bJ$ and $a$ in $Wj$. Since $F$ already takes values in $\Catic$, $F$ and $\Kartwo F$ are equivalent with an equivalence given by $\iota_F:F\to \Kartwo F$ which has components given by $\iota_{Fj}:Fj\to \Kar{Fj}$. We also know that the diagram
\begin{center}
\begin{tikzcd}
Fj \arrow[d, "\iota_{Fj}"] \arrow[r, "\lambda_j(a)"] & \C \arrow[d, "\iota_\C"] \\
\Kar{Fj} \arrow[r, "\Kar{\lambda_j(a)}"]             & \Kar{\C}                 
\end{tikzcd}
\end{center}
commutes. These two facts combined tell us that we have an equivalence
\begin{align*}
((\iota_\C)_*\lambda)^*: \Catic(\Kar{\C},-)\to \Psh{\bJ}(W,\Catic(F,-))
\end{align*}
and since $W$ takes values in $\Catic\op$ and $\Catic$ is locally idempotent complete, we get the desired result that we have an equivalence
\begin{align*}
((\iota_\C)_*\lambda)^*: \Catic(\Kar{\C},-)\to \Pshic{\bJ}(W,\Catic(F,-))
\end{align*}
\end{proof}

Lastly, we also want to show that the bicategory $\Pshic{\bB}$ of presheaves taking values in idempotent complete categories is cocomplete as a locally idempotent complete bicategory.

\begin{lemma}
For a bicategory $\bB$ and a locally idempotent complete bicategory $\bC$ the bicategory $\Bicat(\bB,\bC)$ is locally idempotent complete.
\end{lemma}

\begin{proof}
Let $F,G:\bB\to \bC$ be pseudofuncors, $\alpha:F\to G$ a strong transformation and $p:\alpha\to \alpha$ an idempotent modification, i.e., for each object $b$ in $\bB$ the morphism $p_b:\alpha_b\to \alpha_b$ in $\bC(Fb,Gb)$ is idempotent. Since $\bC$ is locally idempotent complete, we can choose a splitting for each $p_b$ for all $b$ in $\bB$. We therefore have 1-morphisms $s_b:Fb\to Gb$, and 2-morphisms $f_b:\alpha_b\to s_b$ and $g_b:s_b\to \alpha_b$ such that $g_b f_b=p_b$ and $f_b g_b=\id_{s_b}$. $s:F\to G$ forms a strong transformation with component 2-morphisms $s_h:G(h)s_b\to s_c F(h)$ given by
\begin{center}
\begin{tikzpicture}
\node at (0,0) {$Fb$};
\node at (3,0) {$Fc$};
\node at (0,-3) {$Gb$};
\node at (3,-3) {$Gc$};
\node at (1.5,0.25) {$F(h)$};
\node at (1.5,-3.25) {$G(h)$};
\node at (-0.85,-1.5) {$s_b$};
\node at (0.8,-1) {$\alpha_b$};
\node at (2.2,-2) {$\alpha_c$};
\node at (3.85,-1.5) {$s_c$};
\node at (0,-1.4) {$g_b$};
\node at (3,-1.4) {$f_b$};
\node at (1.35,-1.4) {$\alpha_h$};
\draw[->] (0.3,0) to (2.7,0);
\draw[->] (0.3,-3) to (2.7,-3);
\draw[->, out=-120, in=120] (-0.2,-0.2) to (-0.2,-2.8);
\draw[->, out=-60, in=60] (0.2,-0.2) to (0.2,-2.8);
\draw[->, out=-120, in=120] (2.8,-0.2) to (2.8,-2.8);
\draw[->, out=-60, in=60] (3.2,-0.2) to (3.2,-2.8);
\draw[-implies, double equal sign distance] (-0.4,-1.8) to (0.4,-1.8);
\draw[-implies, double equal sign distance] (2.6,-1.8) to (3.4,-1.8);
\draw[-implies, double equal sign distance] (1.3,-2.2) to (2.1,-0.8);
\end{tikzpicture}
\end{center}
for a 1-morphism $h:b\to c$ in $\bB$. $f:\alpha\to s$ and $g:s\to \alpha$ now form modifications such that $gf=p$ and $fg=\id_s$ and thus $p$ splits.
\end{proof}

\begin{corollary} \label{Pshlic}
Let $\bB$ be a bicategory. The bicategory $\Pshic{\bB}$ of idempotent complete presheaves on $\bB$ is locally idempotent complete.
\end{corollary}

Combining the two theorems \ref{PshCocomp} and \ref{x2}, and their proofs furthermore yields the following result.

\begin{corollary} \label{Pshicc}
Let $\bB$ be a bicategory. The locally idempotent complete bicategory $\Pshic{\bB}$ is ic-cocomplete.
\end{corollary}

\subsection{Extensions along Embeddings} \label{app3}

In the following we will look at the theory of extending pseudofunctors along embeddings, i.e., pseudofunctors that are fully faithful.

\begin{definition}{(Extensions)}
Let $\bB$, $\bC$ and $\bD$ be bicategories, $F:\bB\to \bC$ an arbitrary pseudofunctor and $\iota:\bB\to\bD$ a fully faithful pseudofunctor. An \textit{extension} of $F$ along $\iota$ is a pseudofunctor $F\p:\bD\to\bC$ together with an equivalence of pseudofunctors $\phi:F\p\iota\to F$.
\end{definition}

\begin{proposition}\label{ext}
Let $\bB$, $\bC$, $\bD$, $F:\bB\to \bC$  and $\iota:\bB\to\bD$ be as above. An extension of $F$ along $\iota$ exists, if for all $d\in\bD$ the colimit of $F$ weighted by $\bD(\iota-,d):\bB\op\to\Cat$ exists.
\end{proposition}

\begin{proof}
We define $F\p$ on objects $d\in\bD$ to be the weighted colimit $\colim^{\bD(\iota-,d)}F$. On Hom-categories, we have functors $\bD(d,d\p) \to \bC(F\p d,F\p d\p)$ which are defined in the following way:

For every 1-morphism $f:d\to d\p$, we have a strong transformation $f_*:\bD(\iota-,d)\to \bD(\iota-,d\p)$ given by postcompostion. By our assumption, we have a universal strong transformation $\lambda^d:\bD(\iota-,d)\to\bC(F-,F\p d)$ for each $d\in\bD$. We can now form $\lambda^{d\p}\circ f_*: \bD(\iota-,d)\to\bC(F-,F\p d\p)$.

For all $d\in \bD$ we have an equivalence of categories
\begin{align*}
\bC(F\p d, c)\equiv \Psh{\bJ}(\bD(\iota-,d),\bC(F-,c))
\end{align*}
given by precomposition with $\lambda^d$. We define $\eta^d$ to be an inverse to this equivalence and can now define the functors $F\p:\bD(d,d\p) \to \bC(F\p d,F\p d\p)$ via $F\p f =\eta^d(\lambda^{d\p}\circ f_*)$.

We now need to check that this assignment is functorial, i.e., that for objects $d$, $d\p$, $d\pp$ and morphisms $f:d\to d\p$ and $g:d\p \to d\pp$,  we have coherence isomorphisms $\id_{F\p d}\iso F\p \id_d$ and $F\p g \circ F\p f \iso F\p(g\circ f)$ subject to compatibility conditions.

The first isomorphism is simply given by $F\p \id_d = \eta^d(\lambda^d \circ (\id_d)_*) \iso \eta^d(\lambda^d)\iso \id_{F\p d}$. For the second one, we need to find an isomorphism
\begin{align*}
F\p g\circ F\p f=\eta^{d\p}(\lambda^{d\pp}\circ g_*)\circ \eta^d(\lambda^{d\p} \circ f_*)\iso \eta^d(\lambda^{d\pp}\circ(g\circ f)_*)=F\p(g\circ f). 
\end{align*}
By precomposing both sides with $\lambda^d$, we have a chain of isomorphisms
\begin{align*}
F\p g\circ F\p f \circ \lambda^d & = \eta^{d\p}(\lambda^{d\pp}\circ g_*)\circ \eta^d(\lambda^{d\p} \circ f_*)\circ \lambda^d \iso \eta^{d\p}(\lambda^{d\pp}\circ g_*)\circ \lambda^{d\p} \circ f_* \\
& \iso \lambda^{d\pp}\circ g_* \circ f_* \iso \eta^d(\lambda^{d\pp}\circ(g\circ f)_*) \circ \lambda^d = F\p(g\circ f) \circ \lambda_d.
\end{align*}
Since precomposing with $\lambda^d$ is an equivalence, we have our desired isomorphism. It is a relatively straightforward calculation to check that these isomorphism satisfy compatibility with associators and unitors and therefore define a pseudofunctor.

Finally we need to show that $F\p$ actually extends $F$, i.e., that we have an equivalence $F\p\iota \equiv F$. This equivalence exists because of the Yoneda lemma as we will elaborate briefly.

For an object $b$ in $\bB$, $F\p \iota b$ is defined by taking the colimit of $F$ weighted by $\bD(\iota-,\iota b)$. But since $\iota$ is fully faithful, this weight is equivalent to the representable weight $\bB(-,b)$ and by the Yoneda lemma we have an equivalence
\begin{align*}
\Psh{\bJ}(\bB(-,b),\bC(F-,c))\equiv \bC(Fb,c)
\end{align*}
for all objects $c$ in $\bC$. Thus the colimit of $F$ weighted by $\bD(\iota-,\iota b)$ is equivalent to $Fb$, giving us our desired equivalence.
\end{proof}

\begin{remark}
The extension constructed here is a special case of a bicategorical left Kan extension.
\end{remark}

\urlstyle{same}
\bibliographystyle{alpha}
\addcontentsline{toc}{section}{References}
\bibliography{bibliography}
\end{document}